\documentclass{amsart}
\usepackage{geometry}                
\usepackage{amssymb}
\usepackage{amsmath}
\usepackage{enumerate}

\newtheorem{prop}{Proposition}[section]
\newtheorem{lemma}[prop]{Lemma}

\newtheorem{theorem}[prop]{Theorem}
\newtheorem{remark}[prop]{Remark}
\newtheorem{defn}[prop]{Definition}
\newtheorem{mlem}{Main Lemma}

\newtheorem*{mlem*}{Main Lemma}
\newtheorem*{lemmaRandCF}{Lemma \ref{randCF}}
\newtheorem*{lemmaLemrandommotion}{Lemma \ref{lemrandommotion}}

\newcommand{\VV}{\mathbb{V}}
\newcommand{\WW}{\mathbb{W}}
\newcommand{\PP}{\mathbb{P}}
\newcommand{\TT}{\mathbb{T}}
\newcommand{\tubes}{\mathbb{T}}
\newcommand{\SSS}{\mathbb{S}}
\newcommand{\eps}{\epsilon}
\newcommand{\QQ}{\mathbb{Q}}
\newcommand{\RR}{\mathbb{R}}
\newcommand{\EXP}{\mathbb{E}}
\newcommand{\BBB}{\mathbb{B}}

\newcommand{\exscalb}{\epsilon_{\textrm{scale}}}
\newcommand{\exfact}{\eta_{\textrm{bias}}}

\title{A streamlined proof of the Kakeya set conjecture in $\mathbb{R}^3$}
\author{Larry Guth}
\address{LG: Department of Mathematics, Massachusetts Institute of Technology, Cambridge, MA USA}

\author{Hong Wang}
\address{HW: Department of Mathematics, New York University, New York, NY USA; and IHES, Universit\'e Paris-Saclay, Bures-sur-Yvette, France.}

\author{Joshua Zahl}
\address{JZ: Chern Institute of Mathematics and LPMC, Nankai University, Tianjin 300071, China; and The University of British Columbia, Vancouver, BC, Canada.}

\begin{document}

\begin{abstract}
We present a streamlined and simplified proof of the Kakeya set conjecture in $\mathbb{R}^3$.
\end{abstract}

\maketitle

\section{Introduction to the Kakeya problem}

A \emph{Besicovitch set} is a compact set $K\subset\RR^n$ that contains a unit line segment pointing in every direction. The Kakeya set conjecture asserts that every Besicovitch set in $\RR^n$ has Minkowski and Hausdorff dimension $n$. We refer the reader to \cite{Hickman} and the references therein for an introduction to the Kakeya conjecture and its connections to problems in harmonic analysis.

The Kakeya set conjecture was proved in dimension two by Davies \cite{Dav71}, and was recently proved in dimension three by the second and third authors in a sequence of papers \cite{WZ1, WZ2, WZ}.   The purpose of this paper is to present a new,  streamlined proof of the main result from the third paper \cite{WZ}.  

The first paper, \cite{WZ1}, proves an important special case of this conjecture called the sticky Kakeya conjecture.   The sticky case was formulated by Katz and Tao, who also outlined an approach to proving it,  cf.~\cite{Tblog}.   The paper \cite{WZ1} proves the sticky case of the Kakeya conjecture,  approximately following the outline from \cite{Tblog} but dealing with significant technical issues. The second paper \cite{WZ2} obtains a mild generalization of the sticky Kakeya theorem, and uses it to prove that every Besicovitch set in $\RR^3$ has Assouad dimension 3; this is a weak form of the Kakeya set conjecture. The third paper,  \cite{WZ},  reduces the general case of the Kakeya conjecture in $\RR^3$ to the sticky case that was proved in \cite{WZ1, WZ2}.    In this paper, we present a streamlined version of this reduction. 

We begin with several definitions.  
Suppose that $\TT$ is a set of tubes in $\RR^3$ with radius $\delta$ and length 1.   We write
\[
U(\TT) = \bigcup_{T \in \TT} T.
\]

Our goal is to understand $|U(\TT)|$, the volume of $U(\TT)$.   We write $|\TT|$ for the cardinality of $\TT$ and we write $|T|$ for the volume of a tube $T \in \TT$. We will prove that if $\TT$ satisfies suitable anti-clustering hypotheses, then $|U(\TT)|$ is almost as large as $|\TT||T|$.

In order to describe these anti-clustering hypotheses, we need some language that describes how much convex sets cluster in other convex sets. Suppose that $\WW$ is a finite set of convex sets in $\RR^n$ and suppose that $K$ is another convex set.  We define

\begin{equation} \label{defWK} 
\WW[K] = \{ W \in \WW : W \subset K \},
\end{equation}
and
\begin{equation} \label{defDelta} 
\Delta(\WW, K) = \frac{ \sum_{W \in \WW[K]} |W| } {|K|}.
\end{equation}

We can think of $\Delta (\WW, K)$ as a kind of density that measures how $\WW$ concentrates in $K$.  We define a maximal density

\begin{equation} \label{defDeltamax} 
\Delta_{max}(\WW) = 
\max_{K \text{ convex}} \Delta(\WW, K).
\end{equation}

\noindent
The quantity $\Delta_{max}(\WW)$ measures the extent to which the sets from $\WW$ cluster into larger convex sets.   Quantities of this type were first introduced by Tom Wolff in \cite{Wolff1995}.

If $\WW$ is a set of convex sets, we say that $Y$ is a shading of $\WW$ if $Y(W) \subset W$ for each $W \in \WW$.  We define

\begin{equation} 
\label{UTY} U(\WW, Y) = \bigcup_{W \in \WW} Y(W).
\end{equation} 

If the shading $Y$ is clear from context, we sometimes write $U(\WW)$ for $U(\WW, Y)$.   We also measure how full the shading $Y$ is:

\begin{equation} \label{deflambda}  
\lambda(\WW, Y) = \frac{ \sum_{W \in \WW} |Y(W)| } {\sum_{W \in \WW} |W| }. 
\end{equation}

With these definitions, we can now state the main result from \cite{WZ}.

\begin{theorem}[Kakeya set conjecture in $\RR^3$] \label{thmkak}
 For every $\beta > 0$, there is some $\eta > 0$ so that the following holds for all sufficiently small $\delta>0$.

If $\TT$ is a set of $\delta$-tubes in $\RR^3$ with $\Delta_{max}(\TT) \leq \delta^{-\eta}$, and $Y$ is a shading of $\TT$ with $\lambda(\TT, Y) \geq \delta^\eta$, then

\begin{equation} 
\label{eqkakconc}  |U(\TT, Y)| \geq \delta^\beta | \TT| |T|. 
\end{equation}

\end{theorem}

\begin{remark}
The original Kakeya conjecture concerns direction-separated tubes.   One version says that  if $\TT$ is a set of $\delta^{-(n-1)}$ $\delta$-tubes in $\RR^n$ with $\delta$-separated directions,  then $|U(\TT)| \gtrapprox 1$.   It is straightforward to check that if the tubes of $\TT$ are direction separated, then $\Delta_{max}(\TT) \lesssim 1$,  and so Theorem \ref{thmkak} implies a similar estimate for direction separated tubes in $\RR^3$.  The volume bound \eqref{eqkakconc} thus implies that every Besicovitch set in $\RR^3$ has Minkowski and Hausdorff dimension 3.
\end{remark}

The goal of this paper is to give a detailed proof of Theorem \ref{thmkak}.   A key ingredient is the sticky Kakeya theorem from \cite{WZ1, WZ2}.   We give a self-contained proof that Theorem \ref{thmkak} follows from the sticky Kakeya theorem.   Our proof of Theorem \ref{thmkak} has the same main ideas as the one in \cite{WZ}, but it has been reorganized to make it more concise.  A couple parts of the argument in \cite{WZ}, such as polynomial partitioning, are no longer needed.  

The introduction of \cite{WZ} gives an overview of the main ideas of the proof.  The expository paper \cite{G} gives an overview of the Kakeya problem, the sticky case, and the reduction to the sticky case.  The expository paper \cite{G2} gives a detailed outline of the reduction from the general case to the sticky case.  The current paper closely follows that outline.

\subsection{Formal Verification and Thanks}
The arguments in this paper have been formally verified in Lean by Yichuan Cao, Weikun He, Yihe Huang, Yunfei Li, Xiao Ma, and Jiaqi Wang \cite{CHHLMW}. This paper makes crucial use of the sticky Kakeya theorem; this result has been formally verified in Lean by Mingfeng Chen, Shaoming Guo, Yixuan Pang, Minxing Shen, Zhihong Wang, Zheng Yuan, Huanyu Zheng, and Thomas Zhu \cite{CGPSWYZ}.  These efforts have since been combined to give a complete formalization of Theorem \ref{thmkak} in Lean. Both groups made extensive use of AI tools for autoformalization.

While formalizing these results, the mathematicians from these two groups shared comments and corrections to an earlier version of this manuscript; we are grateful to both groups for their efforts in formalizing this result, and also for their comments and corrections. We are also grateful to Alex Cohen for sharing comments and corrections to an earlier version of this manuscript.

\section{Notation and Definitions}
\subsection{Convex sets and shadings}
We will use blackboard bold letters to denote finite families of convex subsets of $\RR^n$, and their non-bold equivalents to denote members of these sets, i.e. $\VV$ will be a family of convex sets, and $V_1,V_2$ or $V,V'$ will denote elements of $\VV$. Similarly for $\WW$ and $W$, and for $\TT$ and $T$. We will often use the notation $\sum_{\VV}f(V)$ in place of $\sum_{V\in\VV}f(V)$, etc. If $\VV$ is a family of convex sets, each of which have comparable volume, then we will write $|V|$ to denote the volume of a convex set from $\VV$.

If $V$ is a convex set, the ``dimensions'' of $V$ will be the lengths of the axes of its outer John ellipsoid, listed in increasing order. We will denote this by $v_1\times v_2\times\ldots\times v_n$. A \emph{$\delta$-tube} is the $\delta$ neighbourhood of a unit line segment in $\RR^n$. A \emph{box} is a translated and rotated image of the set $[0,v_1]\times\ldots\times[0,v_n]$. In $\RR^3$, a \emph{slab} is a box of dimensions $\theta\times 1\times 1$, for some $0<\theta\leq 1$, and a \emph{plank} is a box of dimensions $a\times b\times 1$ for some $0<a\leq b\leq 1$. We say two convex sets $V,V'$ are \emph{essentially distinct} if $|V\cap V'| < \frac{1}{2}\min(|V|, |V'|)$. We say a set $\VV$ of convex sets is essentially distinct if every pair of elements is essentially distinct.

A \emph{shading} of a set $V$ is a measurable set $Y(V)\subset V$. We will use $(\VV,Y)$ to denote a family of convex sets and their associated shadings $\{Y(V)\colon V\in\VV\}$. For brevity we will sometimes say ``let $(\VV,Y)$ be a family of convex subsets of $\RR^n$,'' or ``let $(\TT,Y)$ be a set of $\delta$-tubes.'' 

We define the multiplicity of a set of convex sets $\WW$ with a shading $Y$ as
\begin{equation} \label{defmu} 
\mu(\WW, Y) = \frac{ \sum_{W \in \WW} |Y(W)| }{| U(\WW, Y)|}. 
\end{equation}
When the shading $Y$ is clear from the context, we write $\mu(\WW)$ instead. 

Information about $\mu(\WW)$ is essentially equivalent to information about $|U(\WW)|$, but it can be more intuitive to think about one than the other.  For instance,  in Theorem \ref{thmkak},  it would be equivalent to replace the conclusion by $\mu(\TT) \lesssim \delta^{-\beta}$.

For $x\in\RR^n$, we define $\VV_Y(x) = \{ V \in \VV: x \in Y(V) \}$ and  $\mu(\VV, Y)(x) = |\VV_Y(x)|. $ We say that $(\VV, Y)$ has constant multiplicity if $\mu(\VV, Y)(x)$ is roughly the same for all $x \in U(\VV, Y)$. We say that $(\VV', Y')$ is a refinement of $(\VV, Y)$ if $\VV'\subset\VV$ and $Y'(V) \subset Y(V)$ for each $V \in \VV'$. For $c>0$, we say that $(\VV', Y')$ is a $c$-refinement of $(\VV,Y)$ if $\sum_{\VV'}|Y'(V)|\geq c \sum_{\VV}|Y(V)|$. If $Y$ is a shading on $\VV$ and $\VV'$ is a subset of $\VV$, we will often abuse notation and continue to use $Y$ to refer to the restriction of the shading to $\VV'$.

\subsection{Pigeonholing and $\lessapprox$ notation}\label{defnLessapprox}
Many steps in our proof involve dyadic pigeonholing of the following types:
\begin{enumerate}[(i)]
\item\label{pigeonholeConvexItem} If $\VV$ is a family of convex subsets of $\RR^n$, each of which has a John ellipsoid with axes of lengths between $\delta$ and $\delta^{-1}$, then we wish to pigeonhole to find numbers $v_1\leq \ldots \leq v_n$ and a subset $\VV'$ consisting of convex sets whose John ellipsoids have axes of lengths roughly $v_1\times \ldots\times v_n$. We can always find a subset with $\sum_{\VV'}|V|\gtrsim(\log 1/\delta)^{-n}\sum_{\VV}|V|$ with this property.
\item\label{pigeonholeShadingItem} If $\VV$ is a family of convex subsets of $\RR^n$ with a shading $Y$, then we wish to find a subset $X\subset\RR^n$ and a number $\nu$ such that $\sum_{\VV}\chi_{Y(V)}(x) \sim \nu$ on $X$. We can always find such a set so that $\nu|X| \gtrsim (\log|\VV|)^{-1}\sum_{\VV}|Y(V)|$.
\item\label{pigeonholeMeasureItem} If $(\VV,Y)$ is a family of convex subsets of $\RR^n$ and their associated shading, then we wish to find a subset $\VV'$ consisting of convex sets whose shadings $\{Y(V)\colon V\in\VV\}$ have roughly the same measure. If $a=\min\{|Y(V)|\colon V\in\VV,\ Y(V)\neq\emptyset\}$ and $A = \max_{\VV}|Y(V)|$, then we can always find such a set so that $(\VV',Y)$ is a $\gtrsim(1+\log(A/a))^{-1}$-refinement of $(\VV,Y)$.
\item\label{pigeonholeMultipleItem} Sometimes we wish to perform the pigeonholing from several of the above items simultaneously. For example if $(\VV, Y)$ is a family of convex subsets of $\RR^n$, each of which has a John ellipsoid with axes of lengths between $\delta$ and $\delta^{-1}$, then there exists a set $\VV'$ as in Item (\ref{pigeonholeConvexItem}), and a subset $X\subset\RR^n$ and a number $\nu$ as in Item (\ref{pigeonholeShadingItem}), so that $\nu|X|\gtrsim (\log|\VV|)^{-1}(\log 1/\delta)^{-n}\sum_{\VV}|Y(V)|$.
\end{enumerate}
Quantities of the form $\log(1/\delta)$ or $\log|\VV|$ are harmless distractions, and in practice any shading $Y(V)$ with small measure (i.e.~$|Y(V)|\leq\delta^n$) can be replaced by $Y(V)=\emptyset$ without affecting our arguments. To streamline our arguments, if $X$ and $Y$ are functions of $\rho$, then we introduce the notation  ``$ X \lessapprox Y$'' to mean that for all $\eps>0$, there is a constant $C=C(\eps)$ so that $X\leq C_\eps\rho^{-\eps}Y$. The quantity $\rho$ will often be implicit---here is how $\lessapprox$ notation is applied to the items described above.

\begin{enumerate}[(i$'$)]
	\item\label{refineConvexItem} For inequalities involving a convex set, or a collection of convex sets $\VV$, we will implicitly define $\rho$ to be $\min_{\VV}\delta(V)$, where $\delta(V)$ is the ratio of the shortest to the longest axes of the John ellipsoid of $V$ (up to harmless constants, we could also take $\delta(V)$ to be the ratio of the diameter of the largest ball contained inside $V$ to the smallest ball that contains $V$). Thus for Item (\ref{pigeonholeConvexItem}) above, we could write $\sum_{\VV'}|V|\gtrapprox\sum_{\VV}|V|$.

	\item\label{refineCardItem} For inequalities involving the cardinality of a set $\VV$, we will implicitly define $\rho$ to be $|\VV|^{-1}$. Thus for Item (\ref{pigeonholeShadingItem}) above, we could write $\nu|X|\gtrapprox\sum_{\VV}|Y(V)|$.

	\item\label{refineFamilyItem} For inequalities involving a family $(\VV,Y)$ of convex sets and their associated shading, we will implicitly define $\rho$ to be $a/A$, as described in Item (\ref{pigeonholeMeasureItem}). Thus for Item (\ref{pigeonholeMeasureItem}) above, $(\VV',Y)$ is a $\gtrapprox 1$ refinement of $(\VV, Y)$. 

	\item Items (\ref{refineConvexItem}$'$), (\ref{refineCardItem}$'$), and (\ref{refineFamilyItem}$'$) combine as described in Item (\ref{pigeonholeMultipleItem}) above.

\end{enumerate}

\subsection{Uniform sets of tubes}

It will sometimes be useful to consider sets $\TT$ that are uniform in the following sense.   

\begin{defn}\label{uniformSetOfTubes}
Let $\delta>0$, let $M = \lceil \log\log 1/\delta\rceil$, and let $\TT$ be a set of $\delta$-tubes in $B_1 \subset \RR^n$. We say that $\TT$ is \emph{uniform} if for each scale of the form $\rho = \delta^{k/M}$, $k=0,\ldots,M$, there is a set of $\rho$-tubes $\TT_\rho$ so that

\begin{itemize}

\item[(i)]  $\TT = \bigcup_{\TT_{\rho}} \TT[T_{\rho}]$.

\item[(ii)] The tubes $\TT_{\rho}$ are essentially distinct.  Therefore each $T \in \TT$ lies in $T_{\rho}$ for $\sim 1$ choice of $T_{\rho} \in \TT_{\rho}$.

\item[(iii)] $| \TT[T_{\rho}] |$ is constant up to a factor $\sim 1$ as $T_{\rho}$ varies in $\TT_{\rho}$.

\end{itemize}
We call the numbers $N_k = | \TT[T_{\rho_k}] |$ the \emph{branching numbers} of $\TT$.
\end{defn}

\begin{defn}\label{uniformSetOfTubesAndShading}
Let $\delta>0$, let $M = \lceil \log\log 1/\delta\rceil$, and let $(\TT,Y)$ be a set of $\delta$-tubes in $B_1 \subset \RR^n$ and their associated shading. We say that $(\TT,Y)$ is \emph{uniform} if the following holds:
\begin{itemize}
\item  $\TT$ is uniform.
\item For each point $x\in U(\TT,Y)$, the sets $\{T\in\TT\colon x\in Y(T)\}$ are uniform, and have roughly the same branching numbers.
\end{itemize}
\end{defn}

The reason for setting $M=\log\log(1/\delta)$ in our above definitions is as follows. First, since $(\log|\TT|)^{\log\log(1/\delta)}\lessapprox 1$, by pigeonholing, every set $\TT$ of $\delta$-tubes in $B_1\subset\RR^n$ has a subset $\TT'$ of cardinality $|\TT'|\gtrapprox |\TT|$ that is uniform, and every pair $(\TT,Y)$ has a $\approx 1$-refinement that is uniform. 

Second, if $\TT$ is uniform, then for \emph{every} scale $\rho\in[\delta,1]$, there is a set of essentially distinct $\rho$-tubes $\tubes_\rho$ that satisfies items (i) and (ii) from Definition \ref{uniformSetOfTubes}, and satisfies the following slightly weaker analogue of (iii): there is a number $N_\rho$ so that $| \TT[T_{\rho}] |\approx N_\rho$ for all $T_\rho\in\tubes_\rho$. Similarly for Definition \ref{uniformSetOfTubesAndShading}; the second bullet point is weakened to: the branching numbers $N_{k}$ are roughly the same up to a $\approx 1$ factor.


\section{Partial Kakeya estimates}

In this section,  we consider Kakeya-type estimates for sets of tubes obeying two important non-clustering conditions.   These non-clustering conditions are called Katz-Tao conditions and Frostman conditions.   They are both descended from the non-concentration conditions introduced by Wolff in \cite{Wolff1995}.  

\begin{defn}
A family of convex sets $\VV$ is called $C$-convex-Katz-Tao (or $C$-Katz-Tao for short) if $\Delta_{max}(\VV) \le C$.  Informally, we say $\VV$ is Katz-Tao if $\Delta_{max}(\VV) \lessapprox 1$.
\end{defn}

\begin{defn}
If $\VV$ is a set of convex sets in $K$, we define
\begin{equation} \label{defcf} 
C_F(\VV, K) = \frac{ \sup_{K' \subset K} \Delta(\VV, K') }{\Delta(\VV,K)}, 
\end{equation}
where the supremum is taken over all convex subsets of $K$. We say that $\VV$ is $C$-convex-Frostman (or $C$-Frostman for short) in $K$ if $C_F(\VV, K) \le  C$. Informally, we say that $\VV$ is Frostman in $K$ if  $C_F(\VV, K) \lessapprox 1$. If $K$ is the unit ball, then we abbreviate this as $C_F(\VV)$, and we say that $\VV$ is $C$-Frostman.
\end{defn}

\begin{remark}\label{inheritedDownwardsUpwardsRemark}$\phantom{1}$\\
(A) The property of being Frostman is \emph{inherited upwards}: Suppose $\VV = \sqcup_{\WW}\VV_W$ is $C$-Frostman inside $K$, where each $W\in\WW$ is contained in $K$ and $V \subset W$ for each $V \in \VV_W$.  Then there is a subset $\WW'\subset \WW$ that is $\lessapprox C$-Frostman in $K$. If in addition the quantity $\sum_{\VV_W}|V|/|W|$ is approximately the same for each $W\in\WW$, then we may take $\WW'=\WW$, and in fact $\WW$ is $O(C)$-Frostman in $K$.  \newline
(B) The property of being Katz-Tao is \emph{inherited downwards}: If $\VV$ is $C$-Katz-Tao, then every subset $\VV'\subset\VV$ is also $C$-Katz-Tao.\\
\end{remark}

\begin{proof} We prove (A).   Recall that $\Delta(\VV_W, W) = \frac{ \sum_{V \in \VV_W} |V| }{|W|}$.   Define

$$\WW_D = \{ W \in \WW: \Delta(\VV_W, W) \sim D \}. $$

 There are $\lessapprox 1$ values of $D$.   Choose $D$ so that $\sum_{W \in \WW_D} \sum_{V \in \VV_W} |V| \approx \sum_{V \in \VV} |V|$. 

Now consider a convex set $U \subset K$.   Since $\VV$ is $C$-Frostman in $K$,  we have

\begin{equation} \label{VCfrost} \frac{ \sum_{V \in \VV[U]} |V| }{|U|} \le C  \frac{ \sum_{V \in \VV[K]} |V| }{|K|}. \end{equation}

Now 

$$ \sum_{V \in \VV[U]} |V| \ge \sum_{W \in \WW_D[U]} \sum_{V \in \VV_W} |V| \sim \sum_{W \in \WW_D[U]} D |W|.  $$

Plugging this equation into the last, we see that

$$  \frac{ \sum_{W \in \WW_D[U]} D |W| }{|U|} \lesssim C  \frac{ \sum_{V \in \VV[K]} |V| }{|K|}. $$

On the other hand,  since $\VV = \VV[K]$,  $\sum_{V \in \VV} |V| \approx \sum_{W \in \WW_D} \sum_{V \in \VV_W} |V| \sim \sum_{W \in \WW_D} D |W|. $   Plugging this into the right-hand of the last equation and cancelling a factor of $D$ leaves

$$   \frac{ \sum_{W \in \WW_D[U]}  |W| }{|U|} \lessapprox C   \frac{ \sum_{W \in \WW_D[K]}  |W| }{|K|}.$$

In other words, $C_F(\WW_D, K) \lessapprox C = C_F(\VV, K)$.   In the special case that $\WW = \WW_D$,  the argument above shows that $C_F(\WW, K) \lesssim C = C_F(\VV, K)$.  

To prove (B),  just note that

\[
\Delta_{max}(\VV') =  \max_K \frac{ \sum_{V \in \VV'[K]} |V| }{|K|} \le  \max_K \frac{ \sum_{V \in \VV[K]} |V| }{|K|} = \Delta_{max}(\VV).\qedhere
\]

\end{proof}

\begin{defn} \label{def: KKT}  
We write $K_{KT}(\beta)$ for the following statement.  For any $\epsilon>0$, there exists $\eta=\eta(\epsilon, \beta)>0$ so that the following holds for any $0<\delta\leq \delta(\epsilon, \beta)$.   Let $\TT$ be a set of $\delta$-tubes in $B_1 \subset \RR^3$ with the following properties.

\begin{itemize}

\item  $ \Delta_{max}(\TT) \leq \delta^{-\eta}$.

\item $Y$ is a shading of $\TT$ with $\lambda(\TT,Y) \geq \delta^{\eta}$.   

\end{itemize}
Then 
\begin{equation}\label{muBoundDefnKKT}
\mu( \TT, Y) \leq  \delta^{-\epsilon} \big( |T|^{-1}| \TT | \big)^{\beta/2}.
\end{equation}

\end{defn}

Note that $K_{KT}(1)$ is trivial. It is not hard to prove $K_{KT}(\beta)$ for some $\beta$ slightly less than 1.   Our goal is to prove $K_{KT}(0)$. 
One important case is when $| \TT | \sim \delta^{-2}$.   In this case,  $K_{KT}(\beta)$ gives $\mu(\TT) \lesssim  \delta^{-2 \beta-\epsilon}$,  which is equivalent to $|U(\TT)| \gtrsim  \delta^{2 \beta +\epsilon}$.  

\begin{defn} \label{def: KF}

We write $K_{F}(\beta)$ for the following statement.  For any $\epsilon>0$, there exists $\eta=\eta(\epsilon, \beta)>0$ so that the following holds for any $0<\delta\leq \delta(\epsilon, \beta)$.   Let $\TT$ be a set of essentially distinct $\delta$-tubes in $B_1 \subset \RR^3$ with the following properties.

\begin{itemize}

\item  $C_F(\TT, B_1) \leq \delta^{-\eta}$. 

\item $Y$ is a shading of $\TT$ with $\lambda(\TT,Y) \geq  \delta^{\eta} $.

\end{itemize}
Then 
\begin{equation}\label{muBoundDefnKF}
\mu( \TT) \leq  \delta^{-\epsilon} \delta^{-2\beta} \left( | \TT |\ |T| \right)^{1 - \beta/2}.
\end{equation}
Or equivalently,
\[
|U(\TT)| \geq \delta^{\epsilon}\delta^{2 \beta} \left( | \TT|\ |T|\right)^{\beta/2}.
\]
\end{defn}

Let us parse the algebra in this bound.   Since $C_F(\TT, B_1) \leq \delta^{-\eta} $,  we have $| \TT | \geq  \delta^{-2+\eta}$.    One important case is when $| \TT | \sim \delta^{-2}$.   In this case,  the bound $C_F(\TT, B_1) \lesssim  \delta^{-\eta} $ is equivalent to the bound $\Delta_{max}(\TT) \lesssim \delta^{-\eta} $.   So $\TT$ is both Frostman and Katz-Tao.   In this case,   $K_{KT}(\beta)$ gives $|U(\TT)| \gtrsim  \delta^{2 \beta+\epsilon}$,  and $K_F(\beta)$ gives the same bound $|U(\TT)| \gtrsim \delta^{2 \beta +\epsilon}$.    When $| \TT | \gg \delta^{-2}$,  then $K_F(\beta)$ gives a stronger volume bound,  because the last factor $ \left( | \TT|\ |T| \right)^{\beta/2}$ is $\gg 1$.   The exact exponent $\beta/2$ in this final factor is not crucial in our analysis,  and the argument would work equally well with any positive exponent $c\beta,\ 0<c<1$. There is one step (see Remark \ref{criticalUseExponentBetaOver2}) where it is important that $c<1$.   

\begin{remark}\label{multBoundsDeltaVsRho}
If $K_{KT}(\beta)$ holds,  then we get for free the following slightly more general version of $K_{KT}(\beta)$:  For any $\epsilon>0$, there exists $\eta=\eta(\epsilon, \beta)>0$ so that the following holds for any $0<\delta\leq \delta(\epsilon, \beta)$.   Suppose $\TT$ is a set of $\delta$-tubes in $B_1 \subset \RR^3$,  $\tau \le \delta$, and

\begin{itemize}

\item  $ \Delta_{max}(\TT) \leq \tau^{-\eta}$.

\item $Y$ is a shading of $\TT$ with $\lambda(\TT,Y) \geq \tau^{\eta}$.   

\end{itemize}
Then 
\begin{equation*}
\mu( \TT, Y) \leq  \tau^{-\epsilon} \big( |T|^{-1}| \TT | \big)^{\beta/2}.
\end{equation*}

Similarly, $K_F(\beta)$ implies a slightly more general version of $K_F(\beta)$ for essentially distinct tubes.  

These slightly more general versions will be useful when we replace a set of $\delta$-tubes $(\TT,Y)$ with a thickening $(\TT_\rho, Y_\rho)$, for which we only know that $\lambda(\TT_\rho, Y_\rho)\geq\delta^{\eta}$.
\end{remark}

\begin{proof} Suppose $\beta, \epsilon > 0$ and suppose $K_{KT}(\beta)$ holds.   Let $\eta_1(\beta, \epsilon) > 0$ and $\delta(\beta, \epsilon) > 0$ be the constants given by the definition $K_{KT}(\beta)$,  and define $\eta(\beta, \epsilon) = \frac{1}{100} \epsilon \eta_1(\beta, \epsilon)$.  

If $\tau \le \delta^{100 \epsilon^{-1} }$,  then $\tau^{- \epsilon} \ge \delta^{-100}$ and the conclusion is trivial.  Otherwise,
$\Delta_{max}(\TT) \le \tau^{-\eta} \le \delta^{- \eta_1}$ and $\lambda(\TT, Y) \ge \tau^\eta \ge \delta^{\eta_1}$,  and so 
we get $\mu(\TT, Y) \le \delta^{ - \epsilon } \big( |T|^{-1}| \TT | \big)^{\beta/2} \le \tau^{- \epsilon} \big( |T|^{-1}| \TT | \big)^{\beta/2}$.

The argument using $K_F(\beta)$ is similar.
\end{proof}

The two main lemmas in the proof are

\begin{mlem} \label{lemmain1} $K_{KT}(\beta)$ implies $K_F(\beta)$.  
\end{mlem}

\begin{mlem} \label{lemmain2}  If $\beta > 0$ and $K_{KT}(\beta)$ and $K_{F}(\beta)$ hold,  then $K_{KT}(\beta - \nu)$ holds for $\nu = \nu(\beta) > 0$.   Moreover $\nu(\beta)$ is monotone in $\beta$.
\end{mlem}

Since the bound $K_{KT}(1)$ is trivial,  the two main lemmas imply that $K_{KT}(\beta)$ and $K_F(\beta)$ hold for all $\beta > 0$,  which implies the Kakeya conjecture,  Theorem \ref{thmkak}.

\subsection{Multiplicity bounds for larger values of $\Delta_{\max}(\TT)$ or $C_F(\TT)$}
The next two results say that $K_{KT}(\beta)$ and $K_F(\beta)$ imply multiplicity bounds for sets of tubes with any value of $\Delta_{max}(\TT)$ or $C_F(\TT)$.  

\begin{lemma} \label{genKKT} Suppose that $K_{KT}(\beta)$ holds.  Then for any $\epsilon>0$, there exists $\eta=\eta(\epsilon, \beta)>0$ such that the following holds for any $0<\delta< \delta(\epsilon, \beta)$.   Suppose that $(\TT,Y)$ is a set of $\delta$-tubes in $B_1$ with a shading with $\lambda(\TT) \geq \delta^{\eta}$.   Then

\[
\mu( \TT,Y) \leq \delta^{-\epsilon} \Delta_{max}(\TT)^{1 - \beta/2}  \big( |T|^{-1}| \TT | \big)^{\beta/2}.
\]

\end{lemma}

\begin{proof} Choose a random subset $\TT' \subset \TT$ with cardinality $\frac{1}{\Delta_{max}(\TT)} | \TT|$.   It is straightforward to check that $\Delta_{max}(\TT') \lessapprox 1$,  and so by $K_{KT}(\beta)$ with $\epsilon/2$ in place of $\epsilon$,
$
\mu(\TT') \leq \delta^{-\epsilon/2}\big( |T|^{-1}|\TT'|\big)^{\beta/2}.
$
But then
\[
\mu(\TT) \lessapprox \Delta_{max}(\TT) \mu(\TT') \leq \delta^{-\epsilon}  \Delta_{max}(\TT)^{1-\beta/2} \big( |T|^{-1}| \TT| \big)^{\beta/2}.\qedhere
\]

\end{proof}

\begin{lemma} \label{randCF} Suppose that $\TT$ is a set of (essentially distinct) $\delta$-tubes in $B_1 \subset \RR^n$.   Let $R_1, ..., R_J$ be random rigid motions of size 1,  with $J = C_F(\TT)$,  and let $\TT' = \bigcup_{j=1}^J R_j(\TT)$.   Then with high probability the tubes of $\TT'$ are essentially distinct up to multiplicity $\lessapprox 1$ and also $C_F(\TT') \lessapprox 1$.    After refining $\TT'$ to make the tubes essentially distinct,  we have $|\TT'| \approx C_F(\TT) |\TT|$ and $C_F(\TT') \lessapprox 1$.
\end{lemma}

We will prove Lemma \ref{randCF} in Appendix \ref{appproblemmas},  where we collect a couple of related probability lemmas.

\begin{lemma}\label{genKF} Suppose that $K_F(\beta)$ holds. Then for any $\epsilon>0$, there exists $\eta=\eta(\epsilon, \beta)>0$ such that the following holds for any $0<\delta<\delta(\epsilon, \beta)$. Suppose that $(\TT,Y)$ is a set of essentially distinct $\delta$-tubes in $B_1$ with a shading with $\lambda(\TT) \geq \delta^{\eta}$. 
	
	Then 
	\[
	  \mu(\TT,Y) \leq \delta^{-\epsilon}  C_F(\TT)^{1-\beta/2}\delta^{-2\beta} (\delta^2|\TT|)^{1-\beta/2}.
	\]
	
\end{lemma}
\begin{proof}
	Let $\TT'$ be the union of $C_F(\TT)$ random rigid motions of $\TT$.  By Lemma \ref{randCF},  with high probability $C_F(\TT')\lessapprox 1$ and $|\TT'|\approx C_F(\TT) |\TT|$. After a harmless refinement, the tubes in $\TT'$ are essentially distinct. Applying $K_F(\beta)$ with $\epsilon/2$ in place of $\epsilon$, we have 
	\[
	\mu(\TT)\leq \mu(\TT') \lessapprox \delta^{-\epsilon/2} \delta^{-2\beta}(\delta^2|\TT'|)^{1-\beta/2} \leq \delta^{-\epsilon}  C_F(\TT)^{1-\beta/2}\delta^{-2\beta} (\delta^2|\TT|)^{1-\beta/2}.\qedhere
	\]
\end{proof}
We note that the analogue of Remark \ref{multBoundsDeltaVsRho} holds for Lemmas \ref{genKKT} and \ref{genKF}

\section{Organizing convex sets}

If $\VV$ is a family of convex sets with $\Delta_{max}(\VV) \gg 1$,  it turns out to be useful to consider the convex sets $W$ so that $\Delta(\VV,  W)$ is near maximal. These convex sets $W$ can be used to organize $\VV$.  

\begin{lemma}[Maximal density factoring lemma] \label{lemmafactmax}  Let $\VV$ be a finite family of convex subsets of $\RR^n$. Then there is a set $\VV'\subset\VV$ with $\sum_{\VV'}|V|\gtrapprox \sum_{\VV}|V|$; a family $\WW$ of convex sets in $\RR^n$, each of approximately the same dimensions; and a partition
\begin{equation}\label{VVPartition}
\VV'=\bigsqcup_{W\in\WW}\VV_W,
\end{equation}
where each set in $\VV_W$ is contained in $W$. This partition has the following properties.

\begin{enumerate}[(i)]
\item\label{lemmaFactMax_Item_DeltaMaxWW} $\WW$ is Katz-Tao, i.e.~$\Delta_{max}(\WW) \sim 1$.
\item\label{lemmaFactMax_Item_DeltaVVWAchievesMax} $\VV_W$ is Frostman in $W$, i.e.~$C_F(\VV_W,W)\lesssim 1$ for each $W \in \WW$. Furthermore $\Delta(\VV_W, W) \sim \Delta_{max}(\VV')$. 
\end{enumerate}
\end{lemma}

\begin{proof} After replacing $\VV$ by a subset for which $\sum_{\VV}|V|$ is preserved up to a $\approx 1$ factor, we may suppose that each convex set in $\VV$ has approximately the same dimensions. We may also suppose that each set in $\VV$ is compact (if not, then we replace each set by its closure, and then undo this step at the end of the proof). We choose sets $W_j$ one at a time by a greedy algorithm.
Set $\VV_0=\VV$. Choose $W_0$ to maximize $\Delta(\VV, W_0)$ (a maximizer exists, since $\VV$ is finite and the sets in $\VV$ are compact).   Then set $\VV_1 = \VV \setminus \VV[W_0]$.    Now choose $W_1$ to maximize $\Delta(\VV_1, W_1)$.   Continue in this way until $\VV_s$ is empty.   Then stop.  For each $j$,  we set $\VV_{W_j} = \VV_{j}[W_j]$; the sets $\VV_{W_j}$ are disjoint.

Next we pigeonhole $\Delta(\VV_{j},  W_j)$. For each dyadic $\nu$ with $1\leq\nu\lesssim|\VV|$, we set $J(\nu) = \{ j: \Delta(\VV_{j}, W_j) \sim \nu\}$.   We set $\VV_{\nu} = \bigcup_{j \in J(\nu)} \VV_{W_j}$.   We note that $\VV = \bigsqcup_{\nu} \VV_\nu$.   We choose $\nu$ so that $|\VV_\nu| \gtrapprox |\VV|$. Choose a subset $\WW$ of $ \{ W_j \}_{j \in J(\nu)}$ of approximately the same dimensions, so that $\VV' = \bigsqcup_{\WW} \VV_{W}$ satisfies $|\VV'|\gtrapprox |\VV_{\nu}|$, and replace $J(\nu)$ with the set of indices $j$ such that $W_j\in \WW$. Observe that we have the desired bound on the size of $\VV'$.

We have $ \sum_{\VV_{W_j}}|V|  \sim \nu |W_j|$ for each $j\in J(\nu)$.  Since $V\in \VV$ and each $W\in\WW$ has approximately the same dimensions, $|\VV_{W_j}|$ is comparable for each $j\in J(\nu)$.

Note that $\Delta(\VV_{j},  W_j) = \Delta_{max}(\VV_{j})$ is non-increasing.  Set $j_1$ to be the first number in $J(\nu)$.    Since $\VV' \subset \VV_{j_1}$ and $\VV_{W_j} \subset \VV'$, we have

\begin{equation}\label{DeltaMaxVV_vs_DeltaVW}
\Delta_{max}(\VV') \le \Delta_{max}(\VV_{j_1}) = \Delta(\VV_{W_{j_1}}, W_{j_1}) \le \Delta_{max}(\VV').
\end{equation}
Therefore,  all the inequalities above are equalities, and thus Item (\ref{lemmaFactMax_Item_DeltaVVWAchievesMax}) is satisfied for $W_{j_1}$. But since each set $\VV_{W_j}$ has approximately the same cardinality, and each set $W_j$ has approximately the same volume, we have that $\Delta(\VV_{W}, W)$ is approximately the same for all $W\in\WW$, and thus Item (\ref{lemmaFactMax_Item_DeltaVVWAchievesMax}) is satisfied for all $W\in\WW$.

It remains to verify Item (\ref{lemmaFactMax_Item_DeltaMaxWW}). Let $U\subset\RR^n$ be a convex set. Then
\[
|\VV'[U]| \geq \sum_{W\in\WW[U]}|\VV_W| \gtrsim \sum_{W\in\WW[U]}\Delta(\VV_W, W)\frac{|W|}{|V|} \gtrsim  |\WW[U]|\Delta_{max}(\VV')\frac{|W|}{|V|}.
\]
Re-arranging, we obtain
\[
|\WW[U]|\lesssim \frac{|\VV'[U]|\,|V|}{\Delta_{max}(\VV') |W|} \leq \frac{\Delta_{max}(\VV') |U|}{\Delta_{max}(\VV') |W|}=\frac{|U|}{|W|}.\qedhere
\]
\end{proof}
\begin{remark}\label{makeLessapproxExplicit}
An examination of the proof shows that if every convex set in $\VV$ has a John ellipsoid with axes of lengths between $\delta$ and $\delta^{-1}$, then the implicit function $c_\eps\delta^{\eps}$ in the bound $|\VV'|\gtrapprox|\VV|$ may be replaced by $\big(\log(\delta^{-1})\log|\VV|\big)^{-100}$. Also, if the sets $\VV$ are tubes or boxes, then the sets $\WW$ may be taken to be boxes.
\end{remark}

\begin{remark}\label{maximalDensityFactoringPreservesFrostmanRefinement}
In general, if $\VV'$ is a collection of convex sets and $\VV''\subset\VV'$ with $|\VV''|=c|\VV'|$, then $C_F(\VV'')\leq c^{-1}C_F(\VV')$, but there is no matching lower bound---$C_F(\VV'')$ might be substantially smaller than $C_F(\VV')$. However, if $\VV'$ satisfies the conclusions of Lemma \ref{lemmafactmax}, then we \emph{do} have a matching lower bound, i.e.~$c C_F(\VV')\lessapprox C_F(\VV'')\leq c^{-1}C_F(\VV') $.
\end{remark}

\vskip10pt

\begin{defn}\label{defnFactors}
Let $\VV$ and $\WW$ be families of convex subsets of $\RR^n$, and suppose there is a partition $\VV = \bigcup_{\WW}\VV_W$. We say that $\WW$ \emph{factors} $\VV$ if the following is true:
\begin{itemize}
	\item $\WW$ is Katz-Tao, i.e.~$\Delta_{max}(\WW) \sim 1$.
	\item $\VV_W$ is Frostman in $W$, i.e.~$C_F(\VV_W,W)\lesssim 1$ for each $W \in \WW$. Furthermore $\Delta(\VV_W, W) \sim \Delta_{max}(\VV)$.
	\item Each $W\in\WW$ has approximately the same dimensions, and each set in $\VV_W$ is contained in $W$.
\end{itemize}
\end{defn}
In the above definition, we choose the implicit constants to match those coming from Lemma \ref{lemmafactmax}. Thus Lemma \ref{lemmafactmax} says that $\WW$ factors $\VV'$. We call the set $\WW$ coming from Lemma \ref{lemmafactmax} the maximal density factoring of $\VV$, and we call Lemma \ref{lemmafactmax} the maximal density factoring lemma. 

Let us now discuss what we gain from the maximal density factoring lemma.   We began with an arbitrary set of convex sets $\VV$.   The set $\VV$ does not obey any ``non-clustering condition'' such as $\Delta_{max}(\VV) \lessapprox 1$.   Using the factoring lemma,  we can often reduce the problem of understanding $\VV$ to the problem of understanding the Frostman sets $\VV_W$ for $W\in\WW$, and the problem of understanding the Katz-Tao set $\WW$.  

Our goal is to prove a bound on $\mu(\TT)$ when $\TT$ is Katz-Tao.  Our proof is by induction, and along the way we will see other sets of tubes and more generally other sets of convex sets.  At first sight, these other sets are not Katz-Tao or Frostman, but by using the maximal density factoring we can reduce matters to Katz-Tao sets and Frostman sets, and we can study these by induction.


\section{Multiplicities, refinements, and factoring}\label{multAndFactoringSection}
As discussed in the previous section, if $\VV$ is a family of convex sets then we will bound $\mu(\VV)$ by applying the maximal density factoring lemma and then bounding $\mu(\VV_W)$ and $\mu(\WW)$.  In this section, we explain exactly how $\mu(\VV)$ is related to $\mu(\VV_W)$ and $\mu(\WW)$.  

Intuitively,  it makes sense to have a bound of the form $\mu(\VV) \lessapprox \mu(\VV_W) \mu(\WW)$.   However, our set $\VV$ has a shading $Y$ and it is important to keep track of the shading.   So we are looking for an inequality of the form $\mu(\VV, Y) \lessapprox \mu(\VV_W, Y) \mu(\WW, Y_{\WW})$.   We will need to define $Y_{\WW}$.   In our applications, we are usually given a lower bound for $\lambda(\VV, Y)$,  and for our product estimate to be useful,  we will need to have a lower bound for $\lambda(\WW, Y_{\WW})$.    The main work in this section is to carefully choose $Y_{\WW}$ and prove a lower bound for $\lambda(\WW,  Y_{\WW})$.   

One main result is the following.

\begin{prop}[Factoring and multiplicity]\label{factoringAndMultPropCombined}
Let $\VV = \bigsqcup_{\WW}\VV_W$, where $\VV$ is a family of convex sets and $\WW$ is a family of boxes in $\RR^3$, and each family $\VV_W$ is $C$-Frostman in $W$. Suppose that the convex sets in $\VV$ have approximately the same dimensions, and similarly for $\WW$.   Let $Y$ be a shading on $\VV$.

Then there is a refinement $(\VV', Y')$ and a subset $\WW'\subset\WW$ with a shading $Y_{\WW'}$ so that the following holds.   Define $\VV'_W = \VV_W \cap \VV'$ so that $\VV' = \sqcup_{W \in \WW'} \VV'_W$.  

\begin{enumerate}

\item\label{propFactoringAndMultPropCombinedRefinementItem} $(\VV', Y')$ is a $\gtrapprox 1$ refinement of $(\VV, Y)$.

\item\label{propFactoringAndMultPropCombinedLambdaItem} $\lambda(\WW', Y_{\WW'}) \gtrapprox C^{-1} \lambda(\VV,Y)^2$.


\item\label{propFactoringAndMultPropCombinedVConstMultItem} For each $W \in \WW'$,  $(\VV'_W,  Y')$ has constant multiplicity.   This multiplicity is also the same for all $W \in \WW'$.

\item\label{propFactoringAndMultPropCombinedMultDominatedtItem} For all $W \in \WW'$,  we have  
\begin{equation}\label{muVVBoundedByProduct}
\mu(\VV,Y)\lessapprox \mu(\WW',Y_{\WW'})   \mu(\VV'_W, Y').
\end{equation}

\item\label{propFactoringAndMultPropCombinedShadingDominatedItem} For every $x \in U(\VV', Y')$, 
\begin{equation}\label{pointwiseContainment}
\VV'_{Y'}(x)\subset \bigcup_{\substack{W\in \WW' \\ x\in Y_{\WW'}(W) }}(\VV'_W)_{Y'}(x).
\end{equation}

\item\label{propFactoringAndMultPropCombinedAvgMultOnBallsItem} $|U(\VV', Y') \cap B(x,w_1)|$ is roughly the same for every $x \in U(\WW',  Y_{\WW'})$,  up to a factor $\lesssim 1$ (here $w_1\times w_2\times w_3$ represent the lengths of the axes of the boxes in $\WW$).

\end{enumerate}

\end{prop}

\begin{remark}
Proposition \ref{factoringAndMultPropCombined} is stated in $\RR^3$. It is possible that a similar result holds in higher dimensions, but in order to obtain a lower bound on $\lambda(\WW,Y_{\WW})$ we use a variant of Cordoba's maximal function estimate in the plane. To prove a version of Proposition \ref{factoringAndMultPropCombined} in $\RR^n$, we would require a suitable Kakeya estimate in $\RR^{n-1}$.
\end{remark}
\begin{remark}\label{relaxFrostmanCondition}
The hypothesis that $\VV_W$ is $C$-Frostman in $W$ can be replaced by the following weaker requirement. Let $W\in\WW$ have dimensions $w_1\times w_2\times w_3$. Then instead of requiring that $\VV_W$ is Frostman in $W$, we require that the $w_1$-thickening of each $V\in\VV_W$ is $C$-Frostman in $W$. 
\end{remark}
\begin{remark}\label{compatibilityOfShadingOnWWwithVV}
The shading $Y_{\WW'}$ on $\WW'$ is compatible with the shading $Y'$ on $\VV$, in the sense that $Y_{\WW'}(W)=W\cap N_{w_1}(U(\VV_W, Y'))$, where $w_1$ is the shortest dimension of $W$. 
\end{remark}
Before proving Proposition \ref{factoringAndMultPropCombined}, we first introduce some basic facts about shadings. The following lemma says that a mild refinement cannot substantially decrease the multiplicity of an arrangement.

\begin{lemma} \label{lemmarefinemult}
Let $(\WW,Y)$ be a family of convex sets, and let $(\WW',Y')$ be a $c$-refinement. Then 
\begin{equation} \label{eq: multiplicityrefinement}
	\mu(\WW, Y) \leq  c^{-1}\mu(\WW', Y').
	\end{equation}
In particular, if $(\WW',Y')$ is a $\gtrapprox 1$ refinement of $(\WW,Y)$, then $\mu(\WW, Y) \lessapprox \mu(\WW', Y')$. 
\end{lemma}

\begin{proof} We have $|U(\WW', Y')| \le |U(\WW, Y)|$, and thus 
\[
\mu(\WW, Y) = \frac{ \sum_{\WW} |Y(W)| } {| U(\WW, Y)|} 
\le c^{-1} \frac{ \sum_{\WW'} |Y'(W)| } {| U(\WW, Y)|} 
\le c^{-1} \frac{ \sum_{\WW'} |Y'(W)| } {| U(\WW', Y')|}
=  c^{-1} \mu(\WW', Y').\qedhere
\]

\end{proof}

We remark that if $(\WW',Y')$ is a $\gtrapprox 1$ refinement of $Y$, then it can happen that $\mu(\WW', Y') \gg \mu(\WW, Y)$, but this is okay in our argument because the inequality goes in a favorable direction.

\noindent We say that $(\VV, Y)$ is not empty if there is some $V \in \VV$ with $Y(V)$ non-empty. We begin with a basic lemma about multiplicity.

\begin{lemma} \label{lemmabasicmult} Suppose that $(\VV, Y)$ is not empty.  If $\mu(\VV, Y)(x) \ge \mu_{lower}$ for every $x \in U(\VV, Y)$, then $\mu(\VV, Y) \ge \mu_{lower}$.   If $\mu(\VV, Y)(x) \le \mu_{upper}$ for every $x \in U(\VV, Y)$, then $\mu(\VV, Y) \le \mu_{upper}$.  

\end{lemma}

\begin{proof} For the lower bound, we have

\[
\mu_{lower} |U(\VV, Y)| \le \int_{U(\VV, Y)} \mu(\VV, Y)(x) dx = \sum_{V \in \VV} |Y(V)|.
\]

For the upper bound, we have

\[ 
\sum_{V \in \VV} |Y(V)| = \int_{U(\VV, Y)} \mu(\VV, Y)(x) dx \le \mu_{upper} |U(\VV, Y)|.\qedhere
\]

\end{proof}

Proposition \ref{factoringAndMultPropCombined} involves a shading $Y_{\WW'}$ on $\WW'$. This shading will come from the shading $Y$ on $\VV$ according to the following procedure. 

\begin{defn}\label{inducedShading}
Let $\VV=\bigsqcup_{\WW}\VV_W$, where $\VV$ and $\WW$ are families of convex subsets of $\RR^n$ of dimensions roughly $v_1\times\ldots\times v_n$ and $w_1\times\ldots\times w_n$ respectively. Suppose that each $V\in\VV_W$ is contained in $W$. Let $Y_\VV$ be a shading on $\VV$. We define the \emph{induced shading}
\[
Y_{\WW}(W) = W \cap N_{w_1}\big(U(\VV_W, Y_{\VV})\big).
\]
Note that the induced shading $Y_{\WW}$ depends on the choice of partition $\bigsqcup_{\WW}\VV_W$. This choice will be stated explicitly if it is not clear from context.
\end{defn}

We need a small variation of Lemma \ref{lemmabasicmult} for induced shadings.

\begin{lemma}\label{multiplicityLowerBdInducedShading}
Let $(\VV, Y)$ be a non-empty collection of convex sets and their associated shading. Suppose we have a set of boxes $\WW$ and a partition $\VV=\bigsqcup_{\WW}\VV_W$ as in Definition \ref{inducedShading}, and let $Y_{\WW}$ be the induced shading on $\WW$. If $\mu(\WW, Y_{\WW})(x) \ge \mu_{lower}$ for every $x \in U(\VV, Y)$, then $\mu(\WW, Y_{\WW}) \gtrsim \mu_{lower}$.
\end{lemma}
\begin{proof}
For each $W\in\WW$, define $A_1(W) = N_{w_1}\big(U(\VV_W, Y)\big)$ and $A_2(W) = N_{2w_1}\big(U(\VV_W, Y)\big)$. We have $Y_{\WW}(W) = A_1(W)\cap W$, and $|A_2(W)|\lesssim |Y_{\WW}(W)|$; this is because $|W\cap B|\gtrsim |B|$ for every ball $B$ of radius $w_1$ with center in $W$ (here $W$ has dimensions roughly $w_1\times\ldots\times w_n$). Note: at this step it is important that the sets in $\WW$ are \emph{boxes}, rather than arbitrary convex sets. We also have $\sum_{\WW}\chi_{A_2(W)}(x)\geq\mu_{lower}$ on $N_{w_1}(U(\VV, Y))$; denote the latter set by $X$. We now argue as in Lemma \ref{lemmabasicmult}. We have
\begin{align*}
\mu_{lower}|U(\WW,Y_{\WW})|& 
\leq \mu_{lower} |X| \leq \int_{X} \sum_{W\in\WW} \chi_{A_2(W)}(x) dx
 \lesssim \int_{\RR^n} \mu(\WW,Y_{\WW})(x)
 =\sum_{W \in \WW} |Y_\WW(W)|. \qedhere
\end{align*}
\end{proof}

The key estimate for the proof of Proposition \ref{factoringAndMultPropCombined} is the following lower bound for the density of the induced shading $(\WW,  Y_\WW)$. This estimate is special for $\RR^3$,  whereas all the preceding discussion worked in any dimension.

\begin{lemma}\label{lambdaForInducedShading}
Let $\VV=\bigsqcup_{\WW}\VV_W$ be a collection of convex sets in $\RR^3$,  where the sets of $\VV_W$ are contained in $W$.  Let $Y$ be a shading on $\VV$,  and let $Y_{\WW}$ be the induced shading on $\WW$, as in Definition \ref{inducedShading}. Suppose that for each $W\in\WW$, the set $\VV_W$ is $C$-Frostman in $W$. 
Then for each $W\in\WW$, we have
\begin{equation}
|Y_{\WW}(W)| \gtrapprox C^{-1} \lambda(\VV_{W}, Y)^2|W|.
\end{equation}
\end{lemma}

\begin{proof}
The argument is similar to Cordoba's $L^2$ proof of the Kakeya maximal function conjecture in $\RR^2$. After dyadic pigeonholing, we may select a set $\VV_W'\subset\VV_W,$ dimensions $v_1\leq v_2\leq v_3$, and a number $\lambda\gtrapprox \lambda(\VV_{W})$ so that each set in $\VV_W'$ has dimensions roughly $v_1\times v_2\times v_3$, $|Y(V)|\sim\lambda |V|$ for each $V\in\VV_W'$, and $\sum_{\VV_W'}|Y(V)|\gtrapprox \sum_{\VV_W}|Y(V)|$. We thus have
\[
\lambda|\VV_W'| \gtrsim |V|^{-1} \sum_{\VV_W'}|Y(V)|\gtrapprox |V|^{-1} \sum_{\VV_W}|Y(V)| \gtrapprox \lambda(\VV_W,Y) |\VV_{W}|.
\]
In particular, since $\VV_W$ is $C$-Frostman relative to $W$, we have that $\VV_W'$ is $\lessapprox \frac{\lambda}{\lambda(\VV_W,Y)}C$-Frostman relative to $W$.

Let $w_1\times w_2\times w_3$ be the dimensions of the set $W$ that we fixed above. We have $v_i\leq w_i$ for $i=1,2,3$. Define $p_i = \max(v_i, w_1)$, and define $\PP=\{N_{w_1}(V)\colon V\in\VV_W'\}$; each $P\in\PP$ has dimensions comparable to $p_1\times p_2\times p_3$.   (We note that the sets $P \in \PP$ need not be essentially distinct.   We have $|\PP| = |\VV_{W}'|$. ) Define $Y_{\PP}(P) = P\cap \bigcup_{\VV_W'}N_{w_1}(Y(V))$.

Since each $P\in\PP$ contains a set $V\in\VV_W'[P]$ with $|Y(V)|\geq\lambda|V|$, and since $v_1\leq w_1$, we have that the $w_1$-covering number of $Y(V)$ satisfies
\[
\mathcal{E}_{w_1}(Y(V))\gtrsim 
\left\{\begin{array}{ll}
\lambda w_1^{-2}v_2v_3,&\ \textrm{if}\ v_2\geq w_1,\\
\lambda w_1^{-1}v_3,&\ \textrm{if}\ v_3\geq w_1>v_2,\\
1,&\ \textrm{if}\ w_1>v_3.
\end{array}
\right.
\]
In each of the above scenarios, we have $|Y_{\PP}(P)|\gtrsim w_1^3\mathcal{E}_{w_1}(Y(V)) \gtrsim \lambda|P|$ for each $P\in\PP$.

By Remark \ref{inheritedDownwardsUpwardsRemark}, there is a subset $\PP'\subset\PP$ that is $\lessapprox \frac{\lambda}{\lambda(\VV_W,Y)}C$-Frostman relative to $W$. 
Define $\theta_0 = \frac{p_2}{\min(w_2,p_3)}$. Observe that if $P,P'\in\PP'$ with $|P\cap P'|\geq\theta|P|$ for some $\theta\in[\theta_0,1]$, then $P'$ must be contained in the rectangular prism of dimensions comparable to $p_1\times \theta^{-1}p_2\times p_3$ given by $N_{\theta^{-1}p_2}(P)\cap W$. This prism has volume roughly $\theta^{-1}|P|$, and thus contains $\lessapprox \frac{\lambda}{\lambda(\VV_W,Y)}C \frac{\theta^{-1}|P|}{|W|}|\PP'|$ sets from $\PP'$. We now compute
\begin{align*}
\Big\Vert\sum_{P\in\mathbb{P}'}\chi_P\Big\Vert_2^2 & =\sum_{P\in\PP'}\sum_{\substack{\theta\in[\theta_0, 1]\\ \theta\ \textrm{dyadic}}}\sum_{\substack{P'\in\PP' \\ |P\cap P'|\sim\theta|P|}}|P\cap P'|\\
& \lessapprox \sum_{P\in\PP'}\sum_{\substack{\theta\in[\theta_0,1]\\ \theta\ \textrm{dyadic}}}\Big(\frac{\lambda}{\lambda(\VV_W,Y)}C \frac{\theta^{-1}|P|}{|W|}|\PP'|\Big)\Big(\theta|P| \Big)\\
& \lessapprox \frac{\lambda}{\lambda(\VV_W,Y)}C |\PP'|^2 |P|^2/|W| .
\end{align*}
Finally, by Cauchy-Schwarz and the fact that $\lambda\gtrapprox \lambda(\VV_{W})$, we have
\begin{align*}
|Y_{\WW}(W)|  &\geq |U(\PP', Y_{\PP})| \gtrsim \Big( \lambda|P| |\PP'| \Big)^2/\Big(\Big\Vert\sum_{P\in\mathbb{P}'}\chi_{Y_{\PP}(P)}\Big\Vert_2^2 \Big)\\
&\gtrapprox C^{-1}\lambda \lambda(\VV_W,Y) |W| \geq C^{-1}\lambda(\VV_W,Y)^2|W|.\qedhere
\end{align*}
\end{proof}
\begin{remark}\label{optionalThickeningRemark}
Observe that the hypothesis that $\VV_W$ is $C$-Frostman in $W$ can be replaced with the weaker hypothesis that the $w_1$-thickening of each $V\in\VV_W$ is $C$-Frostman in $W$ (here $w_1$ is the smallest dimension of $W$). Indeed, the quantities $p_1,p_2,p_3$ from the above proof are not changed if we thicken each $V\in\VV_W$ by $w_1$, and the set $\PP$ remains unchanged by this procedure.
\end{remark}

We now have the tools to prove Proposition \ref{factoringAndMultPropCombined}.

\begin{proof}[Proof of Proposition \ref{factoringAndMultPropCombined}]  Let us abbreviate $\lambda = \lambda(\VV,  Y)$.

By pigeonholing, we can select a set $\VV' \subset\VV$ and $\WW'\subset\WW$ so that $(\VV', Y)$ is a $\gtrapprox 1$ refinement of $(\VV, Y)$, $|Y(V)|\gtrapprox \lambda|V|$ for each $V\in \VV'$, $\VV' = \bigsqcup_{\WW'}\VV'_W$, and each set $\VV'_W$ has approximately the same cardinality. We can do this so that each family $\VV'_W$ is $\lessapprox C$-Frostman in $W$.   

By Lemma \ref{lemmarefinemult} we have $\mu(\VV, Y)\lessapprox \mu(\VV', Y)$.   Since $\VV' = \bigsqcup_{\WW'} \VV'_{W}$,  we have 

\[
\mu(\VV', Y) (x) = \sum_{\WW'} \mu(\VV'_{W}, Y) (x).
\]

For each $x \in U(\VV', Y)$ we can choose $\mu_{inner},$ $\mu_{outer},$ and $\tilde \mu$ with $\mu_{outer}\leq\tilde \mu$ so that

\begin{enumerate}

\item[(i)] $\mu(\VV', Y)(x) \approx \mu_{inner} \mu_{outer}$.

\item[(ii)] There are $\sim \mu_{outer}$ sets $W \in \WW'$ so that $\mu(\VV'_{W}, Y) (x) \sim \mu_{inner}$.

\item[(iii)] There are $\sim \tilde \mu$ sets $W \in \WW'$ so that $ \mu(\VV'_W, Y) (x') \sim \mu_{inner}$ for some $x'  \in B(x, w_1)$.

\end{enumerate}

Next we pigeonhole the values of $\mu_{inner},$ $\mu_{outer},$ and $\tilde \mu$.  Define $\Omega(\mu_{inner}, \mu_{outer}, \tilde \mu)$ to be the set of $x$ obeying Items (i) and (ii) and (iii) above.   $U(\VV', Y)$ is equal to the union of $\lesssim (\log|\VV|)(\log|\WW|)^2$ many sets $\Omega(\mu_{inner}, \mu_{outer}, \tilde \mu)$.  We choose $\mu_{inner},$ $\mu_{outer},$ and $\tilde \mu$ so that 

\begin{equation} \label{presmeas}  
\int_{\Omega(\mu_{inner}, \mu_{outer}, \tilde \mu)} \mu(\VV', Y) (x) dx \approx \int_{\RR^n}  \mu(\VV', Y) (x) dx. 
\end{equation}

Next we define the shading $Y'$ on $\VV'$.   If $V \in \VV'_{W}$, we define
\[ 
Y'(V) = Y(V) \cap \Omega(\mu_{inner}, \mu_{outer}, \tilde \mu) \cap \{ x: \mu(\VV'_{W}, Y) (x) \sim \mu_{inner} \}.
\]
We define $Y_{\WW'}$ to be the shading on $\WW'$ induced by the shading $Y'$  on $\VV'=\bigsqcup_{\WW'}\VV'_{W}$ (recall Definition \ref{inducedShading}).   
We have now defined all the characters and we have to check the six items in the conclusion of Proposition \ref{factoringAndMultPropCombined}.

Since $Y_{\WW'}$ is induced by the shading $Y'$,  it follows Item \ref{propFactoringAndMultPropCombinedShadingDominatedItem} and the claim in Remark \ref{compatibilityOfShadingOnWWwithVV} are satisfied.
We have $\mu(\VV'_{W}, Y')(x) \sim \mu_{inner}$ on $U(\VV'_{W}, Y')$.    This gives Item \ref{propFactoringAndMultPropCombinedVConstMultItem}.
Note that $Y_{\WW'}(W)$ is non-empty if and only if $(\VV_{W}, Y')$ is non-empty.  In this case, by Lemma \ref{lemmabasicmult}, we get
\begin{equation}\label{lowerBdOnVVWMuInner}
\mu(\VV_{W}, Y') \gtrapprox \mu_{inner}.
\end{equation}

By Item (ii) we have $\mu(\WW', Y_{\WW'})(x)\gtrsim \mu_{outer}$ for all $x\in U(\VV', Y')$, and thus by Lemma \ref{multiplicityLowerBdInducedShading} we have 
\begin{equation}\label{lowerBdOnWMuOuter}
\mu(\WW', Y_{\WW'}) \gtrapprox \mu_{outer}.
\end{equation}

We have $\mu(\VV', Y')(x) \lesssim \mu_{inner} \mu_{outer}$ on $U(\VV', Y')$. Thus by Lemma \ref{lemmabasicmult}, \eqref{lowerBdOnVVWMuInner}, and \eqref{lowerBdOnWMuOuter}, we conclude that for each $W\in\WW'$ we get Item \ref{propFactoringAndMultPropCombinedMultDominatedtItem}:
\begin{equation}\label{boundMuVVYForAllOfW}
\mu(\VV, Y) \lessapprox \mu(\VV', Y') \lessapprox \mu_{inner} \mu_{outer} \lessapprox  \mu(\WW', Y_{\WW'})  \mu(\VV'_{W}, Y').
\end{equation}
%
 
 Next we check Item \ref{propFactoringAndMultPropCombinedRefinementItem}: $(\VV', Y')$ is a $\gtrapprox 1$ refinement of $(\VV,Y)$.   Since $(\VV',Y)$ is a $\gtrapprox 1$ refinement of $(\VV, Y)$, we have to check that $(\VV', Y')$ is a $\gtrapprox 1$ refinement of $(\VV', Y)$,  which means that
\begin{equation} \label{keptmass2}
 \sum_{V \in \VV'} |Y'(V)| \gtrapprox \sum_{V \in \VV'} |Y(V)|. 
\end{equation}

To check the claim,  the key point is that if $x \in  \Omega(\mu_{inner}, \mu_{outer}, \tilde \mu)$,  then $\mu(\VV',  Y)(x) \lessapprox \mu_{inner} \mu_{outer} \lesssim \mu(\VV', Y')(x)$.   Combining this with \eqref{presmeas},  we see that 
\[ \sum_{V \in \VV'} |Y'(V)| = \int \mu(\VV',  Y')(x) \ge \int_{\Omega(\mu_{inner}, \mu_{outer}, \tilde \mu)} \mu(\VV', Y')(x) \gtrapprox \]

\[ \gtrapprox \int_{\Omega(\mu_{inner}, \mu_{outer}, \tilde \mu)} \mu(\VV', Y)(x) \approx \int \mu(\VV', Y)(x) = \sum_{V \in \VV'} |Y(V)|. \]
This proves Item \ref{propFactoringAndMultPropCombinedRefinementItem}.   

In order to check Item \ref{propFactoringAndMultPropCombinedAvgMultOnBallsItem},  we make one more refinement.   We cover $U(\WW',  Y_{\WW'})$ by boundedly overlapping balls of radius $w_1$. After dyadic pigeonholing, we may select a $\approx 1$ refinement of $Y'$, obtained by replacing each set $Y'(V)$ by $Y'(V)\cap Z$ for a common set $Z$, so that $|U(\VV', Y')\cap B_{w_1}|$ is roughly the same for each of these balls.   By abuse of notation,  we will still call this refinement $Y'$.   There is a corresponding refinement of $Y_{\WW'}$ from applying Lemma \ref{multiplicityLowerBdInducedShading},  and we will still call it $Y_{\WW'}$.   These refinements do not disturb the items we have already established, and they give us Item \ref{propFactoringAndMultPropCombinedAvgMultOnBallsItem}.  In particular,  since we have made only a $\gtrapprox 1$ refinement of $(\VV', Y')$,  \eqref{keptmass2} and Item \ref{propFactoringAndMultPropCombinedRefinementItem} still hold.

Finally we check Item \ref{propFactoringAndMultPropCombinedLambdaItem}:  $\lambda(\WW', Y_{\WW'}) \gtrapprox C^{-1} \lambda^2$. For dyadic $\tau$, we define $\WW'_\tau$ as
\[ \WW'_\tau = \big\{ W \in \WW': \sum_{V \in \VV'_W} |Y'(V)| \sim \tau \lambda |V| |\VV'_W| \big\}.  \]

Combining  \eqref{keptmass2} and the assumption that $|Y(V)| \gtrapprox \lambda |V|$ for every $V \in \VV'$,  we see that 
\[ \sum_{\tau} \sum_{W \in \WW'_\tau} \tau \lambda |V| |\VV'_W| \sim \sum_{V \in \VV'} |Y'(V)| \gtrapprox \sum_{V \in \VV'} |Y(V)| \gtrapprox \sum_{\tau} \sum_{W \in \WW'_\tau}  \lambda |V|  |\VV'_W| . \]
Since $|\VV'_W|$ are essentially constant in $W$,  we can simplify this formula to
\begin{equation}
\sum_{\tau} \tau |\WW'_\tau| \gtrapprox |\WW'|.
\end{equation}

This implies that
\begin{equation} \label{smalltaunegl}
\sum_{\tau \gtrapprox 1} \tau |\WW'_\tau| \gtrsim \sum_{\tau} \tau |\WW'_\tau| \gtrapprox |\WW'|.
\end{equation}
If $W \in \WW'_\tau$,  then since $\VV'_W$ is $\lessapprox C$-Frostman in $W$, Lemma \ref{lambdaForInducedShading} gives the bound
\[ |Y_{\WW'}(W) | \gtrapprox C^{-1} \tau^2 \lambda^2 |W|. \]

Now we are ready to estimate $\lambda(\WW', Y_{\WW'})$.
\[ \sum_{W \in \WW'} |Y_{\WW'}(W)| \ge \sum_\tau |\WW'_\tau| C^{-1} \tau^2 \lambda^2 |W|. \]
We can bound this using \eqref{smalltaunegl}.   If $\tau \gtrapprox 1$,  then $\tau^2 \gtrapprox \tau$,  and so 
\[  \sum_\tau |\WW'_\tau| C^{-1} \tau^2 \lambda^2 |W|  \gtrapprox \sum_{\tau \gtrapprox 1}  |\WW'_\tau|  \tau C^{-1} \lambda^2 |W| \gtrapprox |\WW'| C^{-1} \lambda^2 |W| . \]
In total,  we have

\[ \sum_{W \in \WW'} |Y_{\WW'}(W)| \gtrapprox |\WW'| C^{-1} \lambda^2 |W|. \]
Rearranging this gives $\lambda(\WW', Y_{\WW'}) \gtrapprox C^{-1} \lambda^2$ as desired.  
\end{proof}

We also state an easier variant of Proposition \ref{factoringAndMultPropCombined} in the special case when $\VV$ and $\WW$ are both sets of tubes.  

\begin{lemma}\label{shadingMultiplicityEstimateForRhoTubes}
Suppose $\TT$ is a uniform set of $\delta$-tubes and $Y$ is a shading on $\TT$. 

Let $\rho\in[\delta,1]$ and let $\TT = \bigcup_{T_{\rho} \in \TT_{\rho}} \TT[T_\rho]$.    
Let $Y$ be a shading on $\TT$ with $|Y(T)|\sim\lambda|T|$ for each $T \in \TT$.

Then there is a subset $\TT_{\rho}'\subset\TT_\rho$ and a shading $Y_{\TT_\rho'}$ on $\TT_{\rho}'$ with $\lambda(\TT_{\rho}',Y_{\TT_
	\rho'})\gtrapprox \lambda$. There is a shading $Y'$ on $\TT'=\bigsqcup_{\TT_{\rho}'} \TT[T_\rho]$ so that $(\TT', Y')$ is a $\approx 1$ refinement of $(\TT,Y)$, and for all $T_\rho \in \TT_{\rho}'$ we have
\begin{equation}\label{boundMuTTYAcrossTwoScales}
\mu(\TT,Y)\lessapprox \mu(\TT_\rho',Y_{\TT_\rho'})\mu(\TT[T_\rho], Y').
\end{equation}

We also have the pointwise containment
\begin{equation}\label{pointwiseContainmenttube}
\TT_{Y'}(x)\subset \bigcup_{\substack{T_\rho \in \TT_\rho' \\ x\in Y_{\TT_\rho'}(T_\rho) }}(\TT[T_\rho])_{Y'}(x).
\end{equation}

Furthermore,  for each $x\in U(\TT_\rho',  Y_{\TT_\rho'})$ we have
\begin{equation}\label{boundVolumeAcrossTwoScales}
|U(\TT,Y')| \gtrapprox |U(\TT_\rho', Y_{\TT_\rho'})|\ \frac{|U(\TT, Y')\cap B(x,\rho)|}{|B(x,\rho)|}.
\end{equation}
\end{lemma}

The proof of \eqref{boundMuTTYAcrossTwoScales} is the same as the proof of Proposition \ref{factoringAndMultPropCombined}, except in place of Lemma \ref{lambdaForInducedShading} we use the fact that if $T_\rho\supset T$, then $|T_\rho \cap N_{\rho}(Y(T))|\gtrsim \frac{|Y(T)|}{|T|}|T_\rho|$. Thus we do not need to assume that $\TT_{T_\rho}$ is $C$-Frostman in $T_\rho$, and we obtain the estimate $\lambda(\TT_\rho', Y_{\TT_\rho'})\gtrapprox \lambda$, which is stronger than its analogue in Proposition \ref{factoringAndMultPropCombined}.  


\section{Multiplicity estimates for slabs and planks}
In Section \ref{multAndFactoringSection} we reduced the problem of controlling $\mu(\VV,Y_{\VV})$ for an arbitrary collection of convex sets in $\RR^3$ to the problem of controlling the multiplicity of two new families of convex sets, one of which is Katz-Tao and the other is Frostman. The estimates $K_{KT}(\beta)$ and $K_{F}(\beta)$ control the intersection multiplicity of sets of \emph{tubes} in $\RR^3$ that are Katz-Tao and Frostman, respectively. In this section we will show that incidence bounds for arbitrary convex sets in $\RR^3$ can be systematically reduced to incidence bounds for tubes. 

\medskip

\noindent{\bf Katz-Tao multiplicity estimates.} Lemma \ref{plankKTUnified} below says that the Kakeya estimate $K_{KT}(\beta)$ implies an analogous estimate for planks of dimensions $a\times b\times 1$. Furthermore, a stronger estimate holds when $a<\!\!<b$, provided that the planks do not concentrate into slabs (i.e. sets of dimensions $\theta\times 1\times 1$ for $\theta<\!\!< 1$). To make this latter statement precise, we must define what it means for planks to concentrate into slabs. 

If $\PP$ is a set of planks in $\RR^3$ of dimensions $a\times b\times 1$, $\theta \in [a/b,1]$, and $S$ is a $\theta\times 1\times 1$ slab, then we define 
\begin{equation} \label{defPS} 
\PP_S = \{ P \in \PP: P \subset S \textrm{ and } \angle ( TP, TS) \leq \theta \}.
\end{equation}
In the above, $\angle ( TP, TS)$ refers to the angle between the planes $TP, TS$, where $TP$ is the plane spanned by the two longest axes of $P$, and similarly for $S$ (this is defined up to accuracy $\theta$). 
\begin{lemma} \label{plankKTUnified}  
Suppose that $K_{KT}(\beta)$ holds. Then for every $\epsilon>0$, there exists $\eta,b_0>0$ such that the following holds for any $0<a\leq b \leq b_0$.    Let $\PP$ be a set of $a \times b \times 1$ planks in $B_1$, and let $Y_{\PP}$ be a shading on $\PP$ with density $\lambda(\PP,Y_{\PP}) \geq  a^{\eta} $. Suppose that the planks do not concentrate into slabs, in the following sense: there exists $0\leq\gamma\leq 1$ so that 
for every $\theta \in [a/b, 1]$ and every $\theta \times 1 \times 1$ plank $S$,  $| \PP_S | \leq a^{-\eta}  \theta^\gamma | \PP|$.

If $\PP$ is Katz-Tao, then
\[
\mu(\PP) \leq a^{-\epsilon}  \left( \frac{a}{b} \right)^{(1+\gamma)\beta/2}\big((ab)^{-1}|\PP|\big)^{\beta/2}.
\]
More generally, we have 
\begin{equation}\label{MuPPEstimate}
\mu(\PP) \leq a^{-\epsilon}  \Delta_{max}(\PP)^{1-\beta/2}  \left( \frac{a}{b} \right)^{(1+\gamma)\beta/2}\big((ab)^{-1}|\PP|\big)^{\beta/2}.
\end{equation}
\end{lemma}
\begin{remark}
 We will only apply Lemma \ref{plankKTUnified} in the special cases $\gamma=0$ (in which case the non-concentration condition is vacuous, and can be ignored), and $\gamma = 1$.  When $a=b$, $\PP$ is Katz-Tao, and $\gamma=0$, the RHS of \eqref{MuPPEstimate} is $\lessapprox (a^{-2}|\PP|)^{\beta/2}$, which matches the estimate $K_{KT}(\beta)$ for $a$-tubes.
\end{remark}

\begin{remark} \label{RHSplankKT} The inequality automatically holds in the case when $\mu(\PP) \le a^{-\eta}$.   The right-hand side of \eqref{MuPPEstimate} is always $\ge a^{- \eta}$, because the hypothesis implies that $|\PP| \ge a^\eta(b/a)^\gamma$ and $\eta$ is much less than $\epsilon$.  
\end{remark}

\medskip

\noindent{\bf Frostman multiplicity estimates.} Lemma \ref{plankF} below says that the Kakeya estimate $K_{F}(\beta)$ implies an analogous estimate for planks of dimensions $a\times b\times 1$, provided we avoid the following bad scenario: Let $\PP$ be a set of planks of dimensions $a\times b\times 1$, let $P\in\PP$, let $\theta\in[a/b, 1]$. We define the ``thickened plank'' $P_\theta=N_{\theta b}(P)$, i.e. $P_\theta$ has dimensions roughly $\theta b\times b\times 1$. If $P_\theta$ contains a very large number of planks from $\PP$, then this can cause an unacceptably large intersection multiplicity $\mu(\PP)$. The next result says that this is the only obstruction to obtaining an analogue of $K_F(\beta)$ for planks.

\begin{lemma} \label{plankF}  
Suppose that $K_{F}(\beta)$ holds.  Then for every $\epsilon>0$, there exists $\eta,b_0>0$ such that the following holds for any $0<a\leq b \leq b_0$.   Let $\PP$ be a set of $a \times b \times 1$ planks in $B_1$, and let $Y$ be a shading on $\PP$ with density $\lambda(\PP,Y) \geq  a^{\eta} $. Suppose that the planks do not concentrate into thickened planks, in the following sense: there is a number $M\geq 1$ so that for each $\theta \in [a/b, 1]$ and every $P\in\PP,$ we have $|\PP[P_\theta]| \leq M\theta.$
Then
\begin{equation}
|U(\PP,Y)|\geq a^{\eps}C_F(\PP)^{\beta/2-1} b^{2\beta}\big(M^{-1}b^2 |\PP|\big)^{\beta/2},
\end{equation}
and
\begin{equation}
\mu( \PP) \leq a^{-\epsilon}C_F(\PP)^{1-\beta/2}   M^{\beta/2} \frac{a}{b} b^{-2 \beta}  \left( b^2 |\PP| \right)^{1 - \beta/2}.
\end{equation}
\end{lemma}

\begin{remark} \label{RHSplankF} Both inequalities automatically hold in the case when $\mu(\PP) \le a^{-\eta}$.   To see this, note that $M \ge b/a$ (by taking $\theta =a/b$) and that $C_F(\PP) \gtrsim \frac{1}{|\PP| a b }$.  
\end{remark}

\medskip

\noindent {\bf Tubes that factor through planks.}
Lemmas \ref{plankKTUnified} and \ref{plankF} allow us to obtain improved multiplicity bounds for families of tubes that factor through planks. The precise statement is as follows.

\begin{prop}\label{factoringThoughFlatPrisms}
Suppose that $K_{KT}(\beta)$ and $K_F(\beta)$ hold. Then for all $\eps>0$, there exists $\eta,\delta_0>0$ so that the following holds for all $\delta\leq\delta_0$. Let $(\tubes,Y)$ be uniform, with $\lambda(\TT,Y)\geq\delta^{\eta}$.

\medskip
\noindent (A) Suppose $\TT$ is essentially distinct and there is a scale $\rho$ so that for each $T_\rho\in\tubes_\rho$, there is a family of $a\times b\times 1$ planks that factors $\TT[T_\rho]$. Then
\begin{equation}\label{multiplicityBoundFactorThroughFlatFrostman}
\mu(\TT,Y)\leq \delta^{-\epsilon}C_F(\TT)^{1-\beta/2}\big(\frac{a}{b}\big)^{3\beta/2} \delta^{-2\beta} \left( \delta^2 | \TT | \right)^{1 - \beta/2}.
\end{equation}

\medskip
\noindent (B) Suppose there is a scale $\rho$ and a family of $a\times b\times 1$ planks that factors $\TT_\rho$. Then 
\begin{equation}\label{boundMuABPlanksFactorTT}
\mu(\TT,Y)\leq \delta^{-\eps}\Delta_{\max}(\TT)^{1-\beta/2}\big(\frac{a}{b}\big)^{3\beta/2}\big(\delta^{-2}|\TT|\big)^{\beta/2}.
\end{equation}
\end{prop}
Note that the analogue of Remark \ref{multBoundsDeltaVsRho} applies to Lemmas \ref{plankKTUnified} and \ref{plankF}, and to Proposition \ref{factoringThoughFlatPrisms}. The remainder of this section is devoted to proving Lemmas \ref{plankKTUnified} and \ref{plankF}, and Proposition \ref{factoringThoughFlatPrisms}.

\subsection{Slabs}

A standard $L^2$ method gives strong bounds for the incidences of slabs.   A key character is the notion of typical intersection angle.  Suppose that $\SSS$ is a set of $\delta \times 1 \times 1$ slabs with a shading $Y$.  Define the set of incidence triples

\[
Tri(\SSS) =  \{ (x, S_1, S_2) : x \in Y(S_1) \cap Y(S_2) \}.
\]

There is a natural measure on $Tri(\SSS)$ coming from the volume measure for $x$ and the counting measure for $S_1, S_2$.  We can define the angle between two slabs $S_1, S_2$, which is well defined up to additive error $\delta$.  Then we define

\[
Tri_\theta(\SSS) =  \{ (x, S_1, S_2) \in Tri(\SSS): \textrm{angle}(S_1, S_2) \sim \theta \}.
\]

\begin{defn}\label{typicalIntersectionAngle}
We say that $\theta$  is a typical angle of intersection for $\SSS$ if $|Tri_\theta(\SSS)| \approx |Tri(\SSS)|$ up to $\log(1/\delta)$ factors.   If $\mu(\SSS) \ge 2$, then $\SSS$ always has a typical intersection angle $\theta \in [\delta, 1]$. 
\end{defn}

\begin{lemma} \label{slab3} Suppose $\SSS$ is a set of $\delta \times 1 \times 1$ slabs contained in $\tilde S$,  a $\theta \times 1 \times 1$ slab,  with $\delta \le \theta \le 1$.   Suppose the slabs of $\SSS$ are shaded with $\lambda(\SSS, Y) \geq \delta^{\eta}$.   If $\theta$ is a typical intersection angle for $\SSS$, then  $|U(\SSS, Y)| \gtrapprox \delta^{2\eta} |\tilde S|$.
\end{lemma}

\begin{proof} We let $f = \sum_{S \in \SSS} 1_{Y(S)}$.   Note that 

\[
\int f^2 = |Tri(\SSS)| = \sum_{S_1, S_2 \in \SSS} |Y(S_1) \cap Y(S_2)|.
\]

By hypothesis, we have

\[ 
\int f^2 =  | Tri(\SSS) | \approx \sum_{S_1, S_2 \in \SSS, \textrm{angle}(S_1, S_2) \sim \theta } |Y(S_1)\cap Y(S_2)| \le |\SSS|^2 \frac{\delta^2}{\theta}.
\]
On the other hand, we know that $\int f \gtrsim  |\SSS| |S| \delta^{\eta} = |\SSS| \delta^{1+\eta}$.  By Cauchy-Schwarz,  since $f$ is supported in $U(\SSS)$,  

\[
\int f^2 \ge \frac{ ( \int f )^2}{|U(\SSS)|} \approx \frac{ |\SSS|^2 \delta^{2+2\eta} }{ |U(\SSS)|}.
\]

Comparing,  we get 

\[ 
|U(\SSS)| \gtrapprox \theta \delta^{2\eta} = |\tilde S|\delta^{2\eta}. \qedhere
\]

\end{proof}

This $L^2$ bound leads to sharp estimates for the analogue of Theorem \ref{thmkak} where tubes are replaced by slabs that have thickness $\delta$ and length and width close to 1.

Finally, we record the following variant of Cordoba's maximal function estimate in the plane. The proof is similar to that of Lemma \ref{slab3}, so we will only provide a brief sketch.
\begin{lemma} \label{slab1} 
Suppose $\SSS$ is a set of $\delta \times b \times c$ slabs in $B_1$ shaded with $\lambda(\SSS) \geq \delta^{\eta}$.  If $b,c\geq\delta^{\eta}$ and $\Delta_{max}(\SSS) \leq \delta^{-\eta}$, then  $ \mu(\SSS,Y) \lesssim \delta^{-8\eta}. $  If $C_F(\SSS) \leq \delta^{-\eta}$, then $|U(\SSS)|\geq \delta^{8\eta}. $
\end{lemma}
\begin{proof}
Let $f = \sum_{S \in \SSS} 1_{Y(S)}$. Let $k_0=\lceil \log 1/\delta\rceil.$ For each $S_1\in\SSS$ and $k=0,\ldots,k_0$, let $S_1^k=N_{2^{k}\delta}(S_1)\cap B_1$ denote the intersection of the $2^k\delta$-neighbourhood of $S_1$ with the unit ball; this set is contained in a box of dimensions $2^{k+1}\delta\times 2\times 2$. Define $S_1^{-1}=\emptyset$. Note that if $S_2\in\SSS$ with $S_2\in  N_{2^{k}\delta}(S_1)\backslash N_{2^{k-1}\delta}(S_1)$, then $|S_1\cap S_2|\lesssim 2^{-k}\delta$. We have
\[
\int f^2 \leq \sum_{S_1,S_2\in\SSS}|S_1\cap S_2| \leq \sum_{S_1\in\SSS}\ \sum_{k=0}^{k_0} \sum_{\substack{S_2\in\SSS \\ S_2\in S_1^k\backslash S_1^{k-1}}}|S_1\cap S_2|\lesssim \sum_{S_1\in\SSS}\ \sum_{k=0}^{k_0} 2^{-k}\delta |\SSS[S_1^k]|.
\]

If $\Delta_{max}(\SSS) \leq \delta^{-\eta}$, then $|\SSS[S_1^k]|\lesssim\delta^{-\eta}|S_1^k||S|^{-1}\lesssim\delta^{-3\eta}2^k$,
and thus $\int f^2\lessapprox \delta^{1-3\eta}|\SSS|$. If $C_F(\SSS) \leq \delta^{-\eta}$, then an analogous computation yields $\int f^2\lessapprox \delta^{1-3\eta}|S||\SSS|^2$.

Comparing these bounds with $\int f \gtrsim  |\SSS| |S| \delta^{\eta}$, with $|S|\sim\delta bc$, completes the proof
\end{proof}

\subsection{Planks}

Next we suppose that $\PP$ is a set of $a \times b \times 1$ planks in $B_1$.   We will study $\mu(\PP)$ by relating it to incidence problems involving slabs and tubes.   In a sense,  an incidence problem for planks can be broken down into an incidence problem for slabs and an incidence problem for tubes.   The incidence problem for slabs is solved completely using Lemma \ref{slab3}.   Assuming $K_{KT}(\beta)$ or $K_F(\beta)$ we get information about the incidence problem for tubes.   In this way,  we are able to prove estimates for $\mu(\PP)$ based on various geometric assumptions about $\PP$.  

Let us describe this process informally first and then state a formal lemma.  Given two intersecting planks, $P_1, P_2$, define the angle of intersection as the angle between the planes $TP_1, TP_2$, where $TP$ is the plane spanned by the two longest axes of $P$. If the planks $P_1$ and $P_2$ have dimensions $a\times b\times 1$, then this angle is meaningfully defined up to accuracy $a/b$. Let $\theta$ denote a ``typical'' angle of intersection for $\PP$.

First we note that if the angle between $P_1$ and $P_2$ is $\gg \theta$, then their intersection will not play an important role.  Therefore, we can subdivide $\PP$ into smaller collections of planks as follows.  If $S$ is a $\theta \times 1 \times 1$ slab in $B_1$, we define $\PP_S$ as in \eqref{defPS}. We can now decompose $\PP$ as $\PP = \sqcup_S \PP_S$, where $S$ varies over some essentially distinct slabs $S$.  If $\theta$ is a typical angle of intersection of $(\PP, Y)$, then morally the sets $U(\PP_S, Y)$ are disjoint and we can study them separately.  Fix a single $S$ and we study $\PP_S$.  

Next let $Q$ denote a $\theta b \times b \times b$ box, and let $\PP_Q$ be the set of planks $P \in \PP_S$ so that $P \cap Q$ has dimensions $a \times b \times b$.  In other words, $\PP_Q$ are the planks that intersect $Q$ ``tangentially'', making the intersection as large as possible.  Assuming that $\mu(\PP) \gg 1$ and assuming that $\PP_Q$ is not empty, then morally Lemma \ref{slab3} tells us that $|U(\PP_S, Y) \cap Q| \approx |Q|$.   For each plank $P \in \PP_S$, the $\theta b$-neighborhood of $P$ can be covered by boxes $Q$ as above, and so we see that morally most of this neighborhood is contained in $U(\PP_S, Y)$.   This motivates defining thickened version of $\PP_S$ as follows.

\begin{defn}\label{thickenedPlanks}
For each $P\in\PP$, define $P_{\theta}$ to be the $\theta b$-neighborhood of $P$. Define $\PP_{\theta}$ to be a set of essentially distinct prisms of the form $P_\theta$, so that each $P'\in\PP$ is contained in a set $P_\theta\in\PP_\theta$ for which $P'_\theta$ is comparable to $P_\theta$. Note that if $P_\theta$ and $P'_\theta$ are comparable, then they are contained in comparable slabs $S$ and $S'$. Thus $\PP_{\theta}$ decomposes into a boundedly overlapping union $\bigcup_S \PP_{\theta,S}$, where the union is taken over incomparable slabs of dimensions $\theta\times 1\times 1$.
\end{defn}

In the previous paragraph, we saw that if $\theta$ is a typical angle of intersection for $\PP$, then morally $U(\PP_S)$ is almost the same as $U(\PP_{\theta, S})$.  

Next we have to study $U(\PP_{\theta, S})$.  There is a change of variables that takes $S$ to the unit ball and takes each $P_{\theta} \in \PP_{\theta, S}$ to a tube of radius $b$.  Therefore, the problem of estimating $|U(\PP_{\theta, S})|$ is equivalent to a problem about tubes which we can bound using $K_{KT}(\beta)$ and/or $K_F(\beta)$.

The following two lemmas give a precise version of the moral statements above. First we carefully define the typical angle of intersection.

\begin{lemma}\label{findingTypicalAngleOfIntersection}
Let $(\PP, Y)$ be a family of $a\times b\times 1$ planks and their associated shading. Let $\eta>0$ and suppose that $\mu(\PP, Y)\geq a^{-\eta}$. Then there is an $\approx 1$ refinement $(\PP, Y')$ of $(\PP, Y)$ and a number $\theta\in[a/b, 1]$ with the following properties. In what follows, let $A = e^{ \sqrt{ \log a^{-1}}}$.

\begin{itemize}
\item $(\PP, Y')$ has constant multiplicity on $U(\PP, Y')$ up to a factor $\sim 1$.


\item If $\PP_{Y''} (x) \subset \PP_{Y'}(x)$ with $|\PP_{Y''}(x)| \ge \frac{1}{ A} | \PP_{Y'}(x)|$,  then

\[
\max_{P_1, P_2\in \PP_{Y''} (x)} \angle (TP_1, TP_2)\approx \theta .
\]

\end{itemize}
\end{lemma}
\begin{proof} 
First we replace $(\PP, Y)$ with an $\approx 1$ refinement $(\PP, Y_1)$ that has constant multiplicity. Next, fix a point $x \in U(\PP, Y_1)$.  

Let $B = e^{ (\log a^{-1} )^{3/4} }$.   The point here is just that $A^N \ll B$ and $B^N \ll a^{-1}$ for any real number $N$.  

For a subset $\PP' \subset \PP_{Y_1}(x)$, we define the maximal angle of $\PP'$ as 

\[
M(\PP') = \max_{P_1, P_2 \in \PP'} \angle (TP_1, TP_2).
\]

We first check if there is a subset $\PP_1 \subset \PP_{Y_1}(x)$ so that $|\PP_1| \ge \frac{1}{ A} | \PP_{Y_1}(x)|$ and with $M(\PP_1) \le \frac{1}{ B}$.  
If not, we set $\PP_{Y_{almost}}(x) = \PP_{Y_1}(x)$ and $\theta(x) = 1$.  

If there is such a $\PP_1$, we select such a set $\PP_1$, and we check if there is a subset $\PP_2 \subset \PP_1$ so that $|\PP_2| \ge \frac{1}{A} | \PP_1 |$ and $M(\PP_2) \le (1/B) M(\PP_1)$.
If not, we output $\PP_{Y_{almost}}(x) = \PP_1$ and $\theta(x) = M(\PP_1)$.  

Since $| \PP_{Y_1}(x)| \approx \mu(\PP, Y_1) \gtrapprox \mu(\PP, Y) \gtrapprox a^{-\eta}$ and since $M(\PP_k) \ge a/b \ge a$ for every $k$,  this process must stop eventually at some $\PP_k$ with $k \lesssim \log_B a^{-1}$ and $|\PP_k| \ge A^{-k} | \PP_{Y_1}(x)| \gtrapprox | \PP_{Y_1}(x) |$.  
Then we output $\PP_{Y_{almost}}(x) = \PP_k$.  Since we have $| \PP_{Y_{almost}}(x) | \gtrapprox | \PP_{Y_1}(x)| $ for each $ x \in U(\PP, Y_1)$, we see that $Y_{almost}$ is a $\gtrapprox 1$ refinement of $Y_1$.  

At this moment, we have $\theta(x)$ which may be different at different points $x$.  We let $Y'$ be a $\gtrsim \frac{1}{ \log a^{-1}}$ refinement of $Y_{almost}$ so that $\theta(x)$ is roughly constant on $U(\PP, Y')$ (up to a factor $\sim 1$) and so that $|\PP_{Y_{almost}}(x)|$ is roughly constant on $U(\PP, Y')$ (up to a factor $\sim 1$).
\end{proof}
\begin{remark}\label{workWithCubesRatherThanPoints}
For some applications, it will be helpful to work at ``resolution $a$'' in the following sense: if each shading $Y(P)$ is a union of grid-aligned $a/100$ cubes, then we can choose the shading $Y'(P)$ so that the same is true. Indeed, in the above proof, we consider each grid-aligned $a/100$ cube $Q$, rather than each point $x$. 
\end{remark}

\begin{defn}  \label{deftypicalangleplank} We say that $\theta$ is a typical angle of intersection for 
$(\PP, Y)$ if $(\PP, Y)$ satisfies the conclusions of Lemma \ref{findingTypicalAngleOfIntersection}  for some $A \ge e^{ (\log(a^{-1}))^{1/4} }$.    
\end{defn}

In particular,  if $(\PP, Y')$ and $\theta$ are the output of Lemma \ref{findingTypicalAngleOfIntersection},  then $\theta$ is a typical angle of intersection for $(\PP, Y')$.   We made the definition of ``typical angle of intersection'' a little bit looser than the conclusion of Lemma \ref{findingTypicalAngleOfIntersection} so that some mild refinements of $(\PP, Y')$ will still have typical angle of intersection $\theta$.
This definition of typical angle of intersection for planks is stronger than the definition that we used for typical intersection angle of slabs,  Definition \ref{typicalIntersectionAngle}.   In this definition for planks,  at every point $x \in U(\PP, Y)$, the intersecting planks have angle at most $\theta$ and there are many pairs of planks with angle $\sim \theta$. Furthermore, this remains true if the shading $Y$ is replaced by a mild refinement.  

\begin{lemma} \label{lemmaredplanktube} Let $\eta> 0$ and let $(\PP, Y)$ be a set of $a \times b \times 1$ planks in the unit ball with $\lambda(\PP, Y) \ge a^\eta$ and $\mu(\PP, Y) \ge a^{- \eta}$. 

Then there is a $\gtrapprox 1$ refinement $(\PP',Y')$ of $(\PP,Y)$ with $|\PP'|\gtrapprox a^\eta|\PP|$ and $\lambda(\PP', Y')\gtrapprox \lambda(\PP, Y)$, a number $N\geq 1$, a number $\theta \in [a/b, 1]$, and a set of $\theta \times 1 \times 1$ slabs $\SSS$ so that the following holds.  
For each $S\in\SSS$, let $\PP'_S$ and $\PP'_{\theta, S}$ be as defined in \eqref{defPS} and Definition \ref{thickenedPlanks}.   $\PP' = \bigcup_{S \in \SSS} \PP'_S$.  Then each $P_{\theta} \in \PP'_{\theta, S}$ contains $\sim N$ planks $P \in \PP'_S$. Furthermore, there is a shading $Y_{\theta, S}$ on $\PP'_{\theta, S}$ so that

\begin{enumerate}

\item \label{itemballfull} If $U(\PP', Y') \cap B_{\theta b}$ is non-empty, then $|U(\PP', Y') \cap B_{\theta b}| \gtrapprox a^{4 \eta} |B_{\theta b}|$.

\item  For each $S \in \SSS$, $\lambda(\PP'_{\theta, S},  Y_{\theta, S}) \gtrapprox a^{4 \eta}$ and for each $P_\theta\in\PP'_{\theta, S}$, $Y_{\theta, S}(P_\theta)\subset U(\PP'_S, Y')$ 

\item  $ |U(\PP, Y)| \gtrapprox  \sum_{S \in \SSS} |U(\PP'_{\theta,S}, Y_{\theta, S})|.$

\item $\mu(\PP,Y) \lessapprox \max_{S \in \SSS} a^{-4 \eta} \frac{aN }{b \theta} \mu(\PP'_{\theta, S}, Y_{\theta, S}).$

\end{enumerate}

\end{lemma}
\begin{remark}\label{typicalAngleAlreadyPresent}
If $(\PP, Y)$ has a typical angle of intersection $\theta$, then Lemma \ref{lemmaredplanktube} holds for this value of $\theta$. The analogue of Remark \ref{workWithCubesRatherThanPoints} also holds in this setting---if each $Y(P)$ is a union of grid-aligned $a/100$ cubes, then we can ensure that the same is true for the refinement $Y'(P)$.
\end{remark}


\begin{proof}
 We begin with a pair $(\PP, Y)$ with $\mu(\PP, Y) \ge a^{-\eta}$. Apply Lemma \ref{findingTypicalAngleOfIntersection} to $(\PP, Y)$.   This yields a refinement $(\PP, Y_1)$ and a typical angle of intersection $\theta$ obeying the conclusions of Lemma \ref{findingTypicalAngleOfIntersection}:
 
 \begin{itemize}
 
 \item Our set $(\PP, Y_1)$ has roughly constant multiplicity: for each $x \in U(\PP, Y_1)$,  $| \PP_{Y_1}(x) | \sim m$. 

\item For each $x \in U(\PP, Y_1)$, $\max_{P_1, P_2\in \PP_{Y_1}(x)} \angle (TP_1, TP_2)\sim \theta $

\item If $\PP_{Y''} (x) \subset \PP_{Y_1}(x)$ with $|\PP_{Y''}(x)| \ge e^{ - \sqrt{ \log (a^{-1} )} } m$,  then

\begin{equation} \label{typtheta}
\max_{P_1, P_2\in \PP_{Y''} (x)} \angle (TP_1, TP_2)\approx \theta .
\end{equation}

\end{itemize}
As per Remark \ref{typicalAngleAlreadyPresent}, if $(\PP, Y)$ already has a typical angle of intersection, then we can skip this step.

  In the rest of the proof, we will continue to refine $\PP$ and $Y_1$.   We will consider various $\tilde \PP \subset \PP$ and $\tilde Y$ refining $Y_1$.   For each refinement $(\tilde \PP, \tilde Y)$ that we consider,  we will check that

\begin{equation} \label{masskept}  \sum_{\tilde \PP} |\tilde Y(P)| \ge  (\log 1/a)^{-O(1)} \sum_{\PP} |Y_1(P)|
\end{equation}

We will call a refinement good if it obeys \eqref{masskept}.   If $\tilde Y$ is a good refinement,  we will further refine  $\tilde Y$ to have roughly constant multiplicity:

\begin{equation} \label{multfull}  \textrm{ For each } x \in U(\tilde \PP, \tilde Y),   | \PP_{\tilde Y}(x)| \gtrsim (\log 1/a)^{-O(1)} m.
\end{equation}
Combining \eqref{multfull} and \eqref{typtheta},  we see that $\theta$ is a typical angle of intersection for $(\tilde \PP,  \tilde Y)$.
Equation \eqref{masskept} also implies that each shading enjoys $\lambda(\tilde \PP,  \tilde Y) \gtrapprox \lambda(\PP,Y) \ge a^\eta$ and $|\tilde \PP|\gtrapprox\lambda(\PP,Y)|\PP|\geq a^\eta|\PP|$.

Informally,  we find good refinements either by pigeonholing or by removing parts of our shaded set that are far below average in some respect.  Now we turn to the details.

By pigeonholing, we may find a good refinement $(\PP_2, Y_2)$ and a number $N\geq 1$ so that each set $P_\theta$ contains about $N$ prisms $P\in\PP_2$. 
Next we consider slabs $S \subset B_1$ with dimensions $\theta \times 1 \times 1$.  For a given $\theta \times 1 \times 1$ slab $S$,  define $(\PP_2)_S$ as in \eqref{defPS}. We note that $\PP_2= \sqcup_S \PP_{2,S},$ where the union is over all essentially distinct $\theta \times 1 \times 1$ slabs $S \subset B_1$.   

Because the maximal angle of intersecting planks in $(\PP_2, Y_2)$ is $\lesssim \theta$, each point $x \in U(\PP_2, Y_2)$ lies in $U(\PP_{2,S}, Y_{2,S})$ for $\lesssim 1$ sets $S$.   We know that $| \PP_{2, Y_{2}}(x)| \ge   (\log 1/a)^{-O(1)} m$ and so there must be some $S$ so that $| \PP_{2,S, Y_{2}}(x) | \ge   (\log 1/a)^{-O(1)} m$.   Now we can find a good refinement $(\PP_2,  Y_3)$  so that for each $S$ and each $x \in U(\PP_{2,S},  Y_{3,S})$,  we have $| \PP_{2, S, Y_{3,S}}(x)| \gtrsim  (\log 1/a)^{-O(1)} m$.   

Let $\SSS$ be the set of $S$ so that $\lambda(\PP_{2,S},Y_{3,S}) \gtrapprox \lambda(\PP_2, Y_3) \gtrapprox a^\eta$.   Now define $\PP' = \sqcup_{S \in \SSS} \PP_{2,S}$ and $\PP'_S = \PP_{2,S}$.   Note that $(\PP',  Y_3)$ is a good refinement of $(\PP_2, Y_3)$.  

Now consider a box $Q$ of dimensions $\theta b \times b \times b$.  Define 
\[
\PP'_Q = \{ P \in \PP': |P \cap Q| \approx a b^2 \}.
\]  
If $P \in \PP'_Q$, then $P \cap Q$ is a slab of dimensions $a \times b \times b$.  We want to use our estimates about slabs to study these slabs inside of $Q$.  Let $\SSS_Q = \{ P \cap Q: P \in \PP'_Q \}$.  (Note that the slabs in $\SSS_Q$ need not be essentially distinct.)  

The slabs of $\SSS_Q$ come with a natural shading $Y_{3, Q}$, where $Y_{3, Q}(P \cap Q) = Y_3(P) \cap Q$.   
We will apply Lemma \ref{slab3} to estimate $|U(\SSS_Q, Y_{3,Q})|$.   We know that $\theta$ is a typical angle of intersection for $\PP'_{Y_3}$ in the sense of Definition \ref{deftypicalangleplank},  and this implies that $\theta$ is also a typical angle of intersection for $(\SSS_Q, Y_{3,Q})$ in the sense of Definition \ref{typicalIntersectionAngle}.   To apply our Lemma \ref{slab3} about slabs, we need to estimate $\lambda(Y_{3,Q})$.   Say $Q$ is dense if $\lambda(Y_{3, Q}) \gtrapprox \lambda(Y_{3,S}) \gtrapprox a^\eta$.   If $Q$ is dense, then
Lemma \ref{slab3} gives the key inequality:

\begin{equation} \label{slabvol1}  |U( \PP'_S, Y_{3, S}) \cap Q| \ge |U(\SSS_Q, Y_{3, Q}) | \gtrapprox a^{3 \eta} |Q| . \end{equation}  

Notice that there is an essentially unique slab $S$ of dimensions $\theta \times 1 \times 1$ so that $Q \subset S$ and $\angle (TQ, TS) \lesssim \theta$.  For this slab $S$, $\PP'_Q \subset \PP'_S$.  Moreover, we can tile $S$ with a set $\QQ_S$ of boxes $Q$ so that for each $P \in \PP'_S$,  

\[ P = \sqcup_{Q \in \QQ_S} P \cap Q, \]
where $P \in \PP'_Q$ whenever $P \cap Q$ is non-empty.

Note that a plank $P$ lies in $\PP_Q$ for many boxes $Q$, and these $Q$ form an essentially disjoint cover of $P$.  Let $\QQ_P = \{ Q: P \in \PP'_Q \}$.  We define
\[
Y_4 (P) = \bigcup_{Q \in \QQ_P, Q \textrm{ dense}} Y_3(P) \cap Q.
\]

Next we show that $Y_{4, S}$ is a good refinement of $Y_{3, S}$ for each $S \in \SSS$.   To see this,  we note that

\[ \sum_{P \in \PP'_S} |Y_{3,S}(P)| = \lambda(\PP'_S,  Y_{3,S}) \sum_{P \in \PP'_S} |P|. \]

On the other hand,

\[ \sum_{P \in \PP'_S} |Y_{3,S}(P) \setminus Y_{4,S}(P)| = \sum_{Q \in \QQ_S,  Q \textrm{ not dense}} \sum_{P \in \PP'_Q} |Y_{3,S}(P) \cap Q| \le \]
\[ (1/100) \lambda(\PP'_S, Y_{3,S}) \sum_{Q \in \QQ_S} \sum_{P \in \PP'_Q} |P \cap Q| = (1/100) \lambda(\PP'_S, Y_{3,S}) \sum_{P \in \PP'_S} |P|. \]

If $B_{\theta b} \subset Q$, we say $B_{\theta b}$ is dense if $|U(\PP'_S, Y_{4,S}) \cap B_{\theta b}| \gtrapprox a^{4 \eta}|B_{\theta b}|$.  We let $D(Q) \subset Q$ be the union of the dense balls $B_{\theta b}$.  We can now further refine $Y_4$ so that it is contained in the union of these dense balls: if $P \in \PP'_S$, we define

\[
Y' (P) = \bigcup_{Q \in \QQ_P, Q \textrm{ dense}} Y_{4, S}(P) \cap D(Q).
\]

Note that if $Q$ is dense,  then $|U(\PP',  Y_4) \cap Q \setminus U(\PP', Y') \cap Q| \ll |U(\PP', Y') \cap Q|$.   Since the multiplicity of $(\PP'_S, Y_{4,S})$ is roughly constant,  removing low density balls is a good refinement.  So all together $(\PP',  Y')$ is a $\gtrapprox 1$ refinement of $(\PP, Y)$.  

The definition of $Y'$ in terms of dense balls gives item 1.

 And now for each plank $P_{\theta}\in \PP'_{\theta, S}$, define 
\begin{equation} \label{eq: Ythick}
	Y_{\theta, S}(P_{\theta})=U(\PP'_S, Y')\cap P_{\theta}.
	\end{equation}  

\noindent With this definition,  it follows immediately that $U(\PP'_{\theta, S}, Y_{\theta, S}) = U(\PP'_S, Y')$.   

If $P \subset P_\theta\in\PP'_{\theta,S}$, then 
\[
\frac{ |Y_{\theta, S} (P_{\theta}) | }{|P_{\theta}|} \gtrapprox a^{3 \eta} \frac{ \# \{Q \subset P_{\theta}, Q \textrm{ dense} \} } {\# \{ Q \subset P_{\theta} \} }\ge { a^{3\eta}}\frac{ |Y'(P)| }{ |P| }.
\]
Recall that for each $P_{\theta} \in \PP'_{\theta, S}$, there are $\sim N$ planks $P \in \PP'_S$ with $P \subset P_{\theta}$.  Therefore, we have

\begin{equation} \lambda(\PP'_{\theta, S}, Y_{\theta, S}) \gtrapprox a^{3 \eta} \lambda( \PP'_S, Y') \gtrapprox a^{4 \eta}\end{equation} 

This proves item 2.

We also have

\begin{equation} 
|U(\PP', Y)| \ge |U(\PP', Y')| \gtrsim \sum_{S \in \SSS} |U(\PP'_S, Y'_{ S})| = \sum_{S \in \SSS} |U(\PP'_{\theta, S}, Y_{\theta, S})|. 
\end{equation}

This proves item 3.

We can also use this setup to control $\mu(\PP', Y)$.  Notice that since each point $x$ lies in $U(\PP'_S, Y'_{ S})$ for $\lesssim 1$ slabs $S \in \SSS$, we have
\[
\mu(\PP', Y) \lessapprox \max_{S \in \SSS} \mu(\PP'_S, Y'_{S}).
\]

Next we estimate $\mu(\PP'_S, Y'_{S})$ using that $U(\PP'_S, Y'_{S}) = U(\PP'_{\theta, S}, Y_{\theta, S})$.  We know that 
\[
\mu(\PP'_S, Y'_{S}) = \frac{ \sum_{P\in \PP'_S}  |Y'(P)|}{|U(\PP'_S, Y'_S)|} \lessapprox  \frac{ \sum_{P\in \PP'_S}  |P|}{|U(\PP'_S, Y_S')|}= \frac{ |\PP'_S|  |P|}{|U(\PP'_S, Y'_S)|}=  \frac{ |\PP'_S|  |P|}{|U(\PP'_{\theta, S}, Y_{\theta, S})|} ,
\]  
and 

\[
\mu(\PP'_{\theta,S}, Y_{\theta, S}) = \frac{ \sum_{P_{\theta}\in \PP'_{\theta, S}} |Y_{\theta, S}(P_{\theta})|}{|U(\PP'_{\theta,S}, Y_{\theta,S})|} \gtrapprox a^{4 \eta} \frac{ \sum_{P_{\theta}\in \PP'_{\theta, S}} |P_{\theta}| }{|U(\PP'_{\theta,S}, Y_{\theta,S})|} = a^{4 \eta} \frac{ | \PP'_{\theta, S} | |P_{\theta}| }{|U(\PP'_{\theta,S}, Y_{\theta,S})|}.
\] 
We have
\[
|\PP'_{\theta, S}| \approx \frac{1}{N} |\PP'_S|,
\]
\[
|P_{\theta}| = \frac{ \theta b}{a} |P|.
\]
Plugging in gives 

\[
\mu(\PP, Y) \lessapprox  \mu(\PP'_S, Y'_S)  \lessapprox   a^{-4\eta} \frac{aN }{b \theta} \mu(\PP'_{\theta, S}, Y_{\theta,S}). 
\]

This proves item 4.
\end{proof}

By a linear change of variables, we can convert $\PP_{\theta, S}$ from a set of $\theta b \times b \times 1$ planks in $S$ into a set $\TT$ of $b$-tubes in $B_1$.   To see this more carefully,  recall that $\PP_{\theta, S}$ consists of $\theta b \times b \times 1$ planks $P$ in the slab $S$ so that $TP$ is tangent to $TS$.   The slab $S$ has dimensions $\theta \times 1 \times 1$.   Let $L: S \rightarrow B_1$ be a linear change of variables.   Because of the tangency condition on $P$,  $L(P)$ has dimensions $b \times b \times 1$.  

To summarize,  we have reduced the original incidence problem for $\PP$ to an incidence problem for $\PP_{\theta, S}$,  which is equivalent to an incidence problem for tubes $\TT$.   If we assume $K_{KT}(\beta)$ or $K_F(\beta)$, then we can get bounds for $\mu(\TT)$ and hence bounds for $\mu(\PP)$. We are now ready to prove Lemmas \ref{plankKTUnified} and \ref{plankF}.

\subsection{A Katz-Tao multiplicity estimate}
\begin{proof}[Proof of Lemma \ref{plankKTUnified}]
We may suppose that $\mu(\PP,Y)\geq a^{-\eta}$, since otherwise we are done by Remark \ref{RHSplankKT}. Apply Lemma \ref{lemmaredplanktube}. Abusing notation, we will continue to refer to the resulting refinement by $(\PP,Y)$. By Lemma \ref{lemmaredplanktube}, there is a number $\theta\in[a/b, 1]$, a number $N$, and a slab $S$ of dimensions $\theta \times 1 \times 1$ so that we have

\begin{equation} \label{eqPvsP'S} 
	\mu(\PP, Y) \lessapprox  a^{-4\eta} \frac{aN }{b \theta} \mu(\PP_{\theta, S}, Y_{\theta,S}). \end{equation}

Lemma \ref{lemmaredplanktube} also gives a density lower bound

\[ \lambda( \PP_{\theta, S}, Y_{\theta, S}) \gtrapprox a^{4 \eta}. \]

By a linear change of variables,  we can convert $\PP_{\theta, S}$ to a set of $b$-tubes $\TT$ in $B_1$.   Since convex sets are invariant with respect to linear change of variables,  we have 

\[
\Delta_{max}(\TT) = \Delta_{max}(\PP_{\theta, S}) \le \Delta_{max}(\PP_{\theta}) \le \frac{b \theta}{a N} \Delta_{max}(\PP).
\]

Now apply Lemma \ref{genKKT} (noting Remark \ref{multBoundsDeltaVsRho}) with $b, \epsilon/2$ in place of $\delta, \epsilon$. If $\eta$ and $b_0$ are chosen sufficiently small, then

\begin{align*}  \mu(\PP_{\theta, S}) & = \mu(\TT) \leq a^{-\epsilon/2} \Delta_{max}(\TT)^{1 - \beta/2} \big(b^{-2}|\TT|\big)^{\beta/2} \\
	&\lessapprox  a^{-\epsilon/2}\Delta_{max}(\PP)^{1-\beta/2} \left(  \frac{b \theta}{a N} \right)^{1 - \beta/2} b^{-\beta}( | \PP_{\theta, S}|)^{\beta/2} \\
	& \lesssim  a^{-\epsilon/2} \Delta_{max}(\PP)^{1-\beta/2}   \left(  \frac{b \theta}{a N} \right)^{1 - \beta/2} b^{-\beta}(N^{-1}  | \PP_S|)^{\beta/2}.
	\end{align*}

And so plugging in (\ref{eqPvsP'S}),   we get

\begin{equation}\label{eqnearendplankKT1}
\begin{split}
  \mu(\PP, Y) &  \lessapprox a^{-4\eta}  \frac{aN }{b \theta} \mu(\PP_{\theta, S})\\
	& \leq a^{-\epsilon/2-4\eta} \Delta_{max}(\PP)^{1-\beta/2} \left(  \frac{b \theta}{a N} \right)^{- \beta/2} b^{-\beta}(N^{-1} | \PP_S|)^{\beta/2} \\  
	& \leq a^{-\epsilon/2-4\eta} \Delta_{max}(\PP)^{1-\beta/2} \left( \frac{a}{b \theta} \right)^{\beta/2} b^{-\beta}| \PP_S|^{\beta/2}.
\end{split}
\end{equation}

We know $a \le b \theta$, and $|\PP_S| \leq a^{-\eta} \theta^\gamma |\PP|$.   Plugging these bounds into \eqref{eqnearendplankKT1},  we get

\[
\mu(\PP) \leq a^{-\epsilon}  \Delta_{max}(\PP)^{1-\beta/2}   \left( \frac{a}{b} \right)^{(1+\gamma)\beta/2}\big((ab)^{-1}| \PP|\big)^{\beta/2}.\qedhere
\]
\end{proof}

\subsection{A Frostman multiplicity estimate}

\begin{proof}[Proof of Lemma \ref{plankF}]
We may suppose that $\mu(\PP,Y)\geq a^{-\eta}$, since otherwise we are done by Remark \ref{RHSplankF}. Apply Lemma \ref{lemmaredplanktube}. Abusing notation, we will continue to refer to the resulting refinement by $(\PP,Y)$.

By Lemma \ref{lemmaredplanktube}, there is a number $\theta\in[a/b, 1]$ and a set  $\SSS$ of $\theta \times 1 \times 1$ slabs so that

\[
|U(\PP,Y) | \gtrapprox | \SSS | |U(\PP_{\theta, S}, Y_{\theta, S})|.
\]

Lemma \ref{lemmaredplanktube} also gives a density lower bound

\[ \lambda( \PP_{\theta, S}, Y_{\theta, S}) \gtrapprox a^{4 \eta}. \]

Define $C = C_F(\PP, B_1)$.  This does not imply that $C_F(\PP_{\theta, S},S)\leq C$.  So we have to work out what we know about $\PP_{\theta, S}$.  Let $\PP_{\theta}$ be formed by thickening the planks of $\PP$ to dimensions $\theta b \times b \times 1$.  By Lemma \ref{lemmaredplanktube}, there is a number $N\leq M\theta$ so that each plank $P_{\theta} \in \PP_{\theta}$ contains $\sim N$ planks of $\PP$ (recall that we have replaced $\PP$ by a $\approx 1$ refinement, but the refinement preserves the upper bound $|\PP[P_\theta]|\leq M\theta$).  

By Remark \ref{inheritedDownwardsUpwardsRemark}(A), $C_F(\PP_{\theta},B_1)\lessapprox C_F(\PP, B_1)=C$,  and so 
\[
|\PP_{\theta, S}| \le |\PP_{\theta}[S]| \leq C \theta |\PP_{\theta}|.
\]

 The set $\PP_{\theta, S}$ is not necessarily Frostman, but since $\PP_{\theta}$ is Frostman,  we do know that for any convex $K \subset S$,

\begin{equation} \label{eqdensPthick}
 \Delta(\PP_{\theta, S}, K) \le \Delta(\PP_{\theta}, K) \leq C  \Delta(\PP_{\theta}, B_1) \sim  C  |P_{\theta}| |\PP_{\theta}|.
 \end{equation}
Since $ |\PP_{\theta}| \approx |\SSS| |\PP_{\theta, S}|$, 
\[
\Delta(\PP_{\theta, S}, K) \lessapprox C |\SSS||S| \frac{|P_{\theta}||\PP_{\theta, S}|}{|S|} = C |\SSS||S|\Delta(\PP_{\theta, S}, S). 
\]
Therefore $C_F(\PP_{\theta, S}, S) \lessapprox C |\SSS||S|$. 

Now by a linear change of variables, we can convert $\PP_{\theta, S}$ to a set of (essentially distinct) $b$-tubes $\TT$ in $B_1$.  Let $Y_{\TT}$ be the image of the shading $Y_{\theta, S}$ under this change of variables. Since convex sets are invariant with respect to linear change of variables, $C_F(\TT) \lessapprox C |\SSS||S|$.   Also,  $\lambda(\TT, Y_{\TT}) = \lambda( \PP_{\theta, S}, Y_{\theta, S}) \gtrapprox a^{4 \eta}$.  

If $\eta$ and $b_0$ are chosen sufficiently small depending on $\eps$, then by Lemma~\ref{genKF} (noting Remark \ref{multBoundsDeltaVsRho}) we have
\begin{align*}
|U(\TT)| & \gtrapprox a^{\epsilon/2}   (C|\SSS||S|)^{\beta/2-1} b^{2\beta}  (b^2|\TT|)^{\beta/2}\\
&\approx a^{\epsilon/2}C^{\beta/2-1}  (|\SSS||S|)^{-1}  b^{2\beta} (b^2 |S| |\PP_{\theta}|)^{\beta/2}\\
& =a^{\epsilon/2}C^{\beta/2-1} (|\SSS||S|)^{-1}  b^{2\beta} (b^2\theta  |\PP|/N )^{\beta/2}\\
&\geq a^{\epsilon/2}C^{\beta/2-1} (|\SSS||S|)^{-1}  b^{2\beta} (M^{-1}b^2 |\PP |)^{\beta/2}.
\end{align*} 
Therefore if $b_0>0$ is sufficiently small, we have 
\begin{align*}
|U(\PP)| & \approx |\SSS| |U(\PP_{\theta, S})| \approx  |\SSS|\ |S|\ |U(\TT)| \geq a^{\epsilon} C^{\beta/2-1} b^{2\beta} (M^{-1}b^2 |\PP |)^{\beta/2}.
\end{align*}
This gives the desired lower bound for $|U(\PP)|$. To get an upper bound for $\mu(\PP)$,  we relate $\mu(\PP)$ to $|U(\PP)|$.   This part is just algebra and we leave the details to the reader.
\end{proof}

\subsection{Tubes that factor through planks}
In this section we will prove Proposition \ref{factoringThoughFlatPrisms}. 

\medskip

\noindent {\bf Part (A)}.
Let $C = C_F(\TT)$. After refining $\tubes_\rho$ by a $O(1)$ factor (this induces a harmless refinement on $\TT$), we may suppose that each $T\in\TT$ is contained in exactly one tube $T_\rho\in\tubes_\rho$. By hypothesis, for each $T_\rho\in\tubes_\rho$, there is a set $\WW_{T_\rho}$ of $a\times b\times 1$ planks that factors $\TT[T_\rho]$, i.e. $\TT[T_\rho] = \bigsqcup_{\WW_{T_\rho}}\TT[T_\rho]_W$, where $\TT[T_\rho]_W$ is $\sim 1$ Frostman in $W$. Define $\WW = \bigsqcup_{\TT_\rho} \WW_{T_\rho}$. There is a corresponding decomposition $\TT=\bigsqcup_{\WW}\TT_W$, where $\TT_W$ is $\sim 1$ Frostman in $W$. After pigeonholing and refining $\TT$ and $\WW$, we may suppose that $|Y(T)|$ is approximately the same for each $T\in\TT$, and $|\TT_W|$ is approximately the same for each $W\in\WW$. After this step we have that $C_F(\TT)\lessapprox C$, and $\lambda(\TT,Y)\gtrapprox\delta^{\eta}$.

Apply Proposition \ref{factoringAndMultPropCombined} to the pair $(\TT,Y)$ and $\WW$; we obtain a set $\WW'\subset\WW$, and a shading $Y_{\WW'}$ on $\WW'$ with $\lambda(\WW', Y_{\WW'})\gtrapprox\delta^{2\eta}$. We also obtain a shading $Y'(T)\subset Y(T)$ on $\TT'=\bigcup_{\WW'}\TT_W$, so that $(\TT', Y')$ is an $\approx 1$ refinement of $(\TT,Y)$, and 
\begin{equation}\label{boundMuTTYByFactors}
\mu(\TT,Y)\lessapprox \mu(\WW', Y_{\WW'})\mu(\TT_W, Y')\ \textrm{for all}\ W\in\WW'.
\end{equation}
We have $C_F(\TT_W, W)\lesssim 1$ for each $W\in\WW'$. Since $C_F(\TT')\lessapprox C$ and $|\TT_W|$ is approximately the same for all $W\in\WW'$, by Remark \ref{inheritedDownwardsUpwardsRemark}(A) we have that $C_F(\WW')\lessapprox C$. Thus we will use Lemma \ref{plankF} to bound each of the terms on the RHS of \eqref{boundMuTTYByFactors}. In order to apply Lemma \ref{plankF}, we need to control to what extent the planks in $\WW'$ can concentrate into thicker planks, and similarly for the (rescaled) tubes in $\TT_W$ (this is the quantity ``$M$'' from Lemma \ref{plankF}). 

We begin with $\WW'$. For each $\theta\in[a/b,1]$ we need to bound the maximum number of $a\times b\times 1$ planks from $\WW'$ that are contained in a $\theta b \times b \times 1$ plank. Denote this number by $N(\theta)$. Fix a plank $W\in\WW'$. Recall that $W\in\WW_{T_\rho}$ for some $T_\rho\in\tubes_\rho$. Let $W_\theta$ be the corresponding $\theta b\times b\times 1$ plank. Let $W'\in \WW'[W_\theta]$. Since $W'$ contains at least one tube from $\TT$, and each tube in $\TT$ is contained in exactly one tube from $T_\rho$, we have $T_\rho$ is the unique tube containing $W'$, and thus $W'\in\WW_{T_\rho}$. We conclude that $\WW'[W_\theta] = \WW_{T_\rho}[W_\theta]$. Recall that $\Delta_{\max}(\WW_{T_\rho})\lesssim 1$, and thus $|\WW_{T_\rho}[W_\theta]|\lesssim |W_\theta|/|W| = \frac{\theta b \times b \times 1}{a\times b\times 1} = \theta(b/a)$. Thus $|\WW_{T_\rho}[W_\theta]|\leq M\theta$ for $M = b/a$. Applying Lemma \ref{plankF} (and recalling Remark \ref{multBoundsDeltaVsRho}), we have
\begin{equation} \label{eqmuW}  
\mu(\WW', Y_{\WW'}) \lessapprox  \delta^{-\epsilon/2} C_F(\TT)^{1-\beta/2} (a/b)^{1 - \beta/2} b^{-2 \beta} (b^2 |\WW|)^{1 - \beta/2}.
\end{equation}
We remark that in order to apply Lemma \ref{plankF}, we need to verify that $b$ is sufficiently small ($b\leq\delta^{\eps}$ will suffice). But if this is not the case, then \eqref{multiplicityBoundFactorThroughFlatFrostman} follows from Lemma \ref{slab1}.

Next we consider $\TT_W$. Fix a prism $W\in\WW'$. We first have to make a change of coordinates that maps $W$ to $B_1$ and each $\delta$ tube  $T \in \TT_W$ to a $b^{-1} \delta \times a^{-1} \delta \times 1$ plank.   Next we need to bound the number of $b^{-1} \delta \times a^{-1} \delta \times 1$ planks  in a $\theta a^{-1} \delta \times a^{-1} \delta \times 1$ thick plank. Undoing the coordinate change, this is the number of tubes of $\TT_W$ in a $\delta \times \frac{\theta b}{a} \delta \times 1$ plank.  Since the tubes of $\TT$ are essentially distinct, there are $\lesssim \left( \frac{\theta b}{a} \right)^2\leq (b/a)^2\theta$ such tubes.
Applying Lemma \ref{plankF} and remembering that the dimensions of each plank are $b^{-1} \delta \times a^{-1} \delta \times 1$ (and using our bound $M\leq (b/a)^2$), we get
\begin{equation} \label{eqmuTW}  
\mu(\TT_W, Y') \leq  \delta^{-\epsilon/2} (a/b)^{1 - \beta} (a^{-1} \delta)^{-2 \beta} (a^{-2} \delta^2 | \TT_W| )^{1 - \beta/2}. 
\end{equation}
Combining \eqref{boundMuTTYByFactors}, \eqref{eqmuW}, and \eqref{eqmuTW}, we get \eqref{multiplicityBoundFactorThroughFlatFrostman}.

\medskip

\noindent {\bf Part (B)}.
By hypothesis, there is a set $\WW$ of $a\times b\times 1$ prisms that factors $\TT_\rho$, i.e.~$\TT_{\rho}=\bigsqcup_{\WW}\TT_{\rho,W}$, and $\TT_{\rho,W}$ is $\sim 1$ Frostman in $W$. This induces a decomposition $\TT=\bigsqcup_{\WW}\TT_W$. Apply Proposition \ref{factoringAndMultPropCombined} to $(\TT,Y)$ and $\WW$; while $\TT_W$ might not be Frostman in $W$, we have that $\TT_{\rho,W}$ is Frostman in $W$. By Remark \ref{relaxFrostmanCondition}, this is enough to apply the proposition. We obtain a set $\WW'\subset\WW$, a shading $Y_{\WW'}$ on $\WW'$ with $\lambda(\WW')\gtrapprox\delta^{2\eta}$, and a sub-shading $Y'(T)\subset Y(T)$ on $\TT'=\bigsqcup_{\WW'}\TT_W$, so that $(\TT',Y')$ is an $\approx 1$ refinement of $(\TT,Y)$, and
\begin{equation}\label{boundMuByWandTTW}
\mu(\TT,Y)\lessapprox\mu(\WW',Y_{\WW'})\mu(\TT_W, Y')\ \textrm{for all}\ W\in\WW'.
\end{equation}
We have $\Delta_{\max}(\TT_W)\leq\Delta_{\max}(\TT)$, and $\Delta_{\max}(\WW)\lesssim 1$. We will bound the two terms on the RHS of \eqref{boundMuByWandTTW} using Lemma \ref{plankKTUnified} with $\gamma=0$ and $\gamma=1$, respectively.

We begin with $\WW$. Applying Lemma \ref{plankKTUnified} with $\gamma=0$ (and recalling Remark \ref{multBoundsDeltaVsRho}), we obtain
\begin{equation}\label{boundMuWWKT}
\mu(\WW', Y_{\WW'})\lesssim \delta^{-\epsilon/2}\left(\frac{a}{b}\right)^{\beta/2}\big((ab)^{-1}|\WW|\big)^{\beta/2}.
\end{equation}

Next we consider $\TT_W$. Select a set $W\in\WW$ so that $\lambda(\TT_W, Y')\gtrapprox\delta^{\eta}$. To apply Lemma \ref{plankKTUnified} with $\gamma=1$, we make a change of coordinates that maps $W$ to $B_1$.   This change of coordinates converts $\TT_W$,  a set of  $\delta$-tubes $T \subset W$,  to $\PP$,  a set of $b^{-1} \delta \times a^{-1}\delta \times 1$ planks $P \subset B_1$.   Notice that the eccentricity of the planks is $\frac{ b^{-1} \delta}{a^{-1} \delta} = \frac{a}{b}$.  Let $Y_{\PP}$ be the image of the shading $Y(T)$ after this transformation.
	
Next we need to check that $\PP$ obeys the assumptions of Lemma \ref{plankKTUnified} with $\gamma=1$.   The first assumption of Lemma \ref{plankKTUnified} says that if $\theta \in [a/b, 1]$ and $S$ is a $\theta \times 1 \times 1$ plank,  then  $| \PP_S | \lessapprox (b^{-1} \delta)^{- \eta} \theta |\PP|. $   If we change coordinates back from $B_1$ to $W$,  then $S$ is converted to a convex set $S' \subset W$ and our desired bound becomes 
$| \TT[S'] | \lessapprox (b^{-1} \delta)^{-\eta} \frac{|S'|}{|W|} |\TT_W|$.   Suppose that $S'$ has dimensions $a' \times b' \times c'$ with $a' \le b' \le c'$.   We claim that $a' \gtrapprox a \ge \rho$.   We will come back to this claim below.   Assuming this claim,  we have
	
\[ 
\frac{ | \TT[S'] |}{|\TT_W|} \approx \frac{ |\TT_{\rho} [S'] |}{ |\TT_{\rho,W} |} \lessapprox \frac{ |S'|}{|W|} C_F(\TT_{\rho,W}, W) \lessapprox \frac{ |S'|}{|W|}.
\]
	
\noindent This gives the desired bound on $|\TT[S']|$. 
	
To finish we have to check that $a' \ge a$.     Choose coordinates $x_1, x_2, x_3$ on $W$ so that $W$ is the rectangular box $0 \le x_1 \le a$,  $0 \le x_2 \le b$,  and $0 \le x_3 \le 1$.    The change of coordinates $\phi$ mapping $W$ to $B_1$ is $\phi(x_1, x_2, x_3) = (\frac{x_1}{a}, \frac{x_2}{b}, x_3)$.   
	
We can assume that $\PP_S$ is non-empty,  and so there is at least one $P \in \PP$ with $TP = TS$ (up to angle $\theta$).  	Suppose $P = L(T)$ where $T$ is a tube from $\TT[W]$.  
	
Recall that $TS$ is a 2-plane.    Let $TS' = \phi^* (TS)$,  another 2-plane.     Choose unit vectors $v_1, v_2$ to make a basis for $TS'$ with $v_1$ in the $(x_1,x_2)$-plane and $v_2$ parallel to $T$.   We have $|v_1 \wedge v_2| \sim 1$,  since $T$ is transverse to the $(x_1,x_2)$ plane.   Since $\phi$ maps $TS'$ to $TS = TP$,  we have $| \phi(v_1) \wedge \phi(v_2) | \sim a^{-1}$.   Since $|\phi(v_2)| \sim 1$ we see $|\phi(v_1)| \sim a^{-1}$,  and so $v_1$ is transverse to the $x_2$-direction.   Let $n(S')$ be the vector normal to $TS'$.  Since $v_2$ is almost in the $x_3$ direction and $v_1$ is transverse to the $x_2$-direction,  we have $n(S')$ transverse to the $x_1$ direction.   Now $W$ can be covered by $\theta^{-1}$ parallel copies of $S'$ and so $a' \theta^{-1} \gtrsim b$ and so $a' \ge \theta b \gtrapprox a$. This confirms the first assumption in \ref{plankKTUnified} with $\gamma=1$.

We can now apply Lemma \ref{plankKTUnified} with $\gamma=1$. Recall that $\Delta_{max}(\PP) = \Delta_{max}(\TT_W) \le \Delta_{max}(\TT).$ We have
\begin{equation}\label{boundMuTTW}
\mu(\TT_W,Y') \leq \delta^{-\epsilon/2} \Delta_{\max}(\TT)^{1-\beta/2} \left(\frac{a}{b}\right)^{\beta} \left(\frac{ab}{\delta^2}|\TT_W|\right)^{\beta/2}. 
\end{equation}

Combining \eqref{boundMuByWandTTW}, \eqref{boundMuWWKT}, and \eqref{boundMuTTW}, we obtain \eqref{boundMuABPlanksFactorTT}.


\section{Sticky Kakeya and a sticky / non-sticky dichotomy}
In this section we will discuss two variants of the sticky Kakeya theorem. 
Let $\TT$ be a uniform set of $\delta$-tubes. Let $\rho\in[\delta,1]$ and let $\TT_\rho$ be the thickening of $\TT$ at scale $\rho$. We begin with the following two observations
\begin{itemize}
	\item[(A)] If $\TT$ is Frostman, then by Remark \ref{inheritedDownwardsUpwardsRemark}(A), $\TT_\rho$ is Frostman. In general, however, the tubes in $\TT[T_\rho]$ will not be Frostman inside $T_\rho$.

	\item[(B)] If $\TT$ is Katz-Tao, then by Remark \ref{inheritedDownwardsUpwardsRemark}(B), the tubes in $\TT[T_\rho]$ will be Katz-Tao. In general, however, the tubes in $\TT_\rho$ will not be Katz-Tao. 
\end{itemize}

It will be helpful to consider the special cases where the sets $\TT[T_\rho]$ \emph{are} Frostman in $T_\rho$, or the set $\TT_\rho$ is Katz-Tao.

\begin{defn}\label{FrostmanOrKatzTaoAtEveryScaleDefn}
Let $\TT$ be a set of $\delta$-tubes in $\RR^n$.\\
\medskip
\noindent (A) We say that $\TT$ is \emph{Frostman at every scale with error $C$} if for every scale $\delta\leq \rho\leq 1$, we have 
\[
C_F(\TT[T_\rho], T_\rho)\leq C\quad\textrm{for every}\ T_\rho\in\tubes_\rho.
\]	
\medskip
\noindent (B) We say that $\TT$ is \emph{Katz-Tao at every scale with error $C$} if for every scale $\delta\leq \rho\leq 1$, we have
\[
\Delta_{\max}(\TT_\rho)\leq C.
\]	
\end{defn}
We will usually consider sets of tubes that are uniform, in the sense of Definition \ref{uniformSetOfTubes}, but this is not required in the above definition.
\begin{remark}\label{equivalenceOfDefinitionWithWangZahlAssouad}
Definition \ref{FrostmanOrKatzTaoAtEveryScaleDefn}(A) is closely related to Definition 2.12 from \cite{WZ2}: if $\eps>0$ and $\delta>0$ is sufficiently small depending on $\eps$, and if $\TT$ is a uniform set of $\delta$-tubes that is Frostman at every scale with error $\delta^{-\eps}$, then $\TT$ ``satisfies the Convex Wolff Axioms at every scale with error $\delta^{-\eps}$'' in the sense of \cite{WZ2}, Definition 2.12.
\end{remark}

Recall that a set $\TT$ is sticky if it is both Frostman at every scale and Katz-Tao at every scale. In that case, the sticky Kakeya theorem says that $|U(\TT)|\gtrapprox 1$ and $\mu(\TT)\lessapprox 1$. A mild generalization of this result says that if $\TT$ is Frostman at every scale then $|U(\TT)|\gtrapprox 1$, while if $\TT$ is Katz-Tao at every scale, then $\mu(\TT)\lessapprox 1$. 

\begin{theorem}\label{frostmanOrKTAtEveryScale}
For all $\eps>0$, there exists $\eta,\delta_0>0$ so that the following holds for all $\delta\in(0,\delta_0]$. Let $(\tubes,Y)$ be a uniform set of $\delta$-tubes in $B_1\subset\RR^3$, with $\lambda(\TT,Y)\geq\delta^{\eta}$.

\begin{enumerate}[(A)]
\item If $\TT$ is Frostman at every scale with error $\delta^{-\eta}$, then
\[
|U(\TT,Y)|\geq\delta^\eps.
\]

\item
If $\TT$ is Katz-Tao at every scale with error $\delta^{-\eta}$, then
\[
\mu(\TT,Y)\leq\delta^{-\eps}.
\]
\end{enumerate}
\end{theorem}

Theorem \ref{frostmanOrKTAtEveryScale}(A) is Theorem 5.2 from \cite{WZ2} (note that \cite{WZ2} defines ``Frostman at every scale'' slightly differently, but this is not a problem---see Remark \ref{equivalenceOfDefinitionWithWangZahlAssouad}). To prove Theorem \ref{frostmanOrKTAtEveryScale}(B), we first need a simple multiscale lemma about $\Delta_{max}$.

\begin{lemma} \label{lemmasubmultD} 
Let $\TT$ be a set of $\delta$-tubes in $B_1\subset\RR^n$ and let $\delta = \rho_M \le \rho_{M-1} \le ... \le \rho_1 \le \rho_0 = 1$ be a sequence of scales. Suppose that $\TT$ is uniform at each scale $\rho_k$, in the sense that there is a set $\tubes_{\rho_k}$ of $\rho_k$-tubes for which Items (i), (ii) and (iii) from Definition \ref{uniformSetOfTubes} hold. Then

\[
\Delta_{max}(\TT) \leq C(n)^M \prod_{m=1}^M \Delta_{max}( \TT_{\rho_m} [T_{\rho_{m-1}}] ).
\]

\end{lemma}

\begin{proof}  
Suppose that $K \subset B_1$ is a convex set.   We have to estimate $\Delta(\TT, K) = \frac{ |T| | \TT[K] | }{|K|}$.

Let $K_m$ denote the $\rho_m$-neighborhood of $K$.   For a tube $T \in \TT$,  let $T_{\rho_m}$ be the corresponding thickened tube in $\TT_{\rho_m}$,  and notice that if $T \subset K$ then $T_{\rho_m} \subset K_m$.    So if $T \subset K$,  then $T_{\rho_1} \subset K_1$,  and $T_{\rho_2} \subset T_{\rho_1} \cap K_2$ and in general $T_{\rho_m} \subset T_{\rho_{m-1}} \cap K_m$.   So 

\begin{equation} \label{eqtelprod}
 | \TT [K] | \lessapprox \prod_{m=1}^M | \TT_{\rho_m}[ T_{\rho_{m-1}} \cap K_m] | \leq C(n)^M \prod_{m=1}^M \Delta_{max}( \TT_{\rho_m}[T_{\rho_{m-1}}] ) \frac{ |K_m \cap T_{\rho_{m-1}} | } {|T_{\rho_m}|}.
 \end{equation}

Now

\[ 
\frac{ |K_m \cap T_{\rho_{m-1}} | } {|T_{\rho_m}|} =  \frac{ |K_m \cap T_{\rho_{m-1}} | } {|T_{\rho_{m-1}}|} \cdot  \frac{ | T_{\rho_{m-1}} | } {|T_{\rho_m}|} =  \frac{ |K_m | } {|K_{m-1}|} \cdot  \frac{ | T_{\rho_{m-1}} | } {|T_{\rho_m}|}.
\]

Plugging this into (\ref{eqtelprod}) and simplifying the telescoping product,  we get

\[
| \TT [K] |  \leq C(n)^M  \left[ \prod_{m=1}^M \Delta_{max}( \TT_{\rho_m}[T_{\rho_{m-1}}] ) \right] \frac{ |K| } {|T|}.\qedhere
\]
\end{proof} 

\begin{lemma} \label{lemmasubsticky} 
Let $\eps,\delta>0$ and let $\TT$ be a set of $\delta$-tubes in $B_1\subset\RR^n$. Suppose that $\TT$ is uniform at each of the scales $\delta=\rho_M<\ldots<\rho_0=1$, in the sense that there is a set $\tubes_{\rho_k}$ of $\rho_k$-tubes for which Items (i), (ii) and (iii) from Definition \ref{uniformSetOfTubes} hold. Suppose furthermore that for each $k$ we have

\begin{itemize}

\item $\frac{\rho_{k}}{\rho_{k-1}} \gtrapprox \delta^{\eps}$.

\item $\Delta_{max}(\TT_{\rho_k} [T_{\rho_{k-1}}] ) \lessapprox \left( \frac{ \rho_{k-1} }{\rho_{k}} \right)^\eps$.

\end{itemize}

Then provided that $\delta>0$ is sufficiently small depending on $\eps, M$, and $n$, there is a set of translations $R_j$ of cardinality $\lessapprox \delta^{-\eps} (|T| |\tubes|)^{-1}$ so that $\bigcup_j R_j(\TT)$ is Frostman at every scale with error $\delta^{-(n+3)\eps}$.
\end{lemma}  

To prove Lemma \ref{lemmasubsticky},  we will use random translations $R_j$.    If $v \in \RR^n$,  we let $R_v (x) = x+v$.  
A random translation of size $\le \rho$ is a translation $R_v$ by a vector $v \in B_\rho$ selected randomly from the uniform measure on $B_\rho$.    We will use the following probability lemma,  which is a small variation on Lemma \ref{randCF}.

\begin{lemma} \label{lemrandommotion} 
Let $\delta \le \rho \le 1$.   Suppose that $\TT$ is a set of essentially distinct $\delta$-tubes in $T_\rho\subset\RR^n$ with $\frac{|T_\rho|}{|\TT| |T_\delta|} \sim J \ge 1$.     Let $R_1,  ...,  R_J$ be a set of $J$ random translations of size $\le \rho$.   Let $\TT' = \bigcup_{j=1}^J R_j(\TT)$. Let $A\geq 1$.   Then with probability at least $1-(\rho/\delta)^{O(1)}e^{-A}$,  the tubes of $\TT'$ are essentially distinct up to a factor $\lesssim A$ and we have $\Delta_{max}(\TT') \lesssim A \Delta_{max}(\TT)$.  
After refining $\TT'$ to make the tubes essentially distinct,  we have with high probability $|\TT'| \approx J |\TT|$ and $\Delta_{max}(\TT') \lessapprox \Delta_{max}(\TT)$
\end{lemma}

We will prove Lemma \ref{lemrandommotion} in Appendix \ref{appproblemmas},  where we collect a couple of related probability arguments.   Now we are ready to prove Lemma \ref{lemmasubsticky}.   The proof consists of using Lemma \ref{lemrandommotion} at many scales.

\begin{proof} [Proof of Lemma \ref{lemmasubsticky}]

For each $k=1, \dots, M$,  let $\mathcal{R}_k$ be a set of random translations of size $\sim \rho_{k-1}$,  with  cardinality 

\begin{equation} \label{cardRk} 
|\mathcal{R}_k|\sim \max \Big\{1,\  \frac{ (\rho_k/\rho_{k-1})^{-(n-1)}} {  |\TT_{\rho_{k}}[T_{\rho_{k-1}}] |  } \Big\}.  
\end{equation}

Define the set of translations $\{R_j\}$ as $R_1\circ R_2 \circ \cdots \circ R_{M}$ for $R_k\in \mathcal{R}_k$, $k=1, \dots, M$.    Recall that $\TT' = \bigcup_j R_j(\TT)$.   

Write $d_k=(\rho_k/\rho_{k-1})^{n-1}|\tubes_{\rho_k}[T_{\rho_{k-1}}]|$. Our hypothesis on $\Delta_{max}(\tubes_{\rho_k} [T_{\rho_{k-1}}] )$ gives $d_k\lessapprox (\rho_{k-1}/\rho_k)^\eps$. Since $\tubes$ is uniform, we have $\prod_{k=1}^M d_k\approx |T|\ |\tubes|$. Thus
\[
|\{R_j\}|\leq\prod_{k=1}^M|\mathcal R_k|\lessapprox\frac{\prod_{k=1}^M\max\{1,d_k\}}{\prod_{k=1}^M d_k}\lessapprox\frac{\delta^{-\eps}}{|T|\ |\TT|}.
\]

We will bound $\Delta_{max}(\TT')$ using Lemma \ref{lemmasubmultD}.   To do that,  we need to estimate $\Delta_{max}(\TT'_{\rho_m}[T'_{\rho_{m-1}}])$.   Suppose $T'_{\rho_{m-1}} \in \TT'_{\rho_{m-1}}$.   Then we must have $T'_{\rho_{m-1}} \sim R_1 \circ ... \circ R_{m- 1} (T_{\rho_{m-1}})$ for some $T_{\rho_{m-1}} \in \TT_{\rho_{m-1}}$.   By Lemma \ref{lemrandommotion},  with high probability,  the choice of $R_1, ..., R_{ {m-1}}$ is $\lessapprox 1$.   Therefore with probability at least $3/4$ we have, for each $m=1, \dots, M$,

\begin{equation}\label{compareDeltaMaxes}
 \Delta_{max}(\TT'_{\rho_m} [T'_{\rho_{m-1}}] ) \lessapprox \Delta_{max}\Big( \bigcup_{R_{m}\in  \mathcal{R}_{m}} R_{m}\big( \TT_{\rho_m}[T_{\rho_{m-1}}] \big)\Big).
 \end{equation}

\noindent Fix a tube $T_{\rho_{m-1}}$. The family $\TT_{\rho_m}[T_{\rho_{m-1}}]$ is essentially distinct by Item (ii) of Definition \ref{uniformSetOfTubes}. Applying Lemma \ref{lemrandommotion} with $A = \delta^{-\eps/M}$, we have that with probability $1-\delta^{-O(1)}\exp[-\delta^{-\eps/M}]$, 
\begin{equation}\label{DeltaMaxForOneTRho}
\Delta_{max}\Big( \bigcup_{R_{m}\in  \mathcal{R}_{m}} R_{m}\big( \TT_{\rho_m}[T_{\rho_{m-1}}] \big)\Big) \lesssim  \delta^{-\eps/M} \Delta_{max}\big(\TT_{\rho_m} [T_{\rho_{m-1}}] \big)\lessapprox  \delta^{-\eps/M}\left( \frac{ \rho_{m-1} }{\rho_{m}} \right)^{\eps}.
\end{equation}
There are $(\rho_{m-1}/\rho_m)^{O(1)}\leq \delta^{-O(1)}$ many tubes. Thus if $\delta>0$ is sufficiently small depending on $\eps$, $M$ and $n$, we conclude that with probability at least $3/4$, \eqref{DeltaMaxForOneTRho} holds for every $T_{\rho_{m-1}}\in \TT_{\rho_{m-1}}.$

Combining \eqref{compareDeltaMaxes} and \eqref{DeltaMaxForOneTRho},  we see that for each $m$,  
\[ 
\Delta_{max}(\TT'_{\rho_m} [T'_{\rho_{m-1}}] ) \lessapprox  \delta^{-\eps/M}\left( \frac{ \rho_{m-1} }{\rho_{m}} \right)^{\eps}. 
\]

Now we can apply Lemma \ref{lemmasubmultD} to show that for each $m$, 
\begin{equation} \label{deltamaxT'}  \Delta_{max}( \TT'[T'_{\rho_m}] ) \lessapprox \prod_{m' = m+1}^M \Delta_{max}(\TT'_{\rho_{m'}}[T'_{\rho_{m'-1}}] ) \lessapprox \delta^{-2\eps}.  \end{equation}
On the other hand,  \eqref{cardRk} shows that 

\[ | \TT'_{\rho_m}[T'_{\rho_{m-1}}] | \approx \Big| \bigcup_{R_{m}\in  \mathcal{R}_{m}} R_{m}\big( \TT_{\rho_m}[T_{\rho_{m-1}}] \big) \Big| \gtrsim \left( \frac{\rho_m}{\rho_{m-1}} \right)^{-(n-1)}, \]
which implies that
\begin{equation} \label{cardT'}  |\TT'[T'_{\rho_m}] | \gtrapprox \left( \frac{ \delta}{\rho_m}  \right)^{-(n-1)}.  \end{equation}

Combining \eqref{deltamaxT'} and \eqref{cardT'},  we see that for each $m$,

\[ C_F(\TT'[T'_{\rho_m}], T'_{\rho_m}) \lessapprox  \delta^{-2\eps}.  \]

Finally consider an arbitrary $\rho \in [\delta, 1]$.  We can choose $m$ so that $\rho_m \le \rho \le \rho_{m-1}$.   Since $\rho_m \ge \delta^\eps \rho_{m-1}$,  using the last equation for $m$ and $m-1$ we can conclude that

\[ C_F(\TT'[T'_\rho], T'_\rho) \lessapprox \delta^{-(n+3)\eps} . \qedhere
\]
\end{proof}

\begin{proof}[Proof of Theorem \ref{frostmanOrKTAtEveryScale}]
Recall that Part (A) is Theorem 5.2 from \cite{WZ2}. We now turn to Part (B). Fix $\eps>0$ and let $\eta_1>0$ be the corresponding value from Theorem \ref{frostmanOrKTAtEveryScale}(A), with $\eps/2$ in place of $\eps$. Let $\eta_2=\min\{\eta_1/12,\eps/4\}$. Choose an integer $N\geq 1/\eta_2$, and let $\eta=\eta_2/(2N)$. Let $(\TT,Y)$ be a set of $\delta$-tubes and their corresponding shading that is Katz-Tao at every scale with error $\delta^{-\eta}$, and with $\lambda(\TT,Y)\geq\delta^\eta$. After a harmless refinement, we may suppose that $\mu(\TT,Y)(x)$ is roughly the same for all $x\in U(\TT,Y)$. Thus it suffices to prove that $|U(\TT, Y)|\geq \delta^{\eps}|T|\ |\TT|$. 

Consider the sequence of scales $\rho_k=\delta^{k/N}$, $k=N,\ldots,0$. We have $\rho_{k}/\rho_{k-1} = \delta^{1/N}\geq\delta^{\eta_2}$, and $\Delta_{max}(\TT_{\rho_k}[T_{\rho_{k-1}}])\leq\Delta_{max}(\TT_{\rho_k}) \leq \delta^{-\eta} = (\rho_{k-1}/\rho_{k})^{\eta_2/2}$. Apply Lemma \ref{lemmasubsticky} with $\eta_2$ in place of $\eps$. We obtain a set of $M$ rigid motions $\{R_j\}_{j=1}^M$, with $M\lessapprox \delta^{-\eta_2} (|T| |\tubes|)^{-1}$, so that the set $\TT' = \bigcup_{j=1}^M R_j(\TT)$ is Frostman at every scale with error $\delta^{-6\eta_2}\leq\delta^{-\eta_1/2}$. Define the shading $Y'$ on $\TT'$ in the obvious way: If $T' = R_j(T)$ for some $T\in\TT$ and $1\leq j\leq M$, then $Y'(T') = R_j(Y(T))$. Then $\lambda(\TT',Y') = \lambda(\TT, Y)$, and hence $(\TT', Y')$ satisfies the hypotheses of Theorem \ref{frostmanOrKTAtEveryScale}(A). For part (B), we have 
\begin{align*}
|U(\TT,Y)|& =\frac{1}{M}\sum_{j=1}^M |R_j(U(\TT, Y))| = \frac{1}{M}\sum_{j=1}^M |(U(R_j \TT, R_j Y))|  \\
& \geq \frac{1}{M}|\bigcup_{j=1}^M (U(R_j \TT, R_j Y))| = \frac{1}{M}|U(\TT', Y')| \geq \frac{1}{M}\delta^{\eps/2}\gtrapprox \delta^{\eps/2+\eta_2}|T|\ |\TT|.\qedhere
\end{align*}
\end{proof}

\subsection{A multi-scale decomposition of $\TT$}
In order to apply Theorem \ref{frostmanOrKTAtEveryScale}(A) or Theorem \ref{frostmanOrKTAtEveryScale}(B), our collection of tubes must be Frostman at every scale or Katz-Tao at every scale, respectively. The next result says that either this is true, or else there is a long sequence of scales that are far from being Frostman (resp.~Katz-Tao).

\begin{lemma}\label{dividingScalesLemma}
Let $N\geq 1$ be an integer. Define $\eps = 1/\sqrt N$, and let $\eta=\eta_0\leq \eta_1\leq\ldots\leq\eta_N\leq \eps,$ with $\eta_{j}\leq\eps\eta_{j+1}$ for each index $j$. Let $\delta>0$ and let $\tubes$ be a set of $\delta$-tubes. 

\medskip
\noindent (A) If $C_F(\TT)\leq\delta^{-\eta}$, then after replacing $\tubes$ by a $\gtrapprox 1$ refinement, at least one of the following holds. 
\begin{enumerate}[(i)]
	\item $\tubes$ is $\lessapprox \delta^{-5\eps}$ Frostman at every scale.
	\item There are scales $\delta\leq\tau\leq \theta\leq 1$ with $\tau\leq\delta^{\eps}\theta$, and an integer $1\leq j\leq N$ so that for every $\gtrapprox 1$ refinement $\tubes'$ of $\tubes$, the following holds after replacing $\tubes'$ by a further $\gtrapprox 1$ refinement:
	\begin{itemize}
			\item $C_F(\TT'[T_\tau], T_\tau)\lessapprox (\tau/\delta)^{\eta_{j-1}}$ for all $T_\tau\in\tubes'_\tau$.
			\item $C_F(\TT'_\tau[T_\theta], T_\theta)\lessapprox (\theta/\tau)^{\eta_{j-1}}$ for all $T_\theta\in\tubes'_\theta.$
			\item For all $\rho\in[\tau(\theta/\tau)^\eps,\ \theta(\tau/\theta)^\eps]$, we have $C_F(\TT'_\tau[T_\rho], T_\rho)\geq(\rho/\tau)^{\eta_j}$  for all $T_\rho \in\tubes'_\rho.$
	\end{itemize}
\end{enumerate}

\medskip
\noindent (B)  If $\Delta_{\max}(\TT)\leq\delta^{-\eta}$, then after replacing $\tubes$ by a $\gtrapprox 1$ refinement, at least one of the following holds.
\begin{enumerate}[(i)]
	\item $\tubes$ is $\lessapprox \delta^{-\eps}$ Katz-Tao at every scale.
	\item There are scales $\delta\leq\tau\leq \theta\leq 1$ with $\tau\leq\delta^{\eps}\theta$, and an integer $1\leq j\leq N$ so that for every $\gtrapprox 1$ refinement $\tubes'$ of $\tubes$, the following holds after replacing $\tubes'$ by a further $\gtrapprox 1$ refinement:
	\begin{itemize}
			\item $\Delta_{\max}(\TT'_{\theta})\lessapprox \theta^{-\eta_{j-1}}$.
			\item $\Delta_{\max}(\TT'_\tau[T_\theta])\lessapprox (\theta/\tau)^{\eta_{j-1}}$ for all $T_\theta\in\TT'_\theta$
			\item For all $\rho\in[\tau(\theta/\tau)^\eps,\ \theta(\tau/\theta)^\eps]$, we have $\Delta_{\max}(\TT'_\rho[T_\theta])\geq(\theta/\rho)^{\eta_j}$  for all $T_\theta\in\TT'_\theta$.
	\end{itemize}
\end{enumerate}
\end{lemma}
\begin{remark}\label{interpretationOfGtrapproxRemark}
Note that Item (ii)  (in both parts (A) and (B)) contains nested quantifiers that involve ``$\lessapprox$,'' which in turn is a statement involving nested quantifiers. The precise interpretation is as follows: for every function $f(\delta)\gtrapprox 1$, there is a function $g(\delta)\gtrapprox 1$ so that for every set $\tubes'\subset\tubes$ with $|\tubes'|\geq f(\delta)|\tubes|$, there exists a subset $\tubes''\subset\tubes'$ with $|\tubes''|\geq g(\delta)|\tubes'|$ such that....'' In practice, $f(\delta)$ will be a fixed power of $(\log 1/\delta)^{-1}$.
\end{remark}
\begin{proof}
We begin with part (A). First, we replace $\tubes$ by a refinement with the following property: $\tubes$ is uniform, the Frostman constant $C_F(\TT[T_\rho], T_\rho)$ is the same (up to a $\lessapprox 1$ factor) for all scales $\rho$ and all $T_\rho\in\tubes_{\rho}$, and furthermore, while this property might be destroyed if $\tubes$ is replaced by a $\gtrapprox 1$ refinement $\tubes'$, the property can be restored by replacing $\tubes'$ by a further $\gtrapprox 1$ refinement (c.f.~Remark \ref{interpretationOfGtrapproxRemark}). To obtain such a refinement, first replace $\tubes$ by a uniform set of $\tubes$, then apply Lemma \ref{lemmafactmax} (maximal density factoring) to each set $\tubes[T_\rho]$, $T_\rho\in\tubes_\rho$, for all $\rho$ of the form $\rho = \delta^{k/M}$, $k=0,\ldots,M$, $M = \lceil \log\log 1/\delta\rceil$, and then finally replace the surviving set by a uniform set of tubes. Remarks \ref{makeLessapproxExplicit} and \ref{maximalDensityFactoringPreservesFrostmanRefinement} ensure that the surviving set has the claimed properties.

Next, we check if there is a scale $\rho \in [\delta^{1 - \eps}, \delta^{\eps}]$ so that $C_F(\TT[T_\rho], T_\rho) \leq (\rho/\delta)^{\eta_1}$. If there is no such scale we stop.  If there is such a scale,  then we define $\TT_1$ to be $\TT[T_\rho]$ modified by a change of variables that takes $T_\rho$ to $B_1$,  and we set $\TT_2 = \TT_\rho$. We get the following output with $J=2$:
	
	\begin{enumerate}
		
		\item For each $j = 1, ..., J$,  $\TT_j$ is a set of $\delta_j$-tubes in $B_1$.
		
		\item $\prod_{j=1}^J \delta_j \approx \delta$.
		
		\item $\prod_{j=1}^J | \TT_j | \approx | \TT |$.
		
		\item $C_F(\TT_j) \leq  \delta_j^{-\eta_{J-1}}$.

\end{enumerate}
	
\noindent{\bf General step.} 
	Suppose that for some $J$ we have $\TT_j$ obeying the criteria listed just above.   We check whether there is a $j \in \{1, ..., J \}$ and a scale $\rho$ so that the following holds:
	
	\begin{itemize}
		
		\item $\delta_j < \delta^{\eps}$ and 
		
		\item $\rho \in [\delta_j^{1 - \eps}, \delta_j^{\eps}]$,  and
		
		\item $C_F(\TT_j[T_\rho], T_\rho) \leq (\delta_j/\rho)^{- \eta_J}$.
		
	\end{itemize}
	
	If there is no such $j$ and $\rho$,  we stop.   If there is such a $j$,  $\rho$,  then we replace $\TT_j$ by $\TT_{j, \rho}$ and $\TT_j[T_\rho]$ (rescaled to be a set of tubes in the unit ball).    After this replacement,  we can relabel our sets of tubes to give $\TT_1,  ...,  \TT_{J+1}$,  which obey all the criteria listed above. 
	
	Since each $\delta_j  < \delta^{\eps^2}$ and since $\prod_{j=1}^J \delta_j \approx \delta$,  we have $J \le N = 1/\eps^2$,  and therefore this process terminates after at most $N$ steps, with some sets of tubes $\TT_1,  ..., \TT_J$, with $J \le N$.    Each of these sets obeys the criteria listed above (1 to 4) and also for each $j$,  either $\delta_j \ge \delta^{\eps}$,  or else
	\begin{equation} \label{nostickyintermed}
		\textrm{For each } \rho \in [\delta_j^{1 - \eps}, \delta_j^{\eps}],\ \textrm{we have}\  C_F(\TT_j[T_\rho], T_\rho) \geq  (\delta_j/\rho)^{-\eta_J}.
	\end{equation}

Unwinding the stopping time argument,  there is a sequence of $\delta = \rho_J \le ... \le \rho_1 \le \rho_0 = 1$ so that $\TT_j$ corresponds to $\TT_{\rho_j}[T_{\rho_{j-1}}]$.   By our stopping criteria,  we see that 
\[
C_F(\TT_{\rho_j}[T_{\rho_{j-1}} ],  T_{\rho_{j-1}}) \leq (\rho_{j-1} / \rho_j)^{\eta_J} \le (\rho_{j-1} / \rho_j)^{\eps}.
\] 
By Lemma \ref{lemmasubmultD}, we have that $C_F(\TT[T_{\rho_j}], T_{\rho_j})\lesssim \delta^{-\eps}$ for each index $j$. 

First we consider the case where every $\delta_j$ obeys $\delta_j \ge \delta^{\eps}$. Then each ratio $\rho_j / \rho_{j-1}$ is at least $\delta^{\eps}$. We claim that $\TT$ is $\lessapprox \delta^{-5\eps}$ Frostman at every scale, and hence Conclusion (i) holds. To verify this, let $\rho\in[\delta,1]$ and choose $j$ so that $\rho_{j+1} \le \rho \le \rho_j$.   Then
\[
C_F(\TT[T_\rho], T_\rho)
\leq C_F(\TT[T_{\rho_j}], T_{\rho_j}) \frac{|\TT[T_{\rho_j}]|}{|\TT[T_\rho]|}
 \leq C_F(\TT[T_{\rho_j}], T_{\rho_j}) \frac{|\TT[T_{\rho_j}]|}{|\TT[T_{\rho_{j+1}}]|}
\lesssim \delta^{-\eps}  \big(\frac{\rho_j}{\rho_{j+1}}\big)^{4} 
\leq \delta^{-5\eps}.
\]

Next we consider the case where there exists at least one index $j$ with $\delta_j < \delta^{\eps}$ and $\TT_j$ obeys (\ref{nostickyintermed}).  Then Conclusion (ii) holds with $\tau=\rho_{j}$ and $\theta=\rho_{j-1}$.

\medskip

The proof of Part (B) is nearly identical, and we leave the details for the reader.
\end{proof}

\section{Main Lemma 1} \label{secmainlemma1}

In this section, we sketch the proof of Main Lemma \ref{lemmain1}.  We first recall the statement.  

\begin{mlem*}  $K_{KT}(\beta)$ implies $K_F(\beta)$.  
\end{mlem*}

Let's first digest what the lemma is saying.   Suppose that $\TT$ is a set of essentially distinct $\delta$-tubes in $B_1$ with $C_F(\TT) \leq \delta^{-\eta}$.   We have to prove the bound

\begin{equation} \label{ml1goal}
 \mu( \TT) \leq \delta^{-\epsilon} ( \delta^{-2})^\beta \left( \delta^2 | \TT | \right)^{1 - \beta/2}. 
\end{equation}

\noindent Since $C_F(\TT, B_1) \leq \delta^{-\eta} $,  we have $|\TT| \geq  \delta^{-2+\eta}$.   If $|\TT| \approx \delta^{-2}$,  then $C_F(\TT, B_1) \leq \delta^{-\eta}$ is equivalent to $\Delta_{max}(\TT) \lessapprox \delta^{-\eta}$,  and then $K_{KT}(\beta)$ gives $\mu(\TT) \leq \delta^{-\epsilon} \delta^{-2 \beta}$,  which is the desired bound.   So the case $| \TT | \approx \delta^{-2}$ is trivial.   So the content of the lemma is in the case when $ | \TT | \gg \delta^{-2}$.

If $| \TT | \gg \delta^{-2}$,  then $\Delta_{max}(\TT) \gg 1$, so we can not immediately apply $K_{KT}(\beta)$ to $\TT$.    If $\tilde \TT \subset \TT$ is a random subset with $| \tilde \TT | \sim \delta^{-2}$,  then it's not hard to check that $\Delta_{max}(\tilde \TT) \lessapprox 1$.   We can apply $K_{KT}(\beta)$ to $\tilde \TT$,  and we get $\mu(\tilde \TT) \lessapprox \delta^{-2 \beta}$ which is equivalent to $|U(\tilde \TT)| \gtrapprox \delta^{2 \beta}$.   We have $|U(\TT)| \ge |U(\tilde \TT)|$,  and this translates into

\begin{equation} \label{ml1triv}
 \mu( \TT) \lessapprox ( \delta^{-2})^\beta \left( \delta^2 | \TT | \right)
\end{equation}

Comparing (\ref{ml1triv}) with our goal (\ref{ml1goal}),  we see that we need to reduce the exponent on the last factor from 1 to $1 - \beta/2$.   In the proof of Kakeya,  any exponent strictly less than 1 will work.   We can also describe the lemma in terms of $|U(\TT)|$,  which may be more intuitive.   Since we assume $K_{KT}(\beta)$,  we already know that if $C_F(\TT) \lessapprox 1$ and $| \TT | \sim \delta^{-2}$,  then $|U(\TT)| \gtrapprox \delta^{2 \beta}$.   Now we need to show that if $C_F(\TT) \lessapprox 1$ and if $ | \TT | $ is far larger than $\delta^{-2}$,  then $|U(\TT)|$ is a little larger than $\delta^{2 \beta}$.   While this may sound intuitive,  it is not at all trivial to prove, and the proof will depend crucially on sticky Kakeya.

We organize the proof as a bootstrapping or self-improving argument.

\begin{lemma} \label{lemmain1boot} 
Suppose that $K_{KT}(\beta)$ holds and that $K_F(\gamma)$ holds for some $\gamma > \beta$.  Then $K_F(\gamma - \nu)$ holds for $\nu = \nu(\gamma, \beta) > 0$. For $\beta$ fixed, $\nu = \nu(\gamma, \beta) > 0$ is monotone in $\gamma$.
\end{lemma}

Since $K_F(1)$ holds trivially, Lemma \ref{lemmain1boot} implies Main Lemma \ref{lemmain1}.

\begin{proof}[Proof of Lemma~\ref{lemmain1boot}]
Let $\eps_0$ and $\delta_0$ be the output of Theorem \ref{frostmanOrKTAtEveryScale} with $\gamma/2$ in place of $\eps$. Let $N \geq 25/\eps_0^2$ be an integer to be chosen later, and define $\eps = 1/\sqrt N$. Let $\eta\leq\eta_1\leq\ldots\leq\eta_N\leq \eps/5$ be a sequence of numbers to be chosen later.   We choose the constants from biggest to smallest and each $\eta_j$ will  be selected very small depending on $\eta_{j+1}$ and $\beta$.

 Apply Lemma \ref{dividingScalesLemma}(A) to $\TT$, with $N$ as above. If Conclusion (i) holds, then $\TT$ is $\delta^{-\eps_0}$ Frostman at every scale, and hence by Theorem \ref{frostmanOrKTAtEveryScale}(A) we have
\[
 |U(\TT,Y)|\geq\delta^{\gamma/2} \geq \delta^{2 (\gamma/2)} \left(|\TT|\ |T| \right)^{(\gamma/2)/2},
\]
where we used essential distinctness of the tubes in $\TT$ to obtain $|\TT|\leq\delta^{-4}$. In particular, $|U(\TT,Y)|$ satisfies the volume estimate corresponding to $K_F(\gamma/2)$. In this case, Lemma \ref{lemmain1boot} holds with  $\nu = \gamma/2$.
\begin{remark}\label{criticalUseExponentBetaOver2}
Recall that the statement $K_F(\beta)$ asserts the estimate $|U(\TT)| \gtrapprox \delta^{2 \beta} \left( |\TT|\ |T| \right)^{\omega},$ with $\omega = \beta/2$. In the argument described above, we must have $\omega<\beta$.
\end{remark}

Henceforth we shall assume that Conclusion (ii) of Lemma \ref{dividingScalesLemma}(A) holds. Let $\delta\leq\tau\leq\theta\leq 1$ and $1\leq j\leq N$ be the output of that lemma.   We apply Lemma \ref{shadingMultiplicityEstimateForRhoTubes} to $(\TT,Y)$ and $\TT_\tau$.   We obtain a $\approx 1$ refinement of $(\TT,Y)$, a subset of $\TT_\tau$, and a shading $Y_{\TT_\tau}$ on $\TT_\tau$. Abusing notation, we will continue to refer to these refinements as $(\TT,Y)$ and $\TT_\tau$.   Apply Lemma \ref{shadingMultiplicityEstimateForRhoTubes} to $(\TT_\tau,Y_{\TT_\tau})$ and $\TT_\theta.$ We obtain a $\approx 1$ refinement of $(\TT_\tau,Y_{\TT_\tau})$, a subset of $\TT_\theta$, and a shading $Y_{\TT_\theta}$ on $\TT_\theta$. Abusing notation, we will continue to refer to these refinements as $(\TT_\tau,Y_{\TT_\tau})$ and $\TT_\theta$. After a further refinement, we may suppose that all sets of the form $(\TT[T_\tau], Y)$, $(\TT_\tau[T_\theta], Y_{\TT_\tau})$ and $(\TT_\theta, Y_{\TT_\theta})$ are uniform, and have the Frostman constants claimed in Conclusion (ii) of Lemma \ref{dividingScalesLemma}(A). Lemma \ref{shadingMultiplicityEstimateForRhoTubes} implies that for all $T_\tau \in \TT_\tau$ and $T_\theta \in \TT_\theta$, 
\begin{equation}\label{boundMuTTYByTripleProduct}
\mu(\TT,Y)\lessapprox \mu(\TT[T_\tau], Y)\ \mu(\TT_\tau[T_\theta], Y_{\TT_\tau})\ \mu(\TT_\theta, Y_{\TT_\theta}).
\end{equation}

Our goal is to estimate the three quantities on the RHS of \eqref{boundMuTTYByTripleProduct}.
For each $T_\tau\in\TT_\tau$ we have $C_F(\TT[T_\tau], T_\tau)\leq \delta^{-\eta_{j-1}}$. Thus by applying Lemma \ref{genKF} (with $\eta_1$ in place of $\eps$) to the re-scaled, essentially distinct sets $\TT[T_\tau]$, we see that for each $T_\tau\in\TT_\tau$, we have
\begin{equation}
\begin{split}
\mu(\TT[T_\tau],Y) 
&\leq (\delta/\tau)^{-\eta_1}  \delta^{-(1-\gamma/2)\eta_{j-1}}((\delta/\tau)^{-2})^{\gamma} ((\delta/\tau)^2|\TT[T_\tau]|)^{1-\gamma/2}\\
&\leq \delta^{-2\eta_{j-1}}((\delta/\tau)^{-2})^{\gamma} ((\delta/\tau)^2|\TT[T_\tau]|)^{1-\gamma/2}.
\end{split}
\end{equation}

By Remark \ref{inheritedDownwardsUpwardsRemark}(A), $C_F(\TT_\theta)\lessapprox\delta^{-\eta}$. Thus by applying $K_F(\gamma)$  (again, with $\eta_1$ in place of $\eps$) we have
\begin{equation}
\mu(\TT_\theta,Y_{\TT_\theta}) 
\leq \delta^{-\eta_{j-1}}  \left( (\theta)^{-2} \right)^{\gamma} (\theta^2|\TT_\theta|)^{1-\gamma/2}.
\end{equation}

Thus in order to prove Lemma \ref{lemmain1boot} with $\nu = \eta_1$, it suffices to show that for each $T_\theta\in\TT_\theta$, we have
\begin{equation}\label{multTTauInsideTTheta}
\mu(\TT_\tau[T_\theta],Y_{\TT_\tau})\leq \delta^{10\eta_{j-1}}   ((\tau/\theta)^{-2})^{\gamma} ((\tau/\theta)^2|\TT_\tau[T_\theta]|)^{1-\gamma/2}.
\end{equation}
We will show that this holds, provided each $\eta_{j-1}$ is selected sufficiently small compared to $\eta_j$.

Fix a tube $T_\theta\in\TT_\theta$, and let $(\tilde\tubes,\tilde Y)$ be the image of the tubes in $\TT_\tau[T_\theta]$ (and their shadings) under the rescaling taking $T_\theta$ to a tube of thickness 1. Hence $\tilde\TT$ is a family of $\tilde\delta$-tubes, where $\tilde\delta=\tau/\theta$. Recall that $\tilde\delta\leq\delta^{\eps}$, and thus to prove \eqref{multTTauInsideTTheta} we must show that
\begin{equation}\label{multTildeT}
\mu(\tilde\TT,\ \tilde Y)\leq {\tilde\delta}^{10\eta_{j-1}/\eps}   \tilde\delta^{-2\gamma} (\tilde\delta^2|\tilde\TT|)^{1-\gamma/2}.
\end{equation}
Recall that $C_F(\tilde\TT)\leq\tilde\delta^{-\eta_{j-1}/\eps}$, and 
\begin{equation}\label{tildeDeltaLargeFrostman}
C_F(\tilde\TT[T_\rho], T_\rho)\geq(\rho/\tilde\delta)^{\eta_j}\quad\textrm{for all}\ T_\rho\in\tilde\TT_\rho,\ \rho\in[\tilde\delta^{1-\eps},\ \tilde\delta^\eps].
\end{equation}
Select $\rho = \tilde\delta^{\eps}$. Apply Lemma \ref{lemmafactmax} to each set $\tilde\TT[T_\rho]$; after pigeonholing, we may suppose that the planks in the resulting set $\WW_{T_\rho}$ have roughly the same dimensions $a\times b\times 1$ for all $T_\rho\in\tilde\TT_\rho$. 

Applying Proposition \ref{factoringThoughFlatPrisms}(A) with $\eta_{j-1}$ in place of $\eps$ and $\gamma$ in place of $\beta$, we have
\[
\mu(\tilde\TT,\tilde Y)\leq \tilde\delta^{-2\eta_{j-1}/\eps}\big(\frac{a}{b}\big)^{3\gamma/2} ( \tilde\delta^{-2})^\gamma \left( \tilde\delta^2 | \tilde\TT | \right)^{1 - \gamma/2}.
\]
In particular, either  \eqref{multTildeT} holds (and hence we are done), or else $a$ is not too much smaller than $b$, i.e.
\[
a  \geq \tilde\delta^{\frac{10 \eta_{j-1}}{\eps\gamma}} b.
\]

 We have 
 
 \[
 \Delta(\tilde\TT, T_b) \ge (a/b) \Delta(\tilde\TT, W) \approx (a/b) \Delta_{max}(\tilde\TT).
 \]
Therefore,  
 
 \begin{equation}\label{upperBdCfTildeTTTB}
 C_F(\tilde\TT[T_b], T_b) \lessapprox b/a\leq \tilde\delta^{-\frac{10 \eta_{j-1}}{\eps\gamma}}.
 \end{equation}
Since $b\leq\rho= \tilde\delta^{\eps}$, if we select $\eta_{j-1}$ sufficiently small depending on $\eta_{j}$, and $\gamma$ (and depending on $\eps$, which in turn depends on $\gamma)$, then we must have $b\leq \tilde\delta^{1 - \eps}$---if not, then \eqref{upperBdCfTildeTTTB} would violate \eqref{tildeDeltaLargeFrostman} (with $b$ in place of $\rho$). 

Define $\eta_{j-1}' = \frac{10 \eta_{j-1}}{\eps\gamma}$, so 

\[
C_F(\tilde\TT[T_b], T_b)\lessapprox \tilde\delta^{-\eta_{j-1}'}.
\]
Apply Lemma \ref{shadingMultiplicityEstimateForRhoTubes} to $(\tilde\TT,\tilde Y)$ and $\tilde\TT_b$ (abusing notation, we will continue to refer to the associated refinements of $(\tilde\TT,\tilde Y)$ and $\TT_b$, and after a refinement we can assume that these sets are still uniform). Applying Lemma \ref{genKF} to bound $\mu(\tilde\TT[T_b], \tilde Y)$, we conclude that in order to establish \eqref{multTildeT}, it suffices to prove that
\begin{equation}\label{multTildeTb}
\mu(\tilde\TT_b,\  Y_{\tilde\TT_b})\leq {\tilde\delta}^{10\eta_{j-1}'}   b^{-2\gamma} (b^2|\tilde\TT_b|)^{1-\gamma/2}.
\end{equation}

Since each tube $T_b \in \tilde \TT_b[T_\rho]$ contains at least one $W \in \WW_{T_\rho}$,  we have

\[ \Delta_{\max} (\tilde \TT_b[T_\rho] ) \lesssim \frac{|T_b|}{|W|} \Delta_{max}(\WW_{T_\rho}) \lesssim \frac{b}{a} \lesssim \tilde \delta^{- \eta_{j-1}'}. \]

Therefore
\[
\Delta_{\max}(\tilde\TT_b) \lesssim (\rho)^{-4}  \Delta_{\max}(\tilde\TT_b[T_\rho])\lesssim \tilde\delta^{-4\eps-\eta_{j-1}'}.
\]
Applying Lemma \ref{genKKT}, we have
\[
\mu(\tilde\TT_b, Y_{\tilde\TT_b})  \leq \tilde\delta^{-4\eps-\eta_{j-1}'}   \big(b^{-2}|\tilde\TT_b|\big)^{\beta/2}=\Big(\tilde\delta^{-4\eps-\eta_{j-1}'}b^{2(\gamma-\beta)}(b^2|\tilde\TT_b|)^{(\beta+\gamma)/2-1} \Big) b^{-2\gamma}(b^2|\tilde\TT_b|)^{1-\gamma/2}.
\]
We wish to bound the first bracketed term.   The key point is that $(b^2 | \tilde \TT_b|)$ is close to 1.   To see this, we note that
\[
\tilde\delta^{\eta_{j-1}'}\lessapprox C_F(\tilde\TT_b)^{-1} \leq b^2|\tilde\TT_b| \lesssim \Delta_{\max}(\tilde\TT_b)\lesssim \tilde\delta^{-4\eps-\eta_{j-1}'},
\]
Now we do not know the sign of $(\beta+\gamma)/2-1$,  but using the above upper and lower bounds for $b^2 |\tilde\TT_b|$ we get
\[
\mu(\tilde\TT_b, Y_{\tilde\TT_b})  \leq 
\tilde\delta^{-8\eps-2\eta_{j-1}'+2(\gamma-\beta)}b^{-2\gamma}( b^2 | \tilde \TT_b |)^{1 - \gamma/2}.
\]
Provided we select $N$ sufficiently large (recall that $\eps=1/\sqrt N$) and $\eta_{N}$ sufficiently small (recall $\eta_{j-1}\leq \eta_N)$ so that $ -8\eps-2\eta_{j-1}'+2(\gamma-\beta)\geq (\gamma-\beta)$, we have 
\[
\mu(\tilde\TT_b, Y_{\tilde\TT_b}) \lessapprox \tilde\delta^{(\gamma-\beta)}b^{-2\gamma}( b^2 | \tilde \TT_b |)^{1 - \gamma/2}.
\]
This implies our desired bound \eqref{multTildeTb}, provided we choose $\eta_{j-1}$ sufficiently small so that $10\eta_{j-1}'\leq (\gamma-\beta)/2$.
\end{proof}


\section{Main Lemma 2} \label{secmainlemma2}

Our last goal is to prove Main Lemma \ref{lemmain2}.  We recall the statement.  

\begin{mlem*} If $\beta>0$ and $K_{KT}(\beta)$ and $K_{F}(\beta)$ hold,  then $K_{KT}(\beta - \nu)$ holds, where $\nu = \nu(\beta) > 0$ is monotone.  
\end{mlem*}

\subsection{Reduction to the strongly non-sticky case}

The sticky Kakeya theorem plays a crucial role in this proof.   By using sticky Kakeya, we are able to reduce the problem to a situation which is not sticky in a strong sense. 

\begin{lemma} \label{lemmain2vns} (Very not sticky case of Main Lemma \ref{lemmain2}) For all $\beta > 0$,  there is some $\exscalb = \exscalb (\beta) > 0$ so that for all $\zeta > 0$, there is some $\nu = \nu(\beta,  \zeta) > 0$ and $\eta = \eta(\beta, \zeta) > 0$ so that the following is true:

If $K_{KT}(\beta)$ and $K_{F}(\beta)$ hold and if $(\TT,Y)$ is a uniform set of $\delta$-tubes in $B_1$ obeying

\begin{itemize}

\item $\Delta_{max}(\TT) \leq \delta^{-\eta}$ and $\lambda(\TT,Y)\geq\delta^\eta$.

\item For each $\rho \in [\delta^{1 - \exscalb}, \delta^{\exscalb}]$,  $| \TT_\rho | \geq  \rho^{-2 -\zeta}$. 

\end{itemize}

Then 
\[
\mu(\TT,Y) \leq \delta^{\nu}  \big(\delta^{-2}|\TT|\big)^{\beta/2}.
\]

\end{lemma}

Our goal in this section is to show that Lemma \ref{lemmain2vns} implies Main Lemma \ref{lemmain2}. The proof is similar to the proof of Main Lemma \ref{lemmain1},  using Lemma \ref{dividingScalesLemma}(B), Proposition \ref{factoringThoughFlatPrisms}(B), and Theorem \ref{frostmanOrKTAtEveryScale}(B) instead of their Part (A) counterparts.  

\begin{proof}[Proof of Main Lemma~\ref{lemmain2} using Lemma~\ref{lemmain2vns}]
		Fix $\eps>0$ and $\nu=\nu(\beta)>0$. Let $\eta=\eta(\epsilon, \beta)>0$ and $\delta_0=\delta_0(\epsilon,\beta)$ be quantities to be chosen below. Let $\TT$ be a set of $\delta$-tubes in $B_1 \subset \RR^3$ with $\Delta_{max}(\TT) \leq  \delta^{-\eta}$ and $\lambda(\TT,Y)\geq\delta^{\eta}$. Our goal is to show that if $\eta$ and $\delta_0$ are chosen appropriately, then 	
	\begin{equation} \label{mugoalml2quant}
		\mu( \TT, Y) \leq   \delta^{-\epsilon} \big(\delta^{-2}|\TT|\big)^{(\beta-\nu)/2}.
	\end{equation}
	We can assume without loss of generality that $(\TT,Y)$ is uniform. Let $\eps_1$ and $\delta_1$ be the output of Theorem \ref{frostmanOrKTAtEveryScale}(B) with $\beta/2$ in place of $\eps$. We will choose $\delta_0$ with $\delta_0\leq\delta_1$. Let $\exscalb=\exscalb(\beta)$ be the quantity from Lemma \ref{lemmain2vns} with $\beta$ as above. Let $\eps_2=\eps_2(\eps_1, \exscalb,\beta)$ be a number to be chosen later (since $\eps_1$ and $\exscalb$ depend only on $\beta$, $\eps_2$ depends only on $\beta$). Define $N = \lceil 25/\eps_2^2\rceil$, and let $0<\eta_0\leq\eta_1\leq\ldots\leq\eta_N\leq N^{-1/2}$ be a sequence of numbers to be chosen later, depending only on $\beta$. Each $\eta_j$, $0\leq j<N$, will be selected very small depending on $\eta_{j+1},\eps_2,$ and $\beta$, with $\eta_j\leq N^{-1/2}\eta_{j+1}$. Set $\nu=\min\{\beta/2,\eta_0/10\}$. We will eventually select $\eta$ sufficiently small, depending on $\eps$ and these parameters.

 Apply Lemma \ref{dividingScalesLemma}(B) to $\TT$, with $\eta_0$ in place of $\eta$. If Conclusion (i) holds, then (provided we select $\eps_2\leq\eps_1/5$) we have that $\TT$ is $\delta^{-\eps_1}$ Katz-Tao at every scale, and hence by Theorem \ref{frostmanOrKTAtEveryScale}(B) we have
\[
\mu(\TT,Y)\leq \delta^{-\beta/2}\leq \big(\delta^{-2}|\TT|\big)^{(\beta-\nu)/2},
\]
where we used $|\TT|\geq 1$ and $\nu\leq\beta/2$. Thus \eqref{mugoalml2quant} is satisfied.

Henceforth we shall assume that Conclusion (ii) of Lemma \ref{dividingScalesLemma}(B) holds. Let $\delta\leq\tau\leq\theta\leq 1$ and $1\leq j\leq N$ be the output of that lemma.   We apply Lemma \ref{shadingMultiplicityEstimateForRhoTubes} to $(\TT,Y)$ at scale $\tau$ and then apply it again to $(\TT_\tau,  Y_{\TT_\tau})$ at scale $\theta$.   We obtain some $\approx 1$ refinements of these sets of tubes, but we will abuse notation and continue to use the same names.   Lemma \ref{shadingMultiplicityEstimateForRhoTubes} gives that for all $T_\tau \in \TT_{\tau},  T_\theta \in \TT_\theta$,  we have
\begin{equation}\label{boundMuTTYByTripleProductMainLem2}
\mu(\TT,Y)\lessapprox \mu(\TT[T_\tau], Y)\ \mu(\TT_\tau[T_\theta], Y_{\TT_\tau})\ \mu(\TT_\theta, Y_{\TT_\theta}),
\end{equation}
and
\begin{equation} \label{lambdabounds}
\lambda(\TT_\theta,  Y_{\TT_\theta}), \lambda(\TT_\tau,  Y_{\TT_\tau}) \gtrapprox \lambda(\TT, Y) \ge \delta^\eta.
\end{equation}

Our goal is to estimate the three quantities on the RHS of \eqref{boundMuTTYByTripleProductMainLem2}.

By Remark \ref{inheritedDownwardsUpwardsRemark}(B), for each $T_\tau\in\TT_\tau$ we have $\Delta_{max}(\TT[T_\tau])\leq \delta^{-\eta}$.   Also,  we can choose $T_\tau$ so that $\lambda(\TT[T_\tau], Y) \gtrapprox \lambda(\TT, Y) \gtrapprox \delta^\eta$.   Thus by applying $K_{KT}(\beta)$ (with $\eta_0$ in place of $\eps$), for this $T_\tau$, we have
\[
\mu(\TT[T_\tau],Y) \leq \delta^{-\eta_0}  \Big((\delta/\tau)^{-2}|\TT[T_\tau]|\Big)^{\beta/2}.
\]
Since $\Delta_{max}(\TT_\theta)\leq\delta^{-\eta_{j-1}}$, by applying Lemma \ref{genKKT} (again, with $\eta_0$ in place of $\eps$), we conclude that 
\[
\mu(\TT_\theta,Y_{\TT_\theta})\lessapprox \delta^{-2\eta_{j-1}} \big(\theta^{-2}|\TT_\theta|\big)^{\beta/2}.
\]
Thus in order to establish \eqref{mugoalml2quant} with $\nu$ chosen above, it suffices to show that for each $T_\theta\in\TT_\theta$, we have
\begin{equation}\label{multTTauInsideTThetaLem2}
\mu(\TT_\tau[T_\theta],Y_{\TT_\tau})\leq \delta^{10\eta_{j-1}} \Big((\tau/\theta)^{-2}|\TT_\tau[T_\theta]|\Big)^{\beta/2}.
\end{equation}
Indeed, $\eta_0\leq\eta_{j-1}$ gives a net gain of $\delta^{7\eta_0}$ in \eqref{boundMuTTYByTripleProductMainLem2}, while $|\TT|\lesssim\delta^{-2-\eta}$  allows the factor $(\delta^{-2}|\TT|)^{\nu/2}$ to be absorbed into this gain.
We will show that this holds, provided each $\eta_{j-1}$ is selected sufficiently small compared to $\eta_j$.

Fix a tube $T_\theta\in\TT_\theta$, and let $(\tilde\tubes,\tilde Y)$ be the image of the tubes in $\TT_\tau[T_\theta]$ (and their shadings) under the rescaling taking $T_\theta$ to a tube of thickness 1. Hence $\tilde\TT$ is a family of $\tilde\delta$-tubes, where $\tilde\delta=\tau/\theta$. Recall that $\tilde\delta\leq\delta^{\eps_2}$, and thus to prove \eqref{multTTauInsideTThetaLem2} we must show that
\begin{equation}\label{multTildeTLem2}
\mu(\tilde\TT,\ \tilde Y)\leq {\tilde\delta}^{10\eta_{j-1}/\eps_2} \big(\tilde\delta^{-2}|\tilde\TT|\big)^{\beta/2}.
\end{equation}
Recall that 
\begin{equation}\label{upperBdDeltaMaxTildeTT}
\Delta_{\max}(\tilde\TT)\leq\tilde\delta^{-\eta_{j-1}/\eps_2},
\end{equation}
 and 
\begin{equation}\label{tildeDeltaLargeDeltamax}
\Delta_{\max}(\tilde\TT_\rho )\geq \rho^{-\eta_j}\quad\textrm{for all}\ \rho\in[\tilde\delta^{1-\eps_2}, \tilde\delta^{\eps_2}].
\end{equation}

We set $\rho = \tilde \delta^{1 - \eps_2}$.  
Apply Lemma \ref{lemmafactmax} to $\tilde\TT_\rho$ and let $\VV$ be the resulting set of planks. Lemma \ref{lemmafactmax} replaces the set $\tilde\TT_\rho$ with an $\approx 1$ refinement. This in turn induces a $\approx 1$ refinement on $(\tilde\TT, \tilde Y)$. Abusing notation, we will continue to refer to these new objects as $\tilde\TT_\rho$ and $\tilde Y$.   Suppose the planks of $\VV$ have dimensions $a \times b \times 1$.   We have two cases, depending on these dimensions.

\medskip

\noindent {\bf The eccentric case}.\\

We pick another small constant $\eta_{j-1}'$ with $\eta_{j-1} \ll \eta_{j-1}' \ll \eta_j$.   In fact,  we pick it by $\eta_{j-1}' = \frac{12\eta_{j-1}}{\eps_2\beta}.$  We are in the eccentric case if

\begin{equation} \label{eqecccase}  b \geq \tilde \delta^{-\eta_{j-1}'}a.  \end{equation}

Applying Proposition \ref{factoringThoughFlatPrisms}(B) with $\eta_{j-1}$ in place of $\eps$ and recalling that $a/b\leq \tilde \delta^{\frac{12\eta_{j-1}}{\eps_2\beta}}$, we have
\[
\mu(\tilde\TT,\tilde Y)\leq \tilde\delta^{-2\eta_{j-1}/\eps_2} \big(\frac{a}{b}\big)^{3\beta/2}\big(\tilde\delta^{-2}|\tilde\TT|\big)^{\beta/2}
\leq \tilde\delta^{10\eta_{j-1}/\eps_2} \big(\tilde\delta^{-2}|\tilde\TT|\big)^{\beta/2}.
\]
Thus \eqref{multTildeTLem2} holds, and we are done. 

Henceforth we shall assume that  $b <  \tilde\delta^{-\eta_{j-1}'} a.$

\medskip

\noindent {\bf The non-eccentric case}.\\

Recall that $\rho = \tilde\delta^{1-\eps_2}$,  $\VV$ is the maximal density factoring of $\tilde\TT_\rho,$ and  each $V \in \VV$ has dimensions $a\times b\times 1$, with 

\[ b\leq\tilde\delta^{-\eta_{j-1}'} a. \]

We first apply Lemma \ref{shadingMultiplicityEstimateForRhoTubes} to $(\tilde\TT,\tilde Y)$ and $\tilde \TT_b$, and, we obtain a shading $\tilde Y_{\tilde\TT_b}$ on $\tilde\TT_b$ so that

\[
\mu(\tilde\TT,\tilde Y)\lessapprox \mu(\tilde\TT_b, \tilde Y_{\tilde\TT_b})\mu(\tilde \TT'[T_b], \tilde Y'). 
\]

Here $(\tilde \TT', \tilde Y')$ is a $\gtrapprox 1$ refinement of $(\tilde \TT,  \tilde Y)$.    By further pigeonholing,  we can replace $\VV$ by $\VV' \subset \VV$ with $|\VV'| \gtrapprox |\VV|$ and so that for each $V \in \VV'$,  

\begin{itemize}

\item  $|\tilde \TT'[V]| \gtrapprox |\tilde \TT[V]|$ 

\item $\lambda(\tilde \TT'[V],  \tilde Y') \gtrapprox \lambda(\tilde \TT[V],  \tilde Y)$.  

\end{itemize}

At this point,  to ease the notation,  we write $(\tilde \TT,  \tilde Y)$ for $(\tilde \TT',  \tilde Y')$ and $\VV$ for $\VV'$.  

Applying Lemma \ref{genKKT} with $\eta_{j-1}$ in place of $\eps$ (and recalling our bound on $\Delta_{\max}(\TT_b)$ from \eqref{upperBdDeltaMaxTb}), we have
\begin{equation}\label{boundOnMuTTb}
\mu(\tilde\TT_b, \tilde Y_{\tilde\TT_b})\leq  \tilde\delta^{-2\eta_{j-1}'} \big(b^{-2}|\tilde\TT_b|\big)^{\beta/2}.
\end{equation}

The main part of the proof is to estimate $\mu(\tilde \TT[T_b], \tilde Y)$.  
First we estimate
\begin{equation}\label{upperBdDeltaMaxTb}
\Delta_{\max}(\tilde\TT_b)\leq \frac{|T_b|}{|V|}\Delta_{\max}(\VV)\leq \tilde\delta^{-\eta_{j-1}'}.
\end{equation}
We will select $\eta_{j-1}$ small enough (depending on $\eta_j,\ \eps_2$, and $\beta$) so that $\eta_{j-1}'\leq \eta_j \eps_2$. Comparing \eqref{tildeDeltaLargeDeltamax} and \eqref{upperBdDeltaMaxTb}, we obtain a contradiction unless $b\geq \tilde\delta^{\eps_2}$. We will suppose henceforth that $b\geq \tilde\delta^{\eps_2}$.

Next we estimate
\[
\Delta(\tilde\TT_\rho, T_b) \geq \frac{|V|}{|T_b|} \Delta(\tilde\TT_\rho, V) \gtrapprox \tilde\delta^{\eta_{j-1}'}\Delta_{\max}(\tilde\TT_\rho),
\]
i.e. $C_F(\tilde\TT_\rho[T_b], T_b)\lesssim \tilde\delta^{-\eta_{j-1}'}$, and thus by Remark \ref{inheritedDownwardsUpwardsRemark}(A), we have
\begin{equation}\label{goodFrostmanBoundTTRhoInsideTb}
C_F(\tilde\TT_\sigma[T_b], T_b)\lesssim \tilde\delta^{-\eta_{j-1}'}\quad\textrm{for all}\ \sigma\in[\rho, b].
\end{equation}

On the other hand,  using \eqref{tildeDeltaLargeDeltamax} and 
applying Lemma \ref{lemmasubmultD} at the scales $\sigma \in [\tilde \delta^{1 - \eps_2}, \tilde \delta^{\eps_2}]$, we have that 
\[
\sigma^{-\eta_j}\leq 
\Delta_{\max}(\tilde\TT_\sigma)
\lesssim \Delta_{\max}(\tilde\TT_\sigma[T_b])\Delta_{\max}(\tilde\TT_b)
\leq \Delta_{\max}(\tilde\TT_\sigma[T_b])\tilde\delta^{-\eta_{j-1}'},
\]
i.e.
\begin{equation}\label{lowerBdOnDeltaMaxTTSigmaTb}
\Delta_{\max}(\tilde\TT_\sigma[T_b])\gtrsim \tilde\delta^{\eta_{j-1}'} \sigma^{-\eta_j}.
\end{equation}
Comparing \eqref{goodFrostmanBoundTTRhoInsideTb} and \eqref{lowerBdOnDeltaMaxTTSigmaTb}, we conclude that
\begin{equation}\label{TTSigmaBigCardinalityV1}
|\tilde\TT_\sigma[T_b]| \gtrsim \tilde\delta^{2\eta_{j-1}'} b^2\sigma^{-2-\eta_j}\quad\textrm{for all}\ \sigma\in[\tilde \delta^{1 - \eps_2}, \tilde \delta^{\eps_2}].
\end{equation}

We will bound $\mu(\tilde\TT[T_b], \tilde Y)$ using Lemma \ref{lemmain2vns}, with $\zeta = \eta_j/2$.   In order to apply Lemma \ref{lemmain2vns},  we need to first rescale $T_b$ to $B_1$, converting $\tilde\TT[T_b]$ to a new set of tubes $\TT'$ of radius $\delta' = \tilde \delta/b \le \tilde \delta^{1 - \eps_2}$.   We have to check that for every $\sigma' \in [(\delta')^{1 - \exscalb}, (\delta')^{\exscalb}]$ we have $| \TT'_{\sigma'} | \ge (\sigma')^{-2 - \zeta}$.   We choose $\eps_2 \le \exscalb/2$.    Then this estimate follows from \eqref{TTSigmaBigCardinalityV1}.


Now let $\eta'$ be the value of $\eta$ from Lemma \ref{lemmain2vns}, with $\beta$ as above and $\zeta = \eta_j/2$. Recall that we have a bound on $\Delta_{\max}(\tilde\TT[T_b])$ from \eqref{upperBdDeltaMaxTildeTT}. To ensure that the hypotheses of Lemma \ref{lemmain2vns} are satisfied, we select $\eta_{j-1}$ sufficiently small compared to $\eta_j,$ $\beta$, and $\eps_2$ so that $\Delta_{max}(\TT') = \Delta_{\max}(\tilde\TT[T_b])\leq\tilde\delta^{-\eta'}$.   Lemma \ref{lemmain2vns} also requires that the shading $\tilde Y$ be sufficiently dense,  and this follows if we select $\eta$ sufficiently small.   Applying Lemma \ref{lemmain2vns} to $\TT'$ and then changing coordinates back to $\tilde \TT[T_b]$,  we conclude that there exists $\nu(\beta,\eta_j)>0$ so that
\begin{equation}\label{boundOnMuTildeTTb}
\mu(\tilde\TT[T_b], \tilde Y)\leq \tilde\delta^{\nu(\beta,\eta_j)}  \Big((\tilde\delta/b)^{-2}|\tilde\TT[T_b]|\Big)^{\beta/2}.
\end{equation}

Multiplying \eqref{boundOnMuTildeTTb} and \eqref{boundOnMuTTb}, we have
\[
\mu(\tilde\TT, \tilde Y)\lessapprox \tilde\delta^{\nu(\beta,\eta_j)-2\eta_{j-1}'} \big(\tilde\delta^{-2}|\tilde \TT|\big)^{\beta/2}.
\]
We will choose $\eta_{j-1}$ sufficiently small (depending on $\beta$ and $\eta_j$) so that $\nu(\beta,\eta_j)-2\eta_{j-1}'\geq 10\eta_{j-1}/\eps_2$. This establishes \eqref{multTildeTLem2}.
\end{proof}

\subsection{Maximal density factoring revisited} \label{subsection: factorrevisit}

In Lemma \ref{lemmafactmax},  we chose sets $W$ to maximize $\Delta(\VV,  W)$. For some families $\VV$ of convex sets (for example if $\VV$ is both Katz-Tao and Frostman), there might be many equally valid candidates for $W$. To prove Lemma \ref{lemmain2vns}, it is important to choose carefully in the case of a tie or a near tie.   When the maximum density is achieved by two convex sets with very different sizes,  we choose $W$ to be the convex set of smaller volume.   More precisely,  we adjust our maximization process as follows.   We let $\exfact > 0$ be a tiny exponent to be chosen later.   Then,  instead of choosing $W$ to maximize $\Delta(\VV, W)$,  we choose $W$ to maximize

\[
|W|^{- \exfact} \Delta(\VV, W).
\]

This leads to a variation of Lemma \ref{lemmafactmax} with two small changes:

\begin{lemma} \label{lemmafactmaxbias} (Biased maximal density factoring lemma) Let $\VV$ be a finite family of convex subsets of $\RR^n$ contained in a convex set $U$,  and suppose $\exfact > 0$.  Then there is a set $\VV'\subset\VV$ with $\sum_{\VV'}|V|\gtrapprox \sum_{\VV}|V|$; a family $\WW$ of convex sets in $\RR^n$, each of approximately the same dimensions; and a partition
\begin{equation}\label{VVPartitionbias}
\VV'=\bigsqcup_{W\in\WW}\VV_W,
\end{equation}
where each set in $\VV_W$ is contained in $W$. This partition has the following properties:

 If $K \subset W \in \WW$, 
\begin{equation} \label{factmaxmod1} 
\Delta(\VV_W, K) \lessapprox \left( \frac{ |K|}{|W|} \right)^{\exfact} \Delta(\VV_W, W).  \end{equation}
In particular,  $C_F(\VV_W,  W) \lesssim 1$.

For each $W \in \WW$,  $\Delta(\VV_W, W) \gtrsim \left( \frac{ |W|}{ |U| } \right)^{\exfact} \Delta_{max}(\VV')$. 

\begin{equation} \label{factmaxmod2} \Delta_{max}(\WW)  \lesssim \left( \frac{ |W|}{ |U| } \right)^{-\exfact}. \end{equation}

\end{lemma}

\noindent The proof of Lemma \ref{lemmafactmaxbias} is the same as Lemma \ref{lemmafactmax},  so we omit the details.

We will call the decomposition $\VV=\bigsqcup_{\WW}\VV_W$ a \emph{maximal density factoring with bias $\exfact$}. Observe that we can still apply Proposition \ref{factoringAndMultPropCombined} to the factoring $\bigsqcup_{\WW}\VV_W$---Proposition \ref{factoringAndMultPropCombined} requires that $\VV_W$ be $C$-Frostman in $W$, and this remains the case with $C\sim 1$.

\subsection{Setup for the proof of Lemma~\ref{lemmain2vns}} \label{subsecproofoverview}

To prove Lemma \ref{lemmain2vns}, we will study how $\TT$ behaves in smaller balls at various scales. Our goal is to prove that 

\begin{equation} \label{eqgoalmuT}  \mu(\TT,Y) \lessapprox \delta^\nu \big(|T|^{-1}| \TT|\big)^{\beta/2},  \end{equation}

\noindent where $\nu = \nu(\beta, \zeta) >0$.   This goal can be written equivalently as

\begin{equation} \label{eqgoalUT}
 |U(\TT,Y)| \gtrapprox \delta^{-\nu } |T|\ |\TT|\big(|T|^{-1}|\TT|\big)^{-\beta/2},
 \end{equation}

 \noindent where  $\nu = \nu(\beta,  \zeta) > 0$.

The argument has several cases.   In some cases,  we will prove \eqref{eqgoalmuT} and in some cases we will prove \eqref{eqgoalUT}.   
In some of the cases,  we will prove \eqref{eqgoalUT} by showing that $U(\TT,Y)$ fills a larger than expected fraction of a typical $a$-ball that intersects $U(\TT,Y)$, and we will pair this with an estimate on $U(\TT_a)$.    

We have to do some pigeonholing to prepare for our argument.
First we refine $(\TT, Y)$ so that for each $x \in U(\TT, Y)$,  the set $\TT_Y(x)$ is uniform.   We identify $\TT_Y(x)$ with a subset of $S^2$ corresponding to the directions of the tubes.   Then for each $x$,  we can choose a uniform subset of $U_x \subset \TT_Y(x)$ with $|U_x| \gtrapprox |\TT_Y(x)|$.   Then we define $Y'$ so that $x \in Y'(T)$ if and only if $T \in U_x$.   This is a $\gtrapprox 1$ refinement.   Because of the uniformity,  there is a function $\mu(\rho)$ so that for any $T_0 \in \TT$ and $x \in Y'(T_0)$,

\begin{equation} \label{defmurho}
\# \{ T \in \TT_{Y'}(x) : \textrm{angle}(T, T_0) \le \rho \} \approx \mu(\rho)
\end{equation}

In particular,  we have $\mu(\TT, Y) \lessapprox \mu(\TT, Y') \approx \mu(1)$.   Moreover,  for each $T_\rho \in \TT_{\rho}$,  we have $\mu(\TT[T_\rho], Y') \approx \mu(\rho)$. 
Abusing notation,  we will refer to $Y'$ as $Y$ from here on.  

Next we apply Lemma \ref{shadingMultiplicityEstimateForRhoTubes} to $(\TT,Y)$ and $\TT_a${, for $a= \delta^{\eta j}, j=1, \dots, \eta^{-1}$}. We obtain a $\approx 1$ refinement of $(\TT,Y)$  (abusing notation, we will continue to refer to this as $(\TT,Y)$ ) so that for each $x\in U(\TT_a, Y_{\TT_a})$, we have

\begin{equation}\label{lowerBoundTTScaleAAndABall}
|U(\TT, Y)|\gtrapprox |U(\TT_a, Y_{\TT_a})|\frac{|U(\TT, Y)\cap B(x,a)|}{|B(x,a)|}.
\end{equation}	

Recall that $\lambda(\TT_a, Y_{\TT_a})\gtrapprox \lambda(\TT,Y)\geq\delta^{\eta}$. We also have 
\[
\Delta_{max}(\TT_a) \leq \delta^{-\eta} \frac{ (a/\delta)^2 }{ | \TT[T_a] |}.
\]
Applying Lemma \ref{genKKT} with $\nu/10$ in place of $\eps$, we have (provided $\eta>0$ is chosen sufficiently small)
\begin{equation}\label{volumeOfTTa}
|U(\TT_a, Y_{\TT_a})| \gtrsim \delta^{\nu/9}\Big(|\TT_a|\ |T_a|\Big)\ \big/\ \Big( \Big(\frac{ |T_a|/|T| }{ | \TT[T_a] |}\Big)^{1 - \beta/2}  \big(|T_a|^{-1}| \TT_a |\big)^{\beta/2}\Big)
=\delta^{\nu/9}\big(\frac{a}{\delta}\big)^{2\beta}|T|\ |\TT|\big(|T|^{-1}|\TT|\big)^{-\beta/2}.
\end{equation}
Chaining \eqref{lowerBoundTTScaleAAndABall} and \eqref{volumeOfTTa}, we conclude that in order to obtain our desired estimate \eqref{eqgoalUT}, it suffices to prove that there exists an $a$-ball $B_a$ with
\begin{equation} \label{eqgoaldens}
  \frac{ |U(\TT) \cap B_a |}{|B_a|} \gtrapprox  \delta^{- 2\nu} \left( \frac{\delta}{a} \right)^{2 \beta}.
\end{equation}

With these preparations,  we are ready to set up the main objects in the proof.  Let $r_1 = \delta^{\exscalb}$.    Divide $U(\TT_{r_1}, Y_{\TT_{r_1}})$ into boundedly overlapping balls $B$ of radius $r_1$. 

For each ball $B$, we define $\TT_B$ to be a set of essentially distinct $\delta\times\delta\times r_1$-tubes contained in $B$ of the form $T\cap B$ with $T\in\TT$.   For each $T_B \in \TT_B$,  we let $\TT(T_B)$ be the set of tubes $T \in \TT$ so that $T \cap B$ is comparable to $T_B$.   


Next we define a shading on $\TT_B$.   For each $T_B \in \TT_B$,  we consider the multiplicity $|\TT(T_B)_Y(x)|$.   We choose $Y_B(T_B) \subset T_B$ by pigeonholing the multiplicity $|\TT(T_B)_Y(x)|$.   Finally we refine $Y$ so that if $T \in \TT(T_B)$,  $Y(T) \cap T_B \subset Y_B(T_B)$.   This is a $\gtrapprox 1$ refinement.  Next we refine $(\TT_B, Y_B)$ by pigeonholing $|Y_B(T_B)|$.   Again we refine $Y$ so that $Y(T) \cap T_B \subset Y_B(T_B)$ and again this is a $\gtrapprox 1$ refinement.

We let $\BBB$ be the set of balls $B$ with $\lambda(\TT_B,  Y_B) \gtrapprox \lambda(\TT, Y)$.   Then we refine $Y$ by removing the parts of $Y(T)$ that are not in the balls of $\BBB$.   Again this is a $\gtrapprox 1$ refinement.  

In the sequel,  we will use the following properties of the shading $(\TT_B,  Y_B)$:

\begin{itemize}

\item $\lambda(\TT_B,  Y_B) \gtrapprox \delta^\eta$ for each $B \in \BBB$.

\item $|Y_B(T_B)|$ is roughly the same for all $T_B \in \TT_B$.

\item If $Y(T) \cap B$ is non-empty,  then $T \in \TT(T_B)$ for some $T_B \in \TT_B$ and $Y(T) \cap T_B \subset Y_B(T_B)$.

\item The multiplicity $|\TT(T_B)_Y(x)|$ is roughly constant for $x \in Y_B(T_B)$; this constant is the same for all $T_B$.  

\end{itemize}

Now we apply the biased maximal density factoring, Lemma \ref{lemmafactmaxbias}, to $\TT_B$,  and we let $\WW_B$ be the resulting family of convex sets (the maximal density factoring replaces $\TT_B$ by a large subset, but we will abuse notation and continue to refer to this set as $\TT_B$; this is harmless because of all the uniformity in the bullet points above).   By pigeonholing,  we can suppose that for each $B$,  the convex sets $W \in \WW_B$ have dimensions $a \times b \times r_1$,  with $a \le b \le r_1$. Thus for each ball $B$ we obtain a decomposition $\TT_B = \bigsqcup_{\WW_B}\TT_{B, W}$.
  
 The way we proceed depends on the dimensions of $W$.   We let $\tau > 0$ be a new small parameter,  to be chosen later.
 If $a \ge \delta^{1 - \tau}$,  then we are in the \emph{thick case}. In this case we will show that $U(\TT, Y)$ is large by proving \eqref{eqgoaldens}.    If $a \le \delta^{1 - \tau}$,  then we are in the \emph{thin case}. This case divides into several sub-cases. In the most important sub-case, we will estimate $\mu(\WW_B)$ by an induction on scales.   We deal with these cases in the next two sections.

\subsection{The thick case}\label{thickCaseSection}
In this section,  we prove Lemma \ref{lemmain2vns} in the thick case. Recall that we applied the maximal density factoring lemma with bias $\exfact$. Thus by \eqref{factmaxmod1}, for any $W\in\WW_B$ and any convex set $K \subset B$,  

\begin{equation} \label{factmaxmodbias}
\Delta(\TT_{B,W},  W) / \Delta(\TT_{B,W},  K) \gtrapprox \left( \frac{|W|}{|K|} \right)^{\exfact}.
\end{equation}

In particular,  taking $K = T_B$,  we have 
\[
\Delta(\TT_{B,W},  W) \gtrapprox \left(\frac{|W|}{|T_B|}\right)^{\exfact}  \Delta(\TT_B,  T_B) \approx \left(\frac{|W|}{|T_B|}\right)^{\exfact},
\]  
and so

\begin{equation} \label{TBWbig}
 | \TT_{B,W} | \gtrapprox \left( \frac{ |W| }{|T_B|} \right)^{1+\exfact}. 
 \end{equation}

In the thick case,  we will prove \eqref{eqgoaldens} for balls of radius $a$:

\[
  \frac{ |U(\TT) \cap B_a |}{|B_a|} \gtrapprox  \delta^{- 2\nu} \left( \frac{\delta}{a} \right)^{2 \beta}.
\]

   This inequality says that $U(\TT,Y)$ fills a surprisingly large fraction of a typical $a$-ball,  and it implies our goal as explained in Subsection \ref{subsecproofoverview}.  For the remainder of Section \ref{thickCaseSection} we will fix an $r_1$-ball $B$ for which $\lambda(\TT_B, Y_B)\gtrapprox\delta^{\eta}$. Since each $W\in\WW_B$ has thickness $a$, in order to establish \eqref{eqgoaldens} it suffices to prove that there exists at least one $W\in\WW_B$ for which 
\begin{equation}\label{denseInsideW}
  \frac{ |U(\TT_B, Y_B) \cap W |}{|W|} \gtrapprox  \delta^{- 2\nu} \left( \frac{\delta}{a} \right)^{2 \beta}.
\end{equation}
We will fix a choice of $W$ for which $\lambda(\TT_{B,W}, Y_B)\gtrapprox \delta^{\eta}.$

\subsubsection{Warmup: A simple case}
We begin with a special case to illustrate the main ideas of the proof. The simplest scenario in the thick case is when $W$ is an $r_1$-ball (i.e.~$W=B$), and hence $a = r_1$.   The strong estimate (\ref{factmaxmod1}) implies that $C_F(\TT_B,  B) \lessapprox  1$,  and so we can apply  $K_F(\beta)$ to $\TT_B$.   Strictly speaking,  to do the application we rescale $B = B(x,r)$ to the unit ball $B_1$,  which sends $\TT_B$ to a set $\tilde \TT_B$ of $\delta/a$-tubes, and transforms the shading $Y_{\TT_B}$ accordingly.   Then we get

\[
  \frac{ |U(\TT_B, Y_B) \cap W |}{|W|}  = \frac{  |U(\tilde \TT_B,  Y_{\tilde \TT_B})| }{|B_1|} \sim | U(\tilde \TT_B, Y_{\tilde\TT_B}) | \geq \delta^{\epsilon}(\delta/a)^{2 \beta} \left( |\tilde T_B| |\tilde \TT_B| \right)^{\beta/2},
\]
where $\epsilon > 0$ is a parameter we are free to choose (provided $\eta$ is sufficiently small depending on $\eps$). We will choose $\eps = \tau\exfact\beta/2$. 
By (\ref{TBWbig}), we know that

\begin{equation}\label{boundOnTTB}
| \TT_B | = |\TT_{B,W}| \gtrapprox \left( \frac{ |B| }{|T_B|} \right)^{1+\exfact} = \left(\frac{ |B_1| }{|\tilde T_B|}\right)^{1+\exfact},
\end{equation}

and so the factor in parentheses in the preceding volume estimate is $\gtrapprox  (\delta/a)^{-2\exfact}$.   All together we have

\[
\frac{  |U(\TT_B, Y_{B})| }{|B|} \gtrapprox  \frac{  |U(\TT_B)| }{|B|} \gtrapprox   \delta^{\eps} \Big(\frac{\delta}{a}\Big)^{2 \beta  -\exfact \beta}.
\]

Since $a \ge \delta^{1 - \tau}$,  this gives

\[
\frac{  |U(\TT_B, Y_{B})| }{|B|} \gtrapprox  \delta^{\eps} \delta^{- \tau \exfact \beta} \Big(\frac{\delta}{a}\Big)^{2 \beta}.
\]

This gives \eqref{denseInsideW} if we choose $\nu < \tau \exfact\beta/8 $ since $r_1 =\delta^{\exscalb} > \delta^{1/2}$.

\subsubsection{The thick case, general version}

In the general version of the thick case,  $W$ has dimensions $a \times b \times r_1$,  with $\delta^{1-\tau} \le a \le b \le r_1$.

We need to prove \eqref{denseInsideW}. To do so,  we make a linear change of variables to convert $W$ to $B_1$. This converts $\TT_{B,W}$ to a set of planks $\PP$ in $B_1$. Let $Y_{\PP}$ denote the image of the shading $Y_B$ under this transformation. We have  $\frac{ |U(\TT_{B,W}, Y_B)| }{|W|} \sim |U(\PP, Y_{\PP})|$, so it suffices to prove

\[
|U(\PP, Y_{\PP})| \gtrapprox  \delta^{-2\nu}  \left( \frac{\delta}{a} \right)^{2 \beta}.
\]

If $W$ has dimensions $a \times b \times r_1$,  then a plank $P \in \PP$ has dimensions $\frac{\delta}{b} \times \frac{\delta}{a} \times 1$.  
Note that $C_F(\PP, B_1) \sim C_F(\TT_{B,W}, W) \lessapprox 1$ by \eqref{factmaxmod1}, and $\lambda(\PP, Y_{\PP})\sim\lambda(\TT_{B,W}, Y_B)\gtrapprox \delta^{\eta}$.

Now we apply Lemma \ref{plankF}. Before doing so, we must estimate the quantity ``$M$'' from that lemma. 
For any $\theta\in[a/b, 1]$, let $P\in\PP$, and let $P_{\theta}$ be the corresponding $\theta \frac{\delta}{a} \times \frac{\delta}{a} \times 1$ thickened plank, as described in Definition \ref{thickenedPlanks}.   From the statement of Lemma \ref{plankF} we can choose any $M$ with $M \ge 1$ and $M\theta \ge |\PP[P_\theta]|$ for each $\theta \in [a/b, 1]$.   So we need to estimate $\PP[P_\theta]$. 

Recall that $\PP$ is related to $\TT_{B,W}$ by a linear change of variables.   When we undo that change of variables,  the planks $\PP$ become $\TT_{B,W}$ and each plank $P$ becomes a $\delta \times \delta \times r_1$ tube.  The thickened plank $P_\theta$ becomes a rectangular solid $T_{\theta}$ with dimensions $\delta \times b' \times r_1$ with $\delta \le b' \le r_1$ (the value of $b'$ depends on the orientation of the plank $P$).

It is here that we make crucial use of the modified maximal density factoring lemma. By \eqref{factmaxmodbias}, we have
\[
\frac{\Delta(\TT_{B,W}, W)}{\Delta(\TT_{B,W}, T_{\theta})} \gtrapprox \left( \frac{|W|}{|T_{\theta}|} \right)^{\exfact}.
\]
Since $W$ has dimensions $ a\times b \times r_1$  and $T_{\theta}$ has dimensions $\delta \times b' \times r_1$ with $b' \le b$,  we have

\[
\left( \frac{|W|}{|T_{\theta}|} \right)^{\exfact} \ge \left( \frac{a}{\delta} \right)^{\exfact}.
\]
and hence
\[
\frac{\Delta(\PP, P_\theta)}{\Delta(\PP,B_1)} = \Delta(\TT_{B,W}, T_{\theta})/\Delta(\TT_{B,W}, W)\lessapprox\left( \frac{\delta}{a} \right)^{\exfact}.
\]

Now we are ready to estimate $\PP[P_\theta]$.   Note that $| \PP[P_\theta] | = | \TT_{B,W}[T_\theta] |$ and $| \PP |  = |\TT_{B,W}|$ and so

\[ \frac{ | \PP[P_\theta] |}{|\PP|} = \frac{ | \TT_{B,W}[T_\theta] |}{| \TT_{B,W} |} = \frac{ \Delta(\TT_{B,W}, T_\theta)}{\Delta(\TT_{B,W}, W)} \cdot \frac{ |T_\theta|} {|W|} \lessapprox \left( \frac{ |T_\theta|} {|W|} \right)^{1 + \exfact} = |P_\theta|^{1 + \exfact} = \left( \theta \frac{\delta^2}{a^2} \right)^{1 + \exfact} \le \theta \left( \frac{\delta}{a} \right)^{2 + \exfact}. \]

Recalling that $P_\theta$ has dimensions $\theta \frac{\delta}{a} \times \frac{\delta}{a} \times 1$,  we see that $ \frac{ | \PP[P_\theta] |}{|\PP|} \le \theta (\delta/a)^{2 + \exfact}$,  and rearranging gives

\[
|\PP[P_\theta]|/\theta \lessapprox (\delta/a)^{2+\exfact}|\PP|.
\]

Let us verify that $(\delta/a)^{2+\exfact}|\PP|\geq 1$, i.e. $|\PP|\geq (a/\delta)^{2+\exfact}$. This follows from \eqref{TBWbig}, which says that 
\[
| \PP|=|\TT_{B, W}| \gtrapprox \left( \frac{ |W| }{|T_B|} \right)^{1+\exfact}=\left( \frac{ a\times b\times r_1 }{\delta\times\delta\times r_1} \right)^{1+\exfact}\geq (a/\delta)^{2+2\exfact}.
\]
Thus we can apply Lemma \ref{plankF} with $M = (\delta/a)^{2+\exfact}|\PP|\geq 1$. We conclude that
\[
|U(\PP,Y_{\PP})|\gtrapprox \delta^{\eps}\Big(\frac{\delta}{a}\Big)^{2\beta}\Big(M^{-1} \Big(\frac{\delta}{a} \Big)^2 |\PP|\Big)^{\beta/2} = \delta^{\eps}\Big(\frac{\delta}{a}\Big)^{2\beta-\exfact\beta/2}.
\]
Since $a\geq\delta^{1-\tau}$, if we select $\eps < \exfact\beta\tau/4$, then this yields \eqref{denseInsideW} with $\nu = \exfact\beta\tau/8$. This finishes the proof in the thick case. 

\begin{remark}
Notice that in the thick case,  the exponent $\nu$ obeys 
\[
\nu \sim \tau \exfact \beta.
\]

This value of $\nu$ does not directly depend on $\zeta$.   But in the thin case,  we will have to choose $\exfact$ and $\tau$ much smaller than $\zeta$, so indirectly $\nu$ depends on $\zeta$.

\end{remark}

\subsection{The thin case}
Now we consider the thin case when $a \in [ \delta, \delta^{1-\tau}]$.  Fix a ball $B$. Recall that each set $W\in\WW_B$  has dimensions $a \times b \times r_1$,  with $r_1 = \delta^{\exscalb}$. We also have that each tube segment $T_B\in\TT_B$ has roughly the same density of shading, i.e. $|Y_B(T_B)| \sim \lambda(\TT_B, Y_B)|T_B|\gtrapprox\delta^{\eta}|T_B|$ for each $T_B\in\TT_B$. After a harmless refinement we can suppose that $Y_B(T_B)$ is covered by a union of $\delta$-balls, each of which satisfies 
\begin{equation}\label{typicalMultiplicityInADeltaBall}
|B_\delta\cap Y_B(T_B)|\gtrapprox\delta^{\eta}|B_\delta|. 
\end{equation}

Apply Proposition \ref{factoringAndMultPropCombined} to $\TT_B=\bigsqcup_{\WW_B}\TT_{B,W}$. We obtain a shading $Y'_B(T_B)\subset Y_B(T_B)$ on $\TT_B$ so that $(\TT_B, Y'_B)$ is an $\approx 1$ refinement of $(\TT_B, Y_B)$; a subset $\WW_B'$ of $\WW_B$; and a shading $Y_{\WW_B'}$ on $\WW_B$ so that

\begin{equation} \label{densityWWB}
\lambda(\WW_B', Y_{\WW_B'})\gtrapprox \delta^{2\eta},
\end{equation}
and
\begin{equation} \label{containmentWWB}
(\TT_B)_{Y'_B}(x)\subset \bigcup_{\substack{W\in \WW'_B \\ x\in Y_{\WW_B'}(W) }}(\TT_{B,W})_{Y'_B}(x).
\end{equation}

Recall from Remark \ref{compatibilityOfShadingOnWWwithVV},  that $Y_{\WW'_B}(W)$ is the $a$-neighborhood of $U(\TT_{B,W}, Y'_B)$.   Combining this with  \eqref{typicalMultiplicityInADeltaBall} we have that
\begin{equation}\label{unionOfTubesfillsOutMostOfUnionOfW}
|U(\TT_B, Y_B')| \gtrsim  (\delta/a)^3\delta^\eta |U(\WW_B', Y_{\WW_B'})|\gtrsim \delta^{3\tau + \eta} |U(\WW_B', Y_{\WW_B'})|.
\end{equation}

After a further harmless refinement, we may suppose that $\sum_{\TT_{B,W}}|Q\cap Y_B'(T_B)|$ is roughly constant for each grid-aligned $a/100$ cube $Q$ that intersects $U(\TT_{B,W}, Y')$, as $W$ ranges over the elements of $\WW_B'$, and that each shading $Y_{\WW_B'}(W)$ is the union of such $a/100$ cubes.

\subsubsection{Slab case}
Recall that $r_1 = \delta^{\exscalb}$, where $\exscalb > 0$ is the constant from Lemma \ref{lemmain2vns}, which is a tiny constant that we get to choose.   We define the slab case to be the case $b \ge \delta^{\exscalb} r_1 = \delta^{2 \exscalb}$.  In this case, 
 $W$ is  a giant slab of dimensions at least $a \times \delta^{2 \exscalb} \times \delta^{ \exscalb}$.   
 
Applying Lemma~\ref{slab1}  to  $(\WW_B', Y_{\WW_B'})$, we have $\mu(\WW_B', Y_{\WW_B'})\lessapprox \delta^{-8\eta-8\exscalb},$ i.e. 
\[
|U(\WW_B', Y_{\WW_B'})|\gtrapprox \delta^{8\eta+8\exscalb}|\WW_B'|\ |W|.
\]
Combining this with \eqref{unionOfTubesfillsOutMostOfUnionOfW}, we have
\[
|U(\TT_B, Y_B')|\gtrapprox  \delta^{8\eta+8\exscalb+3\tau}\delta^\eta|\WW_B'|\ |W|,
\]
where the final inequality used the fact that we are in the thin case, and thus $\delta/a\geq \delta^{\tau}.$ 

Next we estimate $\Delta_{max}(\TT_B)$.   Each tube $T_B \in \TT_B$ is contained in a tube $T \in \TT$.   If $T_B \subset K$,  then the corresponding tube $T$ is contained in $r_1^{-1} K$,  the dilation of $K$ by a factor $r_1^{-1}$.   Therefore $\Delta_{max}(\TT_B) \le r_1^{-2} \Delta_{max}(\TT)$.   Using this, we can see that

\[
|\WW_B'|\ |W| \gtrapprox \Delta_{\max}(\TT_B)^{-1}|\TT_B|\ |T_B|\geq \delta^\eta r_1^2 |\TT_B|\ |T_B|.
\]
Summing over the $r_1$-balls $B$, we conclude that
\[
|U(\TT, Y)|\gtrapprox \delta^{8\eta+8\exscalb+3\tau}\delta^{2\eta} r_1^2\delta^{O(\exscalb)}  |\TT|\ |T| \gtrsim \delta^{10\eta+O(\exscalb)+3\tau}  |\TT|\ |T|.
\]
We choose $\tau, \eta,$ and $\exscalb \ll \beta$,  and so this estimate implies our goal.   

This finishes the case $b \ge \delta^{\exscalb} r_1$,  and so from now on,  we can assume that $b \le \delta^{\exscalb} r_1$.

\subsubsection{Transverse case}  Fix a ball $B$. We will study the typical angle of intersection of the slabs of $\WW'_B$ by applying Lemma \ref{findingTypicalAngleOfIntersection} to $(\WW'_B, Y_{\WW'_B})$.   This gives an $\approx 1$ refinement, $(\WW''_B, Y_{\WW''_B})$  and a typical angle of intersection $\theta \in [a/b, 1]$. 

  Let $\tau'$ be a small parameter to be determined, with $\zeta \gg \tau' > \tau$.    We define the transverse case to be the case that $\theta \ge \delta^{- \tau'} a/b$.  

When $\theta$ is the typical angle of intersection, Lemma \ref{lemmaredplanktube} tells us that $U(\WW''_B, Y_{\WW''_B})$ morally almost fills each ball of radius $\theta b$.   More precisely,  
Lemma \ref{lemmaredplanktube} says that $(\WW''_B, Y_{\WW''_B})$ has a $\gtrapprox 1$ refinement $(\WW'''_B , Y_{\WW'''_B})$ so that for every ball $B_{\theta b}$ that intersects $U(\WW'''_B, Y_{\WW'''_B})$, we have

\[
|U(\WW'_B, Y_{\WW'_B}) \cap B_{\theta b}| \ge |U(\WW'''_B, Y_{\WW'''_B}) \cap B_{\theta b} | \gtrapprox a^{4 \eta}  |B_{\theta b}| \ge \delta^{4 \eta} |B_{\theta b}|.
\]

(Here $\theta$ is the typical angle in the paragraph above -- see Remark \ref{typicalAngleAlreadyPresent} .) 

Recall that $Y_{\WW'_B}$ consists of a union of $a$-balls that intersect $U(\TT, Y)$.  Therefore,  using \eqref{typicalMultiplicityInADeltaBall},
for any ball $B_r$ of radius $r \ge a$, we have 

\[ |U(\TT, Y) \cap B_r | \gtrapprox \delta^\eta (\delta/a)^3 |U(\WW'_B, Y_{\WW'_B}) \cap B_r| . \]

Since we are in the thin case, we have $a \le \delta^{1 - \tau}$. So for any ball $B_r$ of radius at least $a$, we have

\[ |U(\TT, Y) \cap B_r| \gtrapprox \delta^{3 \tau+\eta} |U(\WW'_B, Y_{\WW'_B}) \cap B_r|. \]

In particular, for the ball $B_{\theta b}$ above, we have

\begin{equation} \label{transfillball} 
|U(\TT, Y) \cap B_{\theta b} | \gtrapprox \delta^{3 \tau + 5 \eta} |B_{\theta b}|.
\end{equation}

Informally,  this equation shows that $U(\TT,Y)$ almost fills $B_{\theta b}$.    Formally,  we choose a radius $r$ of the form $r = \delta^{\eta j}$ with $\theta b \le r \le \delta^{-\eta} \theta b$.   Since $\theta b \ge \delta^{1 - \tau'}$,  we have

\begin{equation} \label{transfillball2} 
\frac{ |U(\TT, Y) \cap B_{r} | }{|B_r|}  \gtrapprox \delta^{3 \tau + 8 \eta} \ge \delta^{-2 \tau' \beta + 3 \tau + 8 \eta} \left( \frac{\delta}{r} \right)^{2 \beta}. 
\end{equation}

\noindent Since we choose $\tau' \gg \tau, \eta$,  Equation \eqref{transfillball2} gives \eqref{eqgoaldens} with a value of $\nu = \tau' \beta - O(\tau + \eta)$.    Equation \eqref{eqgoaldens} implies our conclusion,  as explained in Subsection \ref{subsecproofoverview}.

This finishes the case $\theta \ge \delta^{- \tau'} a/b$,  and so from now on,  we can assume that $\theta \leq \delta^{-\tau'}  a/b$.

\subsubsection{Tangential case}

Recall from the start of the transverse case that we fixed a ball $B$ and studied the typical angle of intersection of the slabs of $\WW'_B$ by applying Lemma \ref{findingTypicalAngleOfIntersection} to $(\WW'_B, Y_{\WW'_B})$.   This gave a constant multiplicity $\approx 1$ refinement $(\WW''_B, Y_{\WW''_B})$  with typical angle of intersection $\theta \in [a/b,1]$.   Let $\tau'$ be a small parameter to be determined, with $\zeta \gg \tau' > \tau$.    We defined the transverse case to be the case that $\theta \ge \delta^{- \tau'} a/b$.     Now we consider the tangential case,  when $\theta < \delta^{- \tau'} a/b$.  

Recall from \eqref{containmentWWB} that 

\[
(\TT_B)_{Y'_B}(x)\subset \bigcup_{\substack{W\in \WW'_B \\ x\in Y_{\WW_B'}(W) }}(\TT_{B,W})_{Y'_B}(x), 
\]

Next we refine $(\TT_B, Y_B')$ to make it compatible with $(\WW''_B,  Y_{\WW''_B})$.  Suppose that $T_B \in \TT_{B,W}$.   If $W \in \WW''_B$,  we set $Y''_B(T_B) = Y'_B(T_B)\cap Y_{\WW''_B}(W)$,  and if $W \notin \WW''_B$,  we set $Y''_B(T_B) $ to be empty. We have that $(\TT_B, Y_B'')$ is a $\gtrapprox 1$ refinement of $(\TT_B, Y_B')$, and 

\[
(\TT_B)_{Y''_B}(x)\subset \bigcup_{\substack{W\in \WW''_B \\ x\in Y_{\WW''_B}(W) }}(\TT_{B,W})_{Y''_B}(x).
\]

\noindent We claim that $(\TT_B, Y_B'')$ is a $\gtrapprox 1$ refinement of $(\TT_B, Y_B')$ because of the uniformity on axis-aligned $a/100$ cubes that we established earlier. 

Next we refine $(\TT, Y)$ to make it compatible with $(\TT_B,  Y''_B)$.   Recall that $|\TT(T_B)_{Y}(x)|$ is roughly constant among $x \in Y_B(T_B)$.   We know that $(\TT_B,  Y''_B)$ is a $\gtrapprox 1$ refinement of $(\TT_B,  Y_B)$.    If $T \in \TT(T_B)$,  we define $Y'(T) = Y(T) \cap Y''_B(T_B)$.   Because of the uniformity above,  $(\TT, Y')$ is a $\gtrapprox 1$ refinement of $(\TT,Y)$.  From the definition of $Y'$ we see that

\[
\TT_{Y'}(x) \subset \bigcup_{\substack{T_B\in\TT_{B}\\ x\in Y''_B(T_B)}} \TT(T_B).
\]

and so 

\[
\TT_{Y'}(x) \subset \bigcup_{\substack{W\in \WW''_B \\ x\in Y_{\WW''_B}(W) }}\bigcup_{\substack{T_B\in\TT_{B,W}\\ x\in Y''_B(T_B)}} \TT(T_B).
\]

The angle between any two tubes in $\TT_{B, W}$ is $\lessapprox b / r_1$.   Define $\rho_2 = b / r_1$. This means that for each $W$, there exists a unit vector $v(W)$ so that each $T\in\TT_{B,W}$ makes angle at most $\rho_2$ with $v(W)$. Thus for each $x\in U(\TT, Y')$ we have

\begin{equation}\label{MuTContainedInUnion}
\TT_{Y'}(x) \subset \bigcup_{\substack{W\in \WW''_B \\ x\in Y_{\WW''_B}(W) }}\{T\in\TT_Y(x)\colon \angle(T, v(W))\leq \rho_2\}.
\end{equation}

Now the cardinality of the inner set is $\lessapprox \mu(\rho_2) \approx \mu(\TT[T_{\rho_2}], Y)$ for any $T_{\rho_2} \in \TT_{\rho_2}$.  Since $(\WW''_B,  Y_{\WW''_B})$ is a refinement of $(\WW'_B,  Y_{\WW'_B})$ and the former has constant multiplicity $\mu(\WW''_B, Y_{\WW''_B})$,  we get that for any $T_{\rho_2} \in \TT_{\rho_2}$,  

\begin{equation}\label{criticalBoundOnMuTT}
\mu(\TT, Y) \leq \mu(\WW_B'', Y_{\WW_B''}) \mu(\TT[T_{\rho_2}], Y).
\end{equation}

Using $K_{KT}(\beta)$ with $\exfact$ in place of $\eps$, we have
\begin{equation}\label{muTTRhoSuingKTbeta}
\mu(\TT[T_{\rho_2}], Y)\lesssim \delta^{-\exfact}\Big((\delta/\rho_2)^{-2}|\TT[T_{\rho_2}]|\Big)^{\beta/2}.
\end{equation}

Next we consider $| \TT_{\rho_2} |$.   Since we are not in the slab case,  $b \le \delta^{\exscalb} r_1$,  and so $\rho_2 = b / r_1 \le \delta^{\exscalb}$.   On the other hand,  $b \ge \delta$ and so $\rho_2 = b / r_1 \ge \delta / r_1 = \delta^{1 - \exscalb}$.   So $\rho_2 \in [\delta^{1 - \exscalb},  \delta^{\exscalb}]$,  and so by the hypothesis of Lemma \ref{lemmain2vns},  we have

\[ | \TT_{\rho_2} | \ge \rho_2^{-2 - \zeta}. \]

Our next goal is to estimate $\mu(\WW''_B, Y_{\WW''_B})$.   Recall that $\theta$ is a typical angle of intersection for $(\WW''_B, Y_{\WW''_B})$.  
 In particular,  if $W_1, W_2 \in \WW''_B$ and if $x \in Y_{\WW''_B}(W_1) \cap Y_{\WW''_B}(W_2)$, then
 
\begin{equation} \label{anglebound} \angle(TW_1, T W_2 ) \lesssim \theta  \le \delta^{-\tau'} a/b.
\end{equation} 
 
Let $S$ denote a slab of dimensions $(a/b) r_1 \times r_1 \times r_1$, and let $\SSS$ denote a maximal set of essentially distinct slabs $S \subset B$.   Recall from \eqref{defPS} that we defined

\[ \WW''_S = \{ W \in \WW''_B: W \subset S \textrm{ and } \angle(TW, TS) \le \theta \}, \]

and that

\[ \WW''_B = \sqcup_{S \in \SSS} \WW''_S. \]

By \eqref{anglebound},  $\WW''_{Y_{\WW''_B}}(x)$ is contained in at most $\delta^{- O(\tau')}$ sets $\WW''_S$.  Therefore, 

\[ \mu(\WW''_B, Y_{\WW''_B}) \lessapprox \delta^{-O(\tau')} \max_{S \in \SSS} \mu(\WW''_S, Y_{\WW''_S}). \]

However,  we would like to only include $S$ with $\lambda(\WW''_S,  Y_{\WW''_S})$ not too small.   Recall from \eqref{densityWWB}
that $\lambda(\WW''_B,  Y_{\WW''_B}) \gtrapprox \lambda(\WW'_B, Y_{\WW'_B}) \gtrapprox \delta^{2 \eta}$.   Define 

\[ \SSS_{dens} = \{ S \in \SSS: \lambda(\WW''_S,  Y_{\WW''_S}) \ge (1/100) \lambda(\WW''_B, Y_{\WW''_B})\}. \]

\noindent Then we can define $\WW'''_B = \sqcup_{S \in \SSS_{dens}} \WW''_S$ and $Y_{\WW'''_B} = Y_{\WW''_B}$,  and we see that $(\WW'''_B,  Y_{\WW'''_B})$ is a $\gtrapprox 1$ refinement of $(\WW''_B,  Y_{\WW''_B})$.   Therefore

\[ \mu(\WW''_B,  Y_{\WW''_B}) \lessapprox \mu(\WW'''_B,  Y_{\WW'''_B})\lessapprox \delta^{-O(\tau')}  \max_{S \in \SSS_{dens}} \mu(\WW''_S, Y_{\WW''_S}). \]

Select $S \in \SSS_{dens}$.   We know that $\lambda(\WW''_S,  Y_{\WW''_S}) \gtrapprox \delta^{2 \eta}$.

We can make a linear change of variables that converts $S$ to $B_1$ and converts $\WW''_S$ to a set $\tilde \TT$ of $\rho_2$-tubes in $B_1$.    By the maximal density factoring lemma, we know that $\Delta_{max}(\WW''_S) \le \Delta_{max}(\WW_B)$ is small.  With the original lemma, we would have $\Delta_{max}(\WW_B) \lessapprox 1$.  With our modified lemma, Lemma \ref{lemmafactmaxbias}, we have $\Delta_{max}(\WW_B) \lessapprox \delta^{-O(\exfact)}$.   This implies $\Delta_{max}(\WW''_S)\lessapprox \delta^{-O(\exfact)}$,  and so 

\[ \Delta_{max}(\tilde\TT)\lessapprox \delta^{-O(\exfact)}. \]   

We also have 

\[ \lambda(\tilde \TT,  Y_{\tilde \TT}) = \lambda(\WW''_S,  Y_{\WW''_S}) \gtrapprox \delta^{2 \eta} . \]

 Since $\tilde \TT$ is a set of $\rho_2$-tubes in $B_1$ with $\Delta_{max}(\tilde \TT) \lessapprox \delta^{-O(\exfact)}$, we have $| \tilde \TT | \lessapprox \delta^{-O(\exfact)} \rho_2^{-2}$.    If 
$\eta>0$ is sufficiently small (recalling Remark \ref{multBoundsDeltaVsRho}) we can bound $\mu(\tilde \TT, Y_{\tilde \TT})$ using Lemma \ref{genKKT} with $\exfact$ in place of $\epsilon$,  which gives

\begin{equation}\label{muWWBounda}
\mu(\tilde \TT,  Y_{\tilde \TT})  \lessapprox \delta^{-\exfact}\Delta_{max}(\tilde\TT)^{1-\beta/2}   \big(\rho_2^{-2}|\tilde \TT|\big)^{\beta/2} \lessapprox  \delta^{-O(\exfact)}  \rho_2^{-2\beta}  \leq  \delta^{-O(\exfact)} \rho_2^{\beta \zeta/2}  \big(\rho_2^{-2}| \TT_{\rho_2} |\big)^{\beta/2}.
\end{equation}

Therefore, we get

\begin{equation}\label{muWWBound}
\mu(\WW''_B, Y_{\WW''_B}) \lessapprox \delta^{-O(\tau')} \mu(\tilde \TT, Y_{\tilde \TT})  \leq  \delta^{-O(\exfact + \tau')} \rho_2^{\beta \zeta/2}  \big(\rho_2^{-2}| \TT_{\rho_2} |\big)^{\beta/2}.
\end{equation}

Combining \eqref{criticalBoundOnMuTT}, \eqref{muTTRhoSuingKTbeta}, and \eqref{muWWBound}, we conclude that
\[
\mu(\TT, Y)\lessapprox \delta^{-O(\exfact + \tau')}\rho_2^{\beta \zeta/2} \big(\delta^{-2}|\TT|\big)^{\beta/2}\leq \delta^{-O(\exfact + \tau')+\exscalb\beta \zeta/2} \big(\delta^{-2}|\TT|\big)^{\beta/2},
\]
where for the final inequality we used the fact that we are not in the slab case, and thus $\rho_2= b/r_1< \delta^{\exscalb}$. We choose $\tau,  \tau', \exfact \ll \exscalb \beta \zeta$,  and this gives the desired bound with a gain $\nu = (1/4) \exscalb \beta \zeta$.  

\appendix
\section{Appendix: Probability lemmas} \label{appproblemmas}

In this appendix,  we give probability arguments to prove Lemma \ref{randCF} and Lemma \ref{lemrandommotion}.   The proofs are based on a simple probability lemma which is a small variation on Chernoff's inequality.

\begin{lemma} \label{lemchernoffvar} Suppose that $X_1,  ..., X_N$ are i.i.d.  random variables.   Suppose that $0 \le X_1 \le M$ almost surely and that $m = \EXP(X_1)$.   If $N m \ge M$,  then for every $S$, 

$$ \PP[ X_1 + ... + X_N > S ] \le e^{10} e^{- \frac{S}{Nm} }. $$

If $N m < M$,  then for every $S$, 

$$ \PP[ X_1 + ... + X_N > S ] \le e^{10} e^{- \frac{S}{M} }. $$

\end{lemma}

\begin{proof} Let $\lambda \in \RR$ be a parameter to choose later.   For each $\lambda \in \RR$,  using the Markov inequality and independence,  we see that 

\[ \PP[ X_1 + ... + X_N > S ] \le e^{- \lambda S} \EXP[e^{(X_1 + ... X_N) \lambda}] = e^{-\lambda S} \EXP ( \prod_{n=1}^N e^{\lambda X_n} ) = e^{- \lambda S}  \EXP(e^{\lambda X_1})^N. \]

For each $\lambda \in \RR$,  the function $x \mapsto e^{\lambda x}$ is convex,  and so given our distributional information about $X_1$ we can bound

\[ \EXP(e^{\lambda X_1}) \le \frac{m}{M} e^{\lambda M} + 1 - \frac{m}{M}. \]

Plugging in,  we get

\[ \PP[ X_1 + ... + X_N > S ] \le  e^{- \lambda S}  \left(\frac{m}{M} e^{\lambda M} + 1 - \frac{m}{M} \right)^N. \]

We first consider the regime $N m \ge M$.   In this regime, we choose $\lambda = \frac{1}{N m}$.  Since $N m \ge M$,  we have $0 \le \frac{M}{Nm} \le 1$,  and so $e^{\lambda M} = e^{\frac{M}{Nm}} \le 1 + 10 \frac{M}{Nm}$.   Plugging in and simplifying,  we get

\[ \PP[ X_1 + ... + X_N > S ] \le  e^{- \frac{S}{Nm}}  \left(1 + \frac{10}{N} \right)^N \le e^{10} e^{ -\frac{S}{Nm}}. \]

This gives the conclusion in the regime $N m \ge M$.  On the other hand,  if $N m < M$,  we define $N'$ so that $N' m = M$.   Since $X_n \ge 0$ almost surely,  we can bound 

\[ \PP[ X_1 + ... + X_N > S ] \le  \PP[ X_1 + ... + X_{N'} > S ] \le e^{10} e^{- \frac{S}{N'm}} = e^{10} e^{ - \frac{S}{M} }. \qedhere
\]

\end{proof}

Now we consider random translations and rigid motions of $\RR^n$.    If $v \in \RR^n$,  we let $R_v (x) = x+v$.  
A random translation of size $\le \rho$ is a translation $R_v$ by a vector $v \in B_\rho$ selected randomly from the uniform measure on $B_\rho$.   A random rigid motion is a random rotation composed with a random translation of size $\le 1$.

Now we restate and prove Lemma \ref{randCF}:
\begin{lemmaRandCF} 
Suppose that $\TT$ is a set of (essentially distinct) $\delta$-tubes in $B_1 \subset \RR^n$.   Let $R_1, ..., R_J$ be random rigid motions of size 1,  with $J = C_F(\TT)$,  and let $\TT' = \bigcup_{j=1}^J R_j(\TT)$.   Then with high probability the tubes of $\TT'$ are essentially distinct up to multiplicity $\lessapprox 1$ and also $C_F(\TT') \lessapprox 1$.  
After refining $\TT'$ to make the tubes essentially distinct,  we have $|\TT'| \approx C_F(\TT) |\TT|$ and $C_F(\TT') \lessapprox 1$.
\end{lemmaRandCF}

\begin{proof} By assumption,  the tubes of $\TT$ are essentially distinct.  It could happen that two different tubes of $\TT'$ are not essentially distinct.  However,  we will show that with high probability,  this effect is very mild.   We will show that with high probability,  for any $\delta$-tube $T_0$,  there are $\lessapprox 1$ tubes of $\TT'$ contained in $100 T_0$.    

We define $X_j =|  R_j(\TT) [100 T_0] |$,  which we think of as a random variable.    Note that $|\TT' [100 T_0] | = \sum_{j=1}^J X_j$. 

Since the tubes of $\TT$ are essentially distinct,  we have $X_j = |R_j(\TT) [100 T_0] | = | \TT [ R_j^{-1}(100 T_0)] | \lesssim 1$.   We choose $M \lesssim 1$ and note that $0 \le X_j \le M$ almost surely.  

Next we want to estimate $\EXP(X_j)$.   By linearity of expectation,  $\EXP(X_j) = \sum_{T \in \TT} \PP \left[ R_j(T) \subset 100 T_0 \right]$.  
Since $R_j$ is a random rigid motion,  we note that for each $T \in \TT$,  

\begin{equation} \label{randrigT0} \PP \left[ R_j(T) \subset 100 T_0 \right] \lesssim |T_\delta|^2. \end{equation}

\noindent   Therefore,   

\[ m = \EXP(X_j) \lesssim |\TT| |T_\delta|^2. \]

We claim that $J m \lesssim 1 = M$.   To see this,  we note that $J m \lesssim C_F(\TT) |\TT| |T_\delta|^2$.   Now suppose that $\Delta_{max}(\TT) = \Delta(\TT, K)$ for some convex set $K \subset B_1$.   Then we have $C_F(\TT) = \frac{ \Delta(\TT, K)}{\Delta(\TT,  B_1) } = \frac{ |\TT[K]| }{|\TT| |K|}$.   But since the tubes of $\TT$ are essentially distinct,  we have $|\TT[K]| \lesssim (|K| / |T_\delta|)^2$.   Therefore we have

\[ J m \lesssim C_F(\TT) |\TT| |T_\delta|^2 = \frac{|\TT[K]| |T_\delta|^2}{|K|} \lesssim |K| \lesssim 1. \]

Therefore,  if $S \ge 1$,  then Lemma \ref{lemchernoffvar} (or the standard Chernoff bound) gives 
\[ \PP [ X_1 + ... + X_J > S ] \le e^{10} e^{- S / M} \sim e^{-S}. \]
In other words,  $\PP [ | \TT'[100 T_0] | > S ] \lesssim e^{-S}$.   The number of essentially distinct $T_0$ is at most  $\delta^{-O(1)}$,  and so we can choose $S \lessapprox 1$ so that with high probability $| \TT' [100 T_0] | \le S$ for every $T_0$.  

If we refine $\TT'$ to be a set of essentially distinct tubes then we still have $|\TT'| \approx C_F(\TT) |\TT|$.   

Next we estimate $C_F(\TT')$.   Refining $\TT'$ to make the tubes essentially distinct will affect $C_F(\TT')$ by a factor $\lessapprox 1$,  so we can ignore the refinement.

Let $K \subset B_2 \subset \RR^n$ be a convex set of dimensions $k_1\times \ldots\times k_{n-1}\times 1$, with $\delta\leq k_1\leq\ldots\leq k_{n-1}\leq 1$. Let $X_j = | R_j(\TT) [K ] |$,  which we think of as a random variable.   Note that $|\TT'[K]| =  \sum_{j=1}^J X_j$.   We will show that (for each $K$),  

\begin{equation} \label{rareeventa} \PP \left[ \sum_{j=1}^J X_j > C A  \Delta(\TT',  B_1) \frac{|K|}{|T_\delta|} \right] \lesssim e^{-A},  \end{equation}

where $C\geq 1$ is an absolute constant. Since there are only $(1/\delta)^{O(1)}$ essentially different choices for the convex set $K \subset B_1$, we have that with probability at least $1-(1/\delta)^{O(1)}e^{-A}$, for  every $K$ we have   $| \TT'[K] | \lesssim A \Delta(\TT', B_1) \frac{|K|}{|T_\delta|}$ and so $C_F(\TT') \lesssim A$.    We can choose $A \lessapprox 1$ so that the probability is at least $99/100$.

We will check \eqref{rareeventa} using Lemma \ref{lemchernoffvar}.   Note that $\Delta_{max}(\TT) \sim  C_F(\TT) \Delta(\TT,  B_1) \sim  \Delta(\TT',  B_1) \sim |\TT'| |T_\delta|$. Define

\[ M = |\TT'| |K| = \frac{ \Delta(\TT', B_1) |K|}{ |T_\delta|} = \frac{ \Delta_{max}(\TT) |K|} { |T_\delta|} \ge  | \TT [ R_j^{-1}(K)] | = X_j. \] 

Clearly $X_j \ge 0$,  so $0 \le X_j \le M$ almost surely.     For each $T \in \TT$,  we claim that 

\begin{equation} \label{probRTinKa}  \PP[R_j(T) \subset K ] \lesssim |K| .  \end{equation}

\noindent To see this,  pick a point $x \in T$ (for example $x$ could be the center of $T$) and bound 

\[ \PP(R_j(T) \subset K) \le \PP(R_j(x) \in K) \lesssim \frac{ |K| }{|B_1(x)|} \sim |K| . \]

By linearity of expectation,  \eqref{probRTinKa} gives the bound

\[\EXP(X_j) \lesssim |\TT| |K|  . \]
Define $m=\EXP(X_j)$. Since $J = C_F(\TT)$ we have $J m \lesssim C_F(\TT) |\TT| |K| = |\TT'| |K| = M$.    Therefore,  Lemma \ref{lemchernoffvar} tells us that for every $S$, 

\[ \PP[X_1 + ... + X_J > S ] \lesssim e^{-cS/M}.  \]

We plug in $S = C A  \Delta(\TT', B_1) \frac{|K|}{|T_\delta|} = C A M$, with $C$ sufficiently large, and we get \eqref{rareeventa}.   
\end{proof}

\noindent Finally,  we restate and prove Lemma \ref{lemrandommotion}.
\begin{lemmaLemrandommotion}
Let $\delta \le \rho \le 1$.   Suppose that $\TT$ is a set of essentially distinct $\delta$-tubes in $T_\rho\subset\RR^n$ with $\frac{|T_\rho|}{|\TT| |T_\delta|} \sim J \ge 1$.     Let $R_1,  ...,  R_J$ be a set of $J$ random translations of size $\le \rho$.   Let $\TT' = \bigcup_{j=1}^J R_j(\TT)$. Let $A\geq 1$.   Then with probability at least $1-(\rho/\delta)^{O(1)}e^{-A}$,  the tubes of $\TT'$ are essentially distinct up to a factor $\lesssim A$ and we have $\Delta_{max}(\TT') \lesssim A \Delta_{max}(\TT)$.  
After refining $\TT'$ to make the tubes essentially distinct,  we have with high probability $|\TT'| \approx J |\TT|$ and $\Delta_{max}(\TT') \lessapprox \Delta_{max}(\TT)$
\end{lemmaLemrandommotion}

The proof ideas are the same as for Lemma \ref{randCF} above.

\begin{proof} First we check that with high probability,  for any $\delta$-tube $T_0 \subset T_\rho$,  $|\TT'[100 T_0]| \lessapprox 1$.   This will show that after refining $\TT'$ to be essentially distinct,  we still have $| \TT' | \approx J |\TT|$.   

We define $X_j =|  R_j(\TT) [100 T_0] |$,  which we think of as a random variable.    Note that $|\TT' [100 T_0] | = \sum_{j=1}^J X_j$. 

Since the tubes of $\TT$ are essentially distinct,  we have $X_j = |R_j(\TT) [100 T_0] | = | \TT [ R_j^{-1}(100 T_0)] | \lesssim 1$.   We choose $M \lesssim 1$ and note that $0 \le X_j \le M$ almost surely.  

Next we want to estimate $\EXP(X_j)$.   By linearity of expectation,  $\EXP(X_j) = \sum_{T \in \TT} \PP \left[ R_j(T) \subset 100 T_0 \right]$.  
Since $R_j$ is a random translation of size $\rho$,   we note that for each $T \in \TT$,   

\begin{equation} \label{randtransT0} \PP \left[ R_j(T) \subset 100 T_0 \right] \lesssim \frac{ |T_\delta|}{|T_\rho|}.
\end{equation}

\noindent Therefore,   

\[ m = \EXP(X_j) \lesssim |\TT| |T_\delta| / |T_\rho|. \]

We claim that $J m \lesssim 1 = M$.   To see this,  we note that $J m \lesssim J |\TT| |T_\delta| / |T_\rho| \sim 1$ by our assumption about $J$.

Therefore,  if $S \ge 1$,  then Lemma \ref{lemchernoffvar} (or the standard Chernoff bound) gives 

\[ \PP [ X_1 + ... + X_J > S ] \le e^{10} e^{- S / M} \sim e^{-S}. \]

In other words,  $\PP [ | \TT'[100 T_0] | > S ] \lesssim e^{-S}$.   The number of essentially distinct $T_0$ is at most  $\delta^{-O(1)}$,  and so we can choose $S \lessapprox 1$ so that with high probability $| \TT' [100 T_0] | \le S$ for every $T_0$.  

Next we estimate $\Delta_{max}(\TT')$.  

Let $K \subset 2 T_\rho$ be a convex set of dimensions $k_1\times \ldots\times k_{n-1}\times 1$, with $\delta\leq k_1\leq\ldots\leq k_{n-1}\leq\rho$. Let $X_j = | R_j(\TT) [K ] |$,  which we think of as a random variable.   Note that $|\TT'[K]| =  \sum_{j=1}^J X_j$.   We will show that (for each $K$),  

\begin{equation} \label{rareevent} \PP \left[ \sum_{j=1}^J X_j > C A  \Delta_{max}(\TT) \frac{|K|}{|T_\delta|} \right] \lesssim e^{-A}.  \end{equation}

Since there are only $(\rho/\delta)^{O(1)}$ essentially different choices for the convex set $K \subset T_\rho$, we have that with probability at least $1-(\rho/\delta)^{O(1)}e^{-A}$, for  every $K$ we have   $| \TT'[K] | \lesssim A \Delta_{max}(\TT) \frac{|K|}{|T_\delta|}$ and so $\Delta_{max}(\TT') \lesssim A \Delta_{max}(\TT)$.  

We will check \eqref{rareevent} using Lemma \ref{lemchernoffvar}.   Define $M =  \Delta_{max}(\TT) |K| / |T_\delta|$. Note that 

\[ M  \ge  | \TT [ R_j^{-1}(K)] | = X_j. \] 

Clearly $X_j \ge 0$,  so $0 \le X_j \le M$ almost surely.     For each $T \in \TT$,  we claim that 

\begin{equation} \label{probRTinK}  \PP[R_j(T) \subset K ] \lesssim  |K| /|T_\rho|.  \end{equation}

\noindent To see this,  pick a point $x \in T$ (for example $x$ could be the center of $T$) and bound 

\[ \PP(R_j(T) \subset K) \le \PP(R_j(x) \in K) \le \frac{ |B_\rho(x) \cap K|}{|B_\rho(x)|} \le \frac{ k_1\cdots k_{n-1}\cdot  \rho}{\rho^n} \sim \frac{ |K| }{|T_\rho|} . \]

By linearity of expectation,  \eqref{probRTinK} gives the bound

\[\EXP(X_j) \lesssim \frac{ |\TT| |K|}{|T_\rho|} . \]
Define  $m = \EXP(X_j)$. By our definition of $J$ we have $J m \lesssim \frac{|K|}{|T_\delta|} \lesssim \Delta_{max}(\TT) \frac{|K|}{|T_\delta|} = M$.   Therefore,  Lemma \ref{lemchernoffvar} tells us that

\[ \PP[X_1 + ... + X_J > S ] \lesssim e^{-cS/M}.  \]

We plug in $S = C A  \Delta_{max}(\TT) \frac{|K|}{|T_\delta|} = C A M$, with $C$ sufficiently large, and we get \eqref{rareevent}.   \end{proof}

\section*{Acknowledgements}
Larry Guth is supported by a Simons Investigator Award. 
Hong Wang is supported by NSF CAREER DMS-2238818 and NSF DMS-2055544.  
Joshua Zahl is supported by Discovery and Alliance grants from the Natural Sciences and Engineering Research Council of Canada, by the Fundamental Research Funds for the Central Universities (Project No. 100-63253272), and by the Nankai Zhide Foundation.

\section*{AI Usage}
All results in this paper were proved by the authors, and the paper was written by the authors. AI tools (specifically OpenAI's Astra) were used to proof-read the paper.

\end{document}